\documentclass{article}
\usepackage[utf8]{inputenc}

\usepackage[a4paper,top=2.5cm,bottom=2.5cm,left=3cm,right=3cm,marginparwidth=2cm]{geometry}

\usepackage{upgreek}
\usepackage{amsmath,amssymb,amsthm,amsfonts}
\usepackage{mathtools}
\usepackage{dsfont}
\usepackage{graphicx}
\usepackage{subcaption}
\usepackage{float}
\usepackage{bbm}
\usepackage{bm}
\usepackage{enumitem}
\usepackage{pifont}
\usepackage{comment}

\usepackage[colorinlistoftodos]{todonotes}
\usepackage[allcolors=black]{hyperref}

\usepackage[square,numbers]{natbib}
\newcommand{\PP}{{\mathbb{P}}}

\newcommand{\EE}{{\mathbb{E}}}

\newcommand{\eps}{{\epsilon}}
\newcommand{\HH}{{\mathbb{H}}}

\newtheorem{theorem}{Theorem}[section]

\newtheorem{lemma}[theorem]{Lemma}

\newtheorem{proposition}[theorem]{Proposition}
\newtheorem{claim}[theorem]{Claim}
\theoremstyle{definition}

\theoremstyle{remark}
\newtheorem{remark}[theorem]{Remark}
\theoremstyle{remark}

\title{Massive SLE$_4$ and the scaling limit of the massive harmonic explorer}
\author{L\'eonie Papon \thanks{Durham University}}
\date{\today}

\begin{document}

\maketitle

\begin{abstract}
The massive harmonic explorer is a model of random discrete path on the hexagonal lattice that was proposed by Makarov and Smirnov as a massive perturbation of the harmonic explorer. They argued that the scaling limit of the massive harmonic explorer in a bounded domain is a massive version of chordal SLE$_4$, called massive SLE$_4$, which is conformally covariant and absolutely continuous with respect to chordal SLE$_4$. In this paper, we provide a full and rigorous proof of this statement. Moreover, we show that a massive SLE$_4$ curve can be coupled with a massive Gaussian free field as its level line, when the field has appropriate boundary conditions.
\end{abstract}

\section{Introduction}

Chordal SLE$_\kappa$ curves are a one-parameter family of planar curves, characterized by conformal invariance and a Markovian property, that were introduced by Schramm \cite{Schramm_SLE}. They have been shown to arise as the scaling limits of interfaces in many planar statistical mechanics models at criticality when the boundary conditions are chosen appropriately \cite{percolation, LERW, HE, DGFF, Ising}. From a conformal field theory perspective, this can be understood as a consequence of the conformal invariance of the limiting field that formally describes these models in the continuum and of the locality of the associated action functional. However, many interesting questions also arise when looking at massive perturbations of the models, obtained by sending some of their parameters to their critical values at an appropriate rate. In the scaling limit, these perturbations break the conformal invariance of some of the observables of these models. In \cite{off_SLE}, Makarov and Smirnov asked whether SLE-type curves could nevertheless describe the scaling limits of interfaces in such massively perturbed models. The idea is that, while these interfaces should only be conformally covariant, they should still enjoy a Markovian property similar to that of the interfaces at criticality. Makarov and Smirnov observed that these properties could be captured by requiring that the massive version of an SLE$_\kappa$ martingale observable becomes a martingale. This in turn should be realised by adding an appropriate drift to the driving function of an SLE$_\kappa$ curve. 

This approach has been particularly successful in the study of the massive loop-erased random walk. In \cite{mLERW}, Chelkak and Wan have shown that the scaling limit of massive loop-erased random walk on a subset of the square grid is a massive version of chordal SLE$_2$, called mSLE$_2$, for which the drift term in the driving function can be explicitly identified. This result was then extended to the radial case in connection to the height function associated with a near-critical dimer model \cite{off_dimers}. mSLE$_2$ is absolutely continuous with respect to SLE$_2$ and this absolute continuity also holds at the discrete level, which, to some extent, simplifies the proof of the convergence and the analysis in the continuum. When absolute continuity with respect to SLE$_{\kappa}$ is not conjectured to hold, the problem is more challenging. For instance, scaling limits of interfaces that could be described by massive versions of SLE$_\kappa$ seem to arise in near-critical percolation \cite{near_percolation} and in the massive Fortuin-Kasteleyn model \cite{massive_FK}. However, in these examples, the limiting interface is expected to be singular with respect to SLE$_\kappa$, see for example \cite{Nolin}, which makes the appropriate massive version of SLE$_\kappa$ harder to define or characterize.

\subsection{Main results}

Here, we are interested in another example of such interface: the massive harmonic explorer. This model is a massive version of the harmonic explorer studied by Schramm and Sheffield \cite{HE} that was proposed by Makarov and Smirnov in \cite{off_SLE}. To define this model, let us first recall the definition of the harmonic explorer. It is a random discrete path defined on the hexagonal lattice. To construct it, one considers a subset $\Omega_{\delta}$ of the triangular lattice $\delta \mathbb{T}$ with meshsize $\delta > 0$ together with two marked points $a_{\delta}$ and $b_{\delta}$ on the boundary of $\Omega_{\delta}$. The vertices on the clockwise oriented boundary arc $(a_{\delta}b_{\delta})$ are assigned the sign $+$ while the vertices on the counter-clockwise oriented boundary arc $(a_{\delta}b_{\delta})$ are assigned the sign $-$, $a_{\delta}$ and $b_{\delta}$ being assigned an arbitrary sign. The path starts in the middle of the edge joining $a_{\delta}$ to the vertex on the boundary of $\Omega_{\delta}$ with opposite sign. This singles out a vertex $v$ of $\Omega_{\delta}$ which is linked by an edge to $a_{\delta}$ and to this other boundary vertex. Let $h_{\delta}$ be the unique discrete harmonic function in $\Omega_{\delta}$ with boundary conditions $+1/2$ on the boundary vertices with sign $+$ and  $-1/2$ on the boundary vertices with sign $-$. Then, with probability $1/2+h_{\delta}(v)$, the path turns right, that is follows the edges of the hexagonal lattice linking its starting point to the middle of the edge of $\Omega_{\delta}$ on its left, and the vertex $v$ is assigned the sign $+$. With complementary probability, the path turns left and in this case, the vertex $v$ is assigned the sign $-$. In both cases, the vertex $v$ becomes a boundary vertex and this defines a new graph, with its associated discrete harmonic function $h_{\delta,1}$. One can then repeat the above procedure, with respect to the harmonic function corresponding to the new graph, to continue tracing the path. This gives rise to a sequence of discrete harmonic functions $(h_{\delta, n})_n$ corresponding to the sequence of graphs obtained while constructing the path. The procedure terminates when the path reaches the edge linking $b_{\delta}$ to a boundary vertex with opposite sign. See Figure \ref{fig_mHE} for a dual perspective on the hexagonal lattice.

To define a massive perturbation of this model, which we call the massive harmonic explorer as in \cite{off_SLE}, we assign a weight $1-cm^2\delta^2$ to each edge of the graph. Here, $m^2>0$ and $c>0$ is a constant depending on the lattice, but not on $\delta$. We let $h_{\delta}^{m}$ be the unique discrete massive harmonic function in $\Omega_{\delta}$ with boundary conditions $1/2$ on the boundary vertices with sign $+$ and $-1/2$ on the boundary vertices with sign $-$. The path is then constructed by following the same procedure as above, except that we now consider the function $h_{\delta}^m$ instead of $h_{\delta}$, thus obtaining a sequence of discrete massive harmonic functions $(h_{\delta,n}^m)_n$. See again Figure \ref{fig_mHE} for a dual perspective on the hexagonal lattice.

\begin{figure}
    \centering
    \begin{subfigure}[b]{0.45\textwidth}
    \centering
    \includegraphics[width=\textwidth]{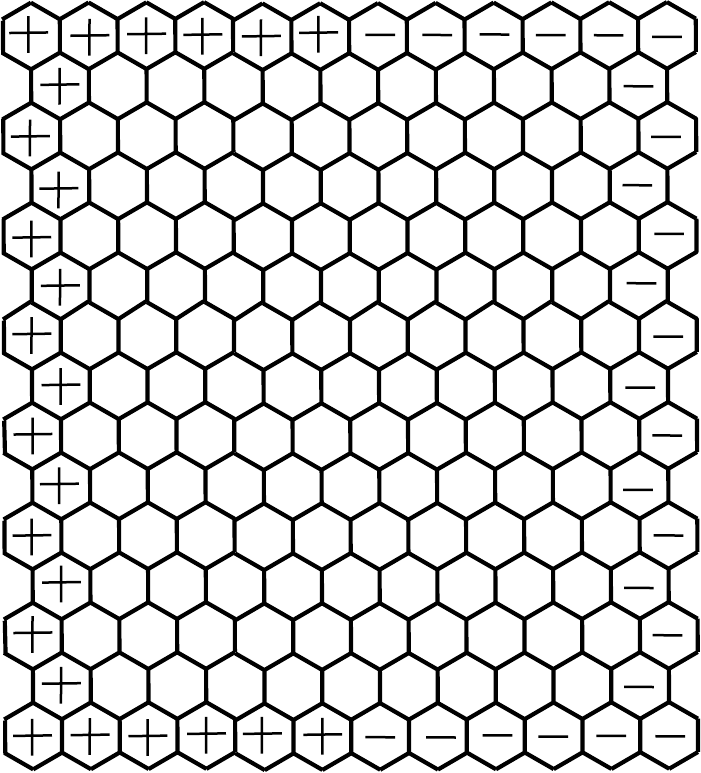}
    \caption{}
    \label{fig_1}
    \end{subfigure}
    \hfill
    \begin{subfigure}[b]{0.45\textwidth}
    \centering
    \includegraphics[width=\textwidth]{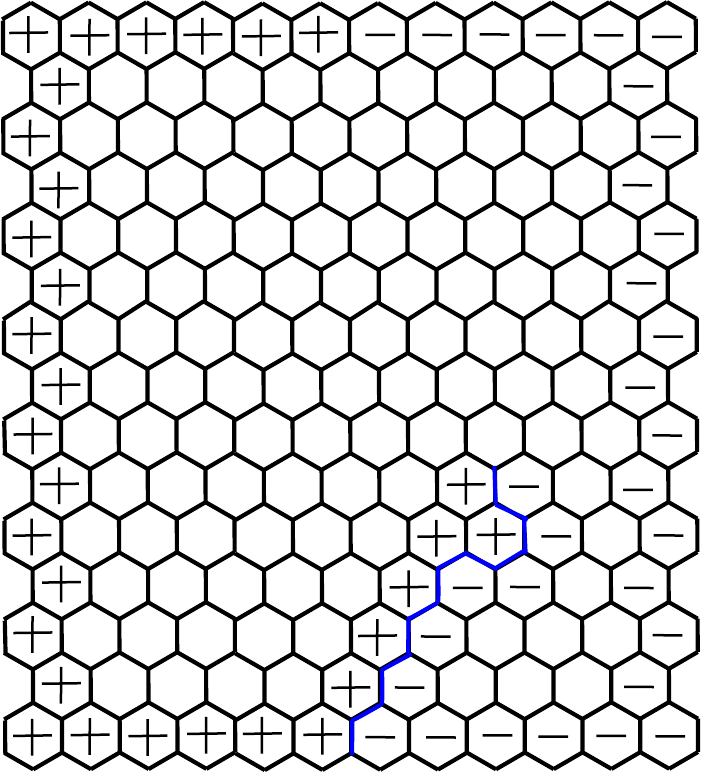}
    \caption{}
    \label{fig_2}
    \end{subfigure}
  \caption{(a) Initial configuration: the hexagons on the right-hand side have sign $-$ while those on the left hand-side have sign $+$. (b) The blue path is a possible path followed by the first steps of the massive harmonic explorer. The hexagons on the right, respectively left, of the path have been assigned the sign $-$, respectively $+$.}
  \label{fig_mHE}
\end{figure}

The harmonic explorer is known to converge to chordal SLE$_4$ in an appropriate topology \cite{HE}. Makarov and Smirnov provided arguments in \cite{off_SLE} to support their assertion that the massive harmonic explorer in turn converges to a massive version of chordal SLE$_4$, which is absolutely continuous with respect to SLE$_4$. This was further investigated in an unpublished manuscript \cite{Ch_mHE}; however, a conclusive argument was not reached. Our main result is a fully detailed and rigorous proof of the statement of Makarov and Smirnov, and can informally be written as follows.

\begin{theorem} \label{theorem_intro}
Let $\Omega \subset \mathbb{C}$ be a bounded, open and simply connected domain with two marked boundary points $a$ and $b$. Let $(\Omega_{\delta}, a_{\delta}, b_{\delta})_{\delta}$ be discrete approximations of $(\Omega, a, b)$, where for each $\delta > 0$, $\Omega_{\delta}$ is a subset of the triangular lattice $\delta \mathbb{T}$. Let $m>0$. Then, as $\delta \to 0$, the scaling limit of the massive harmonic explorer from $a_{\delta}$ to $b_{\delta}$ in $\Omega_{\delta}$ is a random curve $\gamma$ whose law is that of massive SLE$_4$ with mass $m$ in $\Omega$ from $a$ to $b$.
\end{theorem}

Let $\phi: \Omega \to \HH$ be a conformal map from $\Omega$ to the upper half-plane $\HH=\{z \in \mathbb{C}: \Im(z) > 0 \}$ such that $\phi(a)=0$ and $\phi(b)=\infty$. A curve $\gamma$ in $\Omega$ from $a$ to $b$ is said to have the law of massive SLE$_4$ with mass $m$ in $\Omega$ from $a$ to $b$ if $\phi(\gamma)$, when parametrized by half-plane capacity, is a chordal stochastic Loewner evolution whose driving function satisfies the SDE
\begin{equation} \label{SDE_intro}
    dW_t = 2dB_t - 2\pi \bigg(\int_{\Omega_t}m^2P_t^m(w)h_t(w)dw\bigg)dt, \quad W_0=0,
\end{equation}
where $(B_t, t \geq 0)$ is a standard one-dimensional Brownian motion. Above, $\Omega_t$ is defined as $\Omega \setminus K_t$, with $K_t$ being the hull generated by $\phi(\gamma)([0,t])$, $P_t^m$ is related to the massive Poisson kernel with mass $m$ in $\Omega_t = \Omega \setminus K_t$ evaluated at the growth point $\gamma(t)$ and $h_t$ is the unique harmonic function in $\Omega_t$ with boundary conditions $-1/2$ on the counterclockwise oriented boundary arc $(ab)$ and the right side of $\gamma([0,t])$ and $1/2$ on the clockwise oriented boundary arc $(ab)$ and the left side of $\gamma([0,t])$. 

All the quantities appearing in the SDE \eqref{SDE_intro} will be defined precisely in Section \ref{sec_charac}. We will see in Section \ref{subsec_abs_cov} that this SDE has a unique weak solution whose law is absolutely continuous with respect to $(2B_t, t \geq 0)$. This implies that massive SLE$_4$ with mass $m$ in $\Omega$ from $a$ to $b$ is absolutely continuous with respect to SLE$_4$ in $\Omega$ from $a$ to $b$.

To make the statement of Theorem \ref{theorem_intro} precise, we must detail the assumptions on the domain $\Omega$, the boundary points $a$ and $b$ and the discrete approximations $(\Omega_{\delta},a_{\delta}, b_{\delta})_{\delta}$ as well as define the topologies in which convergence holds. This will be done in Section \ref{sec_setup} and Section \ref{sec_annulus_cross} respectively. Theorem \ref{theorem_intro} will be shown under slightly weaker assumptions on the mass $m$: we will establish the result for a space-dependent mass $m: \Omega \to \mathbb{R}_{+}$ and its appropriate discretizations $(m_{\delta}: \Omega_{\delta} \to \mathbb{R}_{+})_{\delta}$, provided that the function $m$ is continuous and bounded. Defining massive SLE$_4$ with space-dependent mass also enables us to show that massive SLE$_4$ is conformally covariant, in a sense made precise in Section \ref{subsec_abs_cov}.

SLE$_4$ has a rich interplay with the planar continuum Gaussian free field (GFF). The prime example of this is the existence of a level line coupling between an SLE$_4$ curve and a GFF with appropriate boundary conditions \cite{Dub_SLE, IG_1}. One may wonder whether the massive version of SLE$_4$ defined via the SDE \eqref{SDE_intro} can be coupled in the same way to a massive GFF. The answer to this question turns out to be positive. Let $\Omega \subset \mathbb{C}$ be a bounded, open and simply connected domain and let $m >0$. The massive GFF in $\Omega$ with mass $m$ and Dirichlet boundary conditions is the centered Gaussian process $\Gamma^m$ indexed by smooth and compactly supported functions whose covariance is, for $f$ and $g$ two such functions,
\begin{equation*}
    \EE[(\Gamma^m,f)(\Gamma^m,g)] = \int_{\Omega} f(x)G_{\Omega}^m(x,y)g(y) dydx.
\end{equation*}
Above, $G_{\Omega}^m$ is the massive Green function in $\Omega$ with mass $m$ and Dirichlet boundary conditions, that is $G_{\Omega}^m$ is the inverse in the sense of the distributions of the operator $-\Delta+m^2$ with Dirichlet boundary conditions. As the GFF, the massive GFF is not defined pointwise but is only a generalized function. For a function $f: \partial \Omega \to \mathbb{R}$ with finitely many discontinuity points, we say that a massive GFF $\Gamma^m$ in $\Omega$ with mass $m$ has boundary conditions $f$ if $\Gamma^m$ has the same law as $\Gamma_0^m + \phi_{f}^m$, where $\Gamma_0^m$ is a massive GFF in $\Omega$ with mass $m$ and Dirichlet boundary conditions and $\phi_f^m$ is the massive harmonic extension of $f$ in $\Omega$. The existence of a coupling between a massive GFF with appropriate boundary conditions and a massive SLE$_4$ curve then reads as follows.

\begin{theorem} \label{coupling_intro}
Set $\lambda := \sqrt{\pi/8}$. Let $\Omega \subset \mathbb{C}$ be a bounded, open and simply connected domain and let $a, b \in \partial \Omega$. Denote by $\partial \Omega^{+}$, respectively $\partial \Omega^{-}$, the clockwise, respectively counterclockwise, oriented boundary arc $(ab)$. Let $m > 0$. Then there exists a coupling $(\Gamma^m, \gamma)$ where $\Gamma^m$ is a massive GFF in $\Omega$ with mass $m$ and boundary conditions $-\lambda$ on $\partial \Omega^{-}$ and $\lambda$ on $\partial \Omega^{+}$ and $\gamma$ is a massive SLE$_4$ in $\Omega$ from $a$ to $b$. In this coupling, for any stopping time $\tau$ for the filtration generated by $\gamma$, conditionally on $\gamma([0,\tau])$,
\begin{equation*}
    \Gamma^m = \Gamma^m_{\tau} + \phi_{\tau}^m
\end{equation*}
where $\Gamma_{\tau}^m$ is a massive GFF in $\Omega \setminus \gamma([0,\tau])$ with mass $m$ and Dirichlet boundary conditions and $\phi_{\tau}^m$ is the massive harmonic function in $\Omega \setminus \gamma([0,\tau])$ with boundary conditions $-\lambda$ on $\partial \Omega^{-}$ and the right side of $\gamma([0,\tau])$ and $\lambda$ on $\partial \Omega^{+}$ and the left side of $\gamma([0,\tau])$. Moreover, $\Gamma_t^m$ and $\phi_{\tau}^m$ are independent.   
\end{theorem}

The existence of such a coupling was already observed in the physics literature \cite{BBC} assuming absolute continuity of massive SLE$_4$ with respect to SLE$_4$ in the upper half-plane. We emphasize that here, Theorem \ref{coupling_intro} is only stated in bounded domains. Its proof, given in Section \ref{sec_coupling}, is analogous to that of the existence of a coupling between a GFF and an SLE$_4$ curve. We will actually establish the result in the case of a space-dependent mass $m: \Omega \to \mathbb{R}_{+}$, provided that $m$ is a bounded and continuous function. Conformal covariance of the massive GFF and of massive SLE$_4$ can then be used to extend this result to unbounded domains with appropriate space-dependent masses, that is masses which are inherited from a bounded domain via conformal mapping, see Section \ref{sec_coupling} for details.

\subsection{Outline of the proof of Theorem \ref{theorem_intro}}

Let us say a few words about the proof of Theorem \ref{theorem_intro}. Its strategy can be decomposed into three main steps. The first one is to show tightness of the sequence of massive harmonic explorer paths $(\gamma_{\delta})_{\delta}$ in an appropriate topology. One natural approach would be to show that the massive harmonic explorer is absolutely continuous with respect to the harmonic explorer and that the Radon-Nikodym derivative is well-behaved in the limit $\delta \to 0$. However, it is unclear whether absolute continuity holds at the level of the discrete curves and we therefore adopt a different approach relying on \cite{Smirnov} and \cite{Karrila}. Thanks to these results, to prove tightness of $(\gamma_{\delta})_{\delta}$, it suffices to show a suitable bound on the probability that the massive harmonic explorer crosses an annulus intersecting the boundary of $\Omega_{\delta}$. This is what we will establish in Section \ref{sec_tightness}.

Tightness of the sequence $(\gamma_{\delta})_{\delta}$ then implies the existence of subsequential limits. Characterizing these subsequential limits thus obtained is the aim of the next two steps of the proof. We will first see in Section \ref{subsec:def_mHE} that, for fixed $\delta >0$, the sequence of discrete massive harmonic functions $(h_{\delta, n}^m)_n$ is a martingale. We will then show that the continuum limit as $\delta \to 0$ of $h_{\delta}^{m}$ is the unique massive harmonic function in $\Omega$ with mass $m$ and boundary conditions $1/2$ on the clockwise oriented boundary arc $(ab)$ and $-1/2$ on the counter-clockwise oriented boundary arc $(ab)$. This result will in fact be shown for each $h_{\delta,n}^m$ under precise assumptions on the convergence of the domain at time $n$ to a continumm domain. These assumptions will hold thanks to the tightness of the sequence $(\gamma_{\delta})_{\delta}$ proved in the previous step. Convergence of these discrete massive harmonic functions is established in Section \ref{sec_martobs} by adapting some of the arguments of \cite{Discrete_analysis} to the massive setting. 

Finally, we will show that massive SLE$_4$ in $\Omega$ from $a$ to $b$ is the unique non-self-crossing curve $\gamma:[0,\infty) \to \overline{\Omega}$ such that the massive harmonic function $h^m_t$ with mass $m$ and boundary conditions $1/2$ on the clockwise oriented boundary arc $(ab)$ and the left side of $\gamma([0,t])$ and $-1/2$ on the counter-clockwise oriented boundary arc $(ab)$ and the right side of $\gamma([0,t])$ is a martingale for the filtration generated by $\gamma$. This characterization of massive SLE$_4$ is reminiscent of the characterization of SLE$_4$ by the martingale property of a certain (massless) harmonic function \cite{Dub_SLE, IG_1}. In the massive case, the proof follows the same strategy but involves some technicalities due to the presence of a mass. It is given in Section \ref{sec_charac}.

For the convenience of the reader, in Subsection \ref{subsec_ccl}, we state a more rigorous version of Theorem \ref{theorem_intro} and show how to prove it by combining the results of Section \ref{sec_tightness}, Section \ref{sec_martobs} and Section \ref{sec_charac}. 

\subsection*{Acknowledgements} The author is grateful to Ellen Powell for her guidance and her careful reading of an earlier version of the manuscript. The author also thanks Tyler Helmuth for his encouragements and discussions about statistical mechanics and Eveliina Peltola for discussions about SLE and massive SLE. This work was supported by an EPSRC
doctoral studentship.

\section{Setup} \label{sec_setup}

\subsection{Assumptions on the domain and Carath\'eodory convergence} \label{subsec:domain}

We consider an open, bounded and simply connected subset $\Omega$ of the complex plane $\mathbb{C}$. We assume that $0 \in \Omega$ and that there exists $R>0$ such that $\Omega \subset B(0,R)$. We fix two marked boundary points $a, b \in \partial \Omega$ and we assume that both $a$ and $b$ are degenerate prime ends of $\Omega$. That is, if $f: \mathbb{D} \to \Omega$ is a conformal map from the unit disc $\mathbb{D}$ to $\Omega$ and $\hat f $ is a bijective mapping from $\partial \mathbb{D}$ to $\partial \Omega$ induced by $f$, see \cite[Theorem~2.15]{Pommeranke}, then $f$ can be extended continuously at $z_a, z_b \in \partial \mathbb{D}$ by taking radial limits, where $z_a \in \partial \mathbb{D}$, respectively $z_b \in \partial \mathbb{D}$, is the preimage of $a$, respectively $b$, by $\hat f$. Here, the radial limit of $f$ at $\upzeta \in \partial \mathbb{D}$ is defined as $\lim_{\eps \to 0} f \circ P_{\eps}(\upzeta)$, whenever this limit exists and where for $\eps > 0$ and $z \in \mathbb{C}$, $P_{\eps}(z) = (z / \vert z \vert) \min \{ 1- \eps, \vert z \vert \}$. For a more detailed discussion on degenerate prime ends, the reader can consult \cite[Section~2.5]{Pommeranke}. These assumptions on $\Omega$ and the boundary points $a$ and $b$ will in particular allow us to use \cite[Theorem~4.2]{Karrila}.

We assume that $(\Omega_{\delta})_{\delta}$ is a sequence of graphs approximating $\Omega$ in a sense that we will now explain. For each $\delta > 0$, $\Omega_{\delta}$ is a simply connected subgraph of the triangular lattice $\delta \mathbb{T}$, so that every edge of $\Omega_{\delta}$ has length $\delta$. We denote by $V(\Omega_{\delta})$ the set of vertices of $\Omega_{\delta}$ and define the boundary $\partial \Omega_{\delta}$ of $\Omega_{\delta}$ as
\begin{equation*}
    \partial \Omega_{\delta} := \{ w \in \delta \mathbb{T} \setminus V(\Omega_{\delta}): \text{there exists $v \in V(\Omega_{\delta})$ such that $v \sim w$}\}
\end{equation*}
where $v \sim w$ means that there is an edge of $\delta \mathbb{T}$ connecting $v$ and $w$. With this definition of $\partial \Omega_{\delta}$, it is legitimate to set $\operatorname{Int}(\Omega_{\delta}):=V(\Omega_{\delta})$.

We associate to each $\Omega_{\delta}$ an open and simply connected polygonal domain
$\hat \Omega_{\delta} \subset \mathbb{C}$ by taking the union of all open hexagons with side length $\delta$ centered at vertices in $V(\Omega_{\delta})$. The marked boundary points $a$ and $b$ of $\partial \Omega$ are then approximated by sequences $(a_{\delta})_{\delta}$ and $(b_{\delta})_{\delta}$ where, for each $\delta >0$, $a_{\delta}$ and $b_{\delta}$ belong to $\partial \hat \Omega_{\delta}$. We assume that for all $\delta >0$, $0$ belongs to $\hat \Omega_{\delta}$, which is necessary to apply \cite[Theorem~4.2]{Karrila}. More importantly, we assume that the approximations $(\hat \Omega_{\delta}; a_{\delta}, b_{\delta})_{\delta}$ converge to $(\Omega; a, b)$ in the Carath\'eodory sense. That is,
\begin{itemize}
    \item each inner point $z \in \Omega$ belongs to $\hat \Omega_{\delta}$ for $\delta$ small enough;
    \item for every boundary point $\zeta \in \partial \Omega$, there exist $\upzeta_{\delta} \in \partial \hat \Omega_{\delta}$ such that $\upzeta_{\delta} \to \upzeta$ as $\delta \to 0$.
\end{itemize}
This can be rephrased in terms of conformal maps. Let $\psi: \Omega \to \mathbb{D}$ be a conformal map such that $\psi(a)=1$ and $\psi(0)=0$. Similarly, for each $\delta >0$, let $\psi_{\delta}: \hat \Omega_{\delta} \to \mathbb{D}$ be a conformal map such that $\psi_{\delta}(a_{\delta})=1$ and $\psi_{\delta}(0)=0$. Then, by \cite[Theorem~1.8]{Pommeranke}, the Carath\'eodory convergence of $(\hat \Omega_{\delta}; a_{\delta}, b_{\delta})_{\delta}$ to $(\Omega; a, b)$ is equivalent to
\begin{align*}
    & \psi_{\delta} \to \psi \quad \text{uniformly on compact subsets of $\Omega$ and}\\
    & \psi_{\delta}^{-1} \to \psi^{-1} \quad \text{uniformly on compact subsets of $\mathbb{D}$}.
\end{align*}
Furthermore, we assume that $a_{\delta}$, respectively $b_{\delta}$, is a close approximation of $a$, respectively $b$, as defined by Karrila in \cite[Section~4.3]{Karrila}. To lighten the notations, we identify the prime ends $a_{\delta}$ and $a$ with their corresponding radial limit points. For $r>0$, let $S_r$ be the arc of $\partial B(a,r) \cap \Omega$ disconnecting in $\Omega$ the prime end $a$ from $0$ and that is closest to $a$. In other words, $S_r$ is the last arc from the (possibly countable) collection $\partial B(a,r) \cap \Omega$ of arcs that a path running from $0$ to $a$ inside $\Omega$ must cross. Such an arc exists by \cite[Lemma~A.1]{Karrila} and approximation by radial limits. $a_{\delta}$ is then said to be a close approximation of $a$ if
\begin{itemize}
    \item $a_{\delta} \to a$ as $\delta \to 0$; and
    \item for each $r$ small enough and for all sufficiently (depending on $r$) small $\delta$, the boundary point $a_{\delta}$ of $\hat \Omega_{\delta}$ is connected to the midpoint of $S_r$ inside $\hat \Omega_{\delta} \cap B(a,r)$. 
\end{itemize}

\subsection{Definition of the discrete model} \label{subsec:def_mHE}

The massive harmonic explorer is a massive version of the harmonic explorer studied in \cite{HE}. Let $m: \Omega \to \mathbb{R}_{+}$ be a continuous function bounded above by some constant $\overline{m} > 0$. For each $\delta > 0$, we assign a weight to the edges of the graph $\Omega_{\delta}$ as follows: if $z \in V(\Omega_{\delta})$ and $w \in V(\Omega_{\delta})$ is connected to $z$ in $\Omega_{\delta}$, then the oriented edge $(zw)$ from $z$ to $w$ has weight $1-m_{d}^2(z)\delta^2$, where we have set, for $v \in \Omega_{\delta}$,
\begin{equation} \label{def_discrete_mass}
    m_d^2(v) := \frac{cm^2(v)}{6\tan(\theta)}.
\end{equation}
Above, the constant $c>0$ is such that the faces of the hexagonal lattice dual to $\delta \mathbb{T}$ have area $A_{\delta} = c\delta^2$, that is $c=9\sqrt{3}/8$. $\theta$ is defined as \cite[Figure~1.B]{Discrete_analysis}, that is $\theta = \pi/6$ for the triangular lattice. If $v \notin \Omega \cap \Omega_{\delta}$, we set $m_d^2(v) = 0$. Notice that the weight of an edge depends on its orientation: the oriented edge $(wz)$ in $\Omega_{\delta}$ has weight $1-m_d^2(w)\delta^2$. Observe also that for any $v \in V(\Omega_{\delta})$,
\begin{equation*}
    m_d^2(v) \leq \frac{c\overline{m}^2}{6\tan(\theta)} =: \overline{m}_{d}^{2}.
\end{equation*}

For each $\delta > 0$ such that $\delta < \overline{m}_d^{-1}$, these edge weights naturally give rise to a massive random walk $X^m$ on $\Omega_{\delta}$. More precisely, if the walk $X^m$ is at a vertex $v \in \Omega_{\delta}$, then at the next step, it jumps to one of its neighbors with probability $1-m^2_d(v)\delta^2$ or is killed with probability $m_d^2(v)\delta^2$. The jumps to its neighbors are uniform, that is if $w$ is connected to $v$ by an edge, then the walk has probability $(1-m^2_d(v)\delta^2)/6$ to jump to $w$ at its next step. We denote by $\tau^{\star}$ the killing time of the walk and by $\tau_{\partial \Omega_{\delta}}$ the hitting time of $\partial \Omega_{\delta}$ by the walk, when it started in $\operatorname{Int}(\Omega_{\delta})$. Accordingly, we denote by $\PP^{(m)}_{w}(\tau^{\star} \leq \tau_{\partial \Omega_{\delta}})$ the probability that the walk is killed before reaching the boundary $\partial \Omega_{\delta}$ when it started at $w \in \operatorname{Int}(\Omega_{\delta})$. Such a massive random walk is intimately connected to discrete massive harmonic functions with mass $m$, in the same way as random walk is connected to discrete harmonic functions. Discrete massive harmonic functions are defined as follows: a function $h: \Omega_{\delta} \to \mathbb{R}$ is said to be discrete massive harmonic with mass $m$ if for any $v \in \operatorname{Int}(\Omega_{\delta})$,
\begin{equation} \label{def_mharm}
    h(v) = \frac{1-m_d^2(v)}{6} \sum_{v \sim w} h(w).
\end{equation}
As in the non-massive case, one can also define the massive harmonic measure $H^{(m)}(\cdot, E_1)$ of a subset $E_1$ of $\partial \Omega_{\delta}$. This is the unique discrete massive harmonic function with mass $m$ in $\Omega_{\delta}$ and boundary value $1$ on $E_1$ and $0$ on $\partial \Omega_{\delta} \setminus E_1.$ For $v \in \operatorname{Int}(\Omega_{\delta})$, $H^{(m)}(v,E_1)$ can be interpreted as the probability that a massive random walk with mass $m$ started at $v$ is not killed before leaving $\Omega_{\delta}$ and exits $\Omega_{\delta}$ through $E_1$. Observe that $H^{(m)}(v, \partial \Omega_{\delta}) = 1 - \PP_v^{(m)}(\tau^\star \leq \tau_{\partial \Omega_{\delta}})$.

For $0<\delta < \overline{m}_d^{-1}$, the massive harmonic explorer with mass $m$ is a random path $\gamma_{\delta}$ on the dual of $\Omega_{\delta}$ defined as follows. On $\partial \Omega$, we assign the sign $+$ to the clockwise oriented boundary arc $(ab)$, denoted $\partial \Omega^{+}$, and the sign $-$ to the counter-clockwise oriented boundary arc $(ab)$, denoted $\partial \Omega^{-}$. Correspondingly, for each $\delta >0$, the boundary $\partial \hat \Omega_{\delta}$ of $\hat \Omega_{\delta}$ is split into two parts: the clockwise oriented boundary arc $(a_{\delta}b_{\delta})$ has sign $+$ while the counter-clockwise oriented boundary arc $(a_{\delta}b_{\delta})$ has sign $-$. This naturally defines a partition of $\partial \Omega_{\delta}$ into two sets of vertices: the vertices of  $\partial \Omega_{\delta}$ that belong to the clockwise oriented boundary arc $(a_{\delta}b_{\delta})$ have sign $+$ while the vertices that belong to counter-clockwise oriented boundary arc $(a_{\delta}b_{\delta})$ have sign $-$. The vertices $a_{\delta}$ and $b_{\delta}$, in case they belong to $\partial \Omega_{\delta}$, are assigned an arbitrary sign. We denote the set of vertices with sign $+$, respectively $-$, by $\partial \Omega_{\delta}^{+}$, respectively $\partial \Omega_{\delta}^{-}$.

Using this partition of $\partial \Omega_{\delta}$, we then let $h_{\delta}^{m}$ be the unique discrete massive harmonic function with mass $m$ in $\Omega_{\delta}$ and boundary values $-1/2$ on $\partial \Omega_{\delta}^{-}$ and $1/2$ on $\partial \Omega_{\delta}^{+}$. Let $m_0$ be the middle point of the unique edge $e_0$ connecting two vertices of opposite sign to which $a_{\delta}$ belongs; $m_0$ is the starting point of the massive harmonic explorer $\gamma_{\delta}$. The edge $e_0$ bounds a triangle $f_0$ of $\Omega_{\delta}$ and we denote by $v_1$ the vertex of $\Omega_{\delta} \cup \partial \Omega_{\delta}$ that is not an endpoint of $e_0$. The explorer then turns right with probability $h_{\delta}^m(v_1)+1/2$, that is, with probability $h_{\delta}^m(v_1)+1/2$, $\gamma_{\delta}$ traces the broken line from $m_0$ to the center of $f_0$ and then from the center of $f_0$ to the middle point $m_1$ of the right edge of $f_0$. In that case, $v_1$ becomes a boundary vertex with sign $+$ and we set $\Omega_{\delta, 1}:= \Omega_{\delta} \setminus \{v_1\}$, $\partial \Omega_{\delta, 1}^{+} := \partial \Omega_{\delta}^{+} \cup \{ v_1 \}$ and $\partial \Omega_{\delta, 1}^{-}:= \partial \Omega_{\delta}^{-}$. With complementary probability, the explorer turns left, that is $\gamma_{\delta}$ traces the broken line from $m_0$ to the center of $f_0$ and then from the center of $f_0$ to the middle point $m_1$ of the left edge of $f_0$. In that case, $v_1$ becomes a boundary vertex with sign $-$ and we set $\Omega_{\delta, 1}:= \Omega_{\delta} \setminus \{v_1\}$, $\partial \Omega_{\delta, 1}^{-} := \partial \Omega_{\delta}^{-} \cup \{ v_1 \}$ and $\partial \Omega_{\delta, 1}^{+}:= \partial \Omega_{\delta}^{+}$. In both cases, on $\Omega_{\delta,1}$, we let $h_{\delta, 1}^m$ be the unique discrete massive harmonic function with mass $m$ and boundary value $-1/2$ on $\partial \Omega_{\delta,1}^{-}$ and $1/2$ on $\partial \Omega_{\delta, 1}^{+}$. For the second step, we repeat the same procedure but with respect to the vertex $v_2$, defined analogously to $v_1$, and using the function $h_{\delta, 1}^{m}$. We continue until the path $\gamma_{\delta}$ hits the edge connecting two vertices of opposite sign to which $b_{\delta}$ belongs, which almost surely happens in finite time. We update the graphs and their boundary at each step to obtain random sequences $(\Omega_{\delta,n})_n$, $(\partial \Omega_{\delta,n}^{+})_n$ and $(\partial \Omega_{\delta,n}^{-})_n$ and the corresponding sequence of discrete massive harmonic functions $(h_{\delta,n}^m)_n$. See again Figure \ref{fig_mHE} for a dual perspective on the hexagonal lattice.

This discrete model is well-defined, in the sense that each step of the massive harmonic explorer is chosen according to a probability measure. Indeed, notice that almost surely, for any $n \in \mathbb{N}$ and any $w \in \operatorname{Int}(\Omega_{\delta,n}) \cup \partial \Omega_{\delta,n}$,
\begin{equation} \label{h_harm_meas}
    h_{\delta,n-1}^m(w) = \frac{1}{2} \bigg( H_{\delta,n-1}^{(m)}(w) - \tilde H_{\delta,n-1}^{(m)}(w)\bigg)
\end{equation}
where $H_{\delta,n-1}^{(m)}(w)$, respectively $\tilde H_{\delta,n-1}^{(m)}(w)$, is the discrete massive harmonic measure with mass $m$ of $\partial \Omega_{\delta,n-1}^{+}$, respectively $\partial \Omega_{\delta,n-1}^{-}$, seen from $w$. This equality is a consequence of the uniqueness of discrete massive harmonic functions with prescribed boundary conditions: the functions on each side of the equality are discrete massive harmonic with mass $m$ and boundary values $1/2$ on $\partial \Omega_{\delta,n-1}^{+}$ and $-1/2$ on $\partial \Omega_{\delta,n-1}^{-}$. This shows that almost surely, for any $n \in \mathbb{N}$,
\begin{equation*}
    h_{\delta,n-1}^m(v_n) + \frac{1}{2} = \frac{1}{2} \bigg( H_{\delta,n-1}^{(m)}(v_n) - \tilde H_{\delta,n-1}^{(m)}(v_n)\bigg) + \frac{1}{2}
\end{equation*}
and the right-hand side almost surely takes values in $[0,1]$ as both $H_{\delta,n-1}^{(m)}(v_n)$ and $\tilde H_{\delta,n-1}^{(m)}(v_n)$ almost surely take values in $[0,1]$. Observe also that the equality \eqref{h_harm_meas} implies that this model has the following symmetry: almost surely, for all $n \in \mathbb{N}$ and all $v \in \operatorname{Int}(\Omega_{\delta}) \cup \partial \Omega_{\delta}$,
\begin{equation*}
    1- (h_{\delta,n}^m(v)+1/2) = \tilde h_{\delta,n}^m(v)+1/2
\end{equation*}
where $\tilde h_{\delta, n}^m$ is the unique discrete massive harmonic function with mass $m$ in $\Omega_{\delta, n}$ and boundary value $-1/2$ on $\partial \Omega_{\delta,n}^{+}$ and $1/2$ on $\partial \Omega_{\delta,n}^{-}$. 

Throughout the text, for $0 < \delta < \overline{m}_{d}^{-1}$, we denote by $\PP_{\delta}^{(\Omega,a,b,m)}$ the probability measure on paths on the dual of $\Omega_{\delta}$ induced by the massive harmonic explorer from $a$ to $b$ in $\Omega$ with mass $m$. $\EE_{\delta}^{(\Omega,a,b,m)}$ denotes the corresponding expectation. For ease of notations, when $\delta > 0$ and it is clear from the context that the path $\gamma_{\delta}$ is distributed according to $\PP_{\delta}^{(\Omega,a,b,m)}$, we simply write $\gamma$ instead of $\gamma_{\delta}$. Notice also that the family of probability measures $(\PP_{\delta}^{(\Omega,a,b,m)})_{\delta}$ implicitly depends on the sequence $(\hat \Omega_{\delta},a_{\delta},b_{\delta})_{\delta}$ approximating $(\Omega,a,b)$. 

Observe that for $0<\delta < \overline{m}_d^{-1}$, $\PP_{\delta}^{(\Omega,a,b,m)}$ satisfies the following domain Markov property. Let $\tau$ be a stopping time for the filtration generated by $\gamma$. Then, almost surely,
\begin{equation} \label{discrete_DMP}
    \PP_{\delta}^{(\Omega,a,b,m)}(\cdot \, \vert \, \gamma([0,\tau])) = \PP_{\delta}^{(\Omega_{\tau}, \gamma(\tau), b, m)}(\cdot)
\end{equation}
where, with a slight abuse of notation, $\PP_{\delta}^{(\Omega_{\tau}, \gamma(\tau), b, m)}$ denotes the probability measure induced on paths by the massive harmonic explorer with mass $m$ from $\gamma(\tau)$ to $b_\delta$ in $\Omega_{\tau} = \Omega_{\delta,\tau}$ and exploration starting with respect to $h_{\delta,\tau}^{m}(v_{\tau+1}) +1/2$. That is, conditionally on $\gamma([0,\tau])$, at step $\tau+1$, the explorer turns right with probability $h_{\delta,\tau}^{m}(v_{\tau+1}) +1/2$ and, with complementary probability, it turns left and the rest of the path is traced following the same procedure as the one described above.

Finally, as in the case of the non-massive harmonic explorer, for each $v \in \operatorname{Int}(\Omega_{\delta}) \cup \partial \Omega_{\delta}$, we notice that for each $\delta$, under $\mathbb{P}^{(\Omega, a, b, m)}_\delta$, the process $(h_{\delta, n}^m(v), n \geq 0)$ is a martingale with respect to the filtration $(\mathcal{F}_{\delta, n})_n$ generated $\gamma_{\delta}$. That is, for $n \in \mathbb{N}$, $\mathcal{F}_{\delta,n} = \sigma(\gamma_{\delta}([0,n]))$. We record this fact in the proposition below.

\begin{proposition} \label{prop_martingale_obs}
    Let $\delta >0$. In the above setting and using the same notations, for each $v \in \operatorname{Int}(\Omega_{\delta}) \cup \partial \Omega_{\delta}$, the process $(h_{\delta, n}^m(v), n \geq 0)$ is a martingale with respect to the filtration $(\mathcal{F}_{\delta, n})_n$.
\end{proposition}

\begin{proof}
    Let $\delta >0$ and fix $n \in \mathbb{N}$. We first consider the case $v=v_{n+1}$. Then, almost surely,
    \begin{align*}
        \EE_{\delta}^{(\Omega,a,b,m)}[h_{\delta, n+1}^m(v_{n+1}) \vert \mathcal{F}_{\delta, n}] &= \frac{1}{2}\times \PP_{\delta}^{(\Omega,a,b,m)}(h_{\delta, n+1}^m(v_{n+1})=1/2 \, \vert \, \mathcal{F}_{\delta, n}) \\
        &- \frac{1}{2} \times (1-\PP_{\delta}^{(\Omega,a,b,m)}(h_{\delta, n+1}^m(v_{n+1})=1/2 \, \vert \, \mathcal{F}_{\delta, n})) \\
        &= \frac{1}{2} \times \bigg(h_{\delta, n}^m(v_{n+1}) + \frac{1}{2} \bigg) + \frac{-1}{2} \times \bigg( \frac{1}{2} -h_{\delta, n}^m(v_{n+1}) \bigg) \\
        &= h_{\delta, n}^m(v_{n+1}).
    \end{align*}
    $h_{\delta, n}^m$ is also the discrete massive harmonic extension of its restriction to $\partial \Omega_{\delta, n+1}$ and similarly for $h_{\delta, n+1}^{m}$. Since taking the massive harmonic extension of a function defined on the boundary of a domain is a linear operation and almost surely $\EE_{\delta}^{(\Omega,a,b,m)}[h_{\delta, n+1}^m(v) \vert \mathcal{F}_{\delta, n}] = h_{\delta, n}(v)$ for all $v \in \partial \Omega_{\delta, n+1}$, the same relation holds almost surely for every $v \in \operatorname{Int}(\Omega_{\delta, n+1})$. Thus, for each $v \in \operatorname{Int}(\Omega_{\delta}) \cup \partial \Omega_{\delta}$, $(h_{\delta,n}^m(v), n \geq 0)$ is a martingale.
\end{proof}

In view of Proposition \ref{prop_martingale_obs}, we call the functions $(h_{\delta,n}^m)_n$ the martingale observables of the massive harmonic explorer.

\section{Tightness} \label{sec_tightness}

In this section, we establish tightness of the sequence of massive harmonic explorer paths $(\gamma_{\delta})_{\delta}$, where for each $\delta > 0$, $\gamma_{\delta}$ is distributed according to $\PP_{\delta}^{(\Omega,a,b,m)}$. Tightness is shown is three different topologies, using the approach laid out in \cite{Smirnov}. This approach applies to families of probability measures supported on a metric space of curves, whose construction is explained in Section \ref{sec_curves}. Under a condition on the probability of a certain crossing event, tightness of a sequence of random curves supported on this space then follows from \cite[Theorem~1.7]{Smirnov}. This theorem is phrased in terms of Loewner chains and therefore, before recalling it, we provide some background on the Loewner equation in Section \ref{sec_Loewner}. \cite[Theorem~1.7]{Smirnov} is then discussed in Section \ref{sec_annulus_cross}, where we also explain the condition on the probability of the crossing event and describe the topologies in which tightness is obtained. Finally, in Section \ref{sec_proof_crossing}, we show that the condition on the probability of the crossing event is satisfied by the sequence $(\gamma_{\delta})_{\delta}$ distributed according to $(\PP_{\delta}^{(\Omega,a,b,m)})_{\delta}$, thus proving tightness of $(\gamma_{\delta})_{\delta}$ in the aforementioned topologies.

\subsection{The space of curves} \label{sec_curves}

Following \cite{Smirnov}, the space of curves that we will consider is a subspace of the space of continuous mappings from $[0,1]$ to $\mathbb{C}$ modulo reparametrization. More precisely, let
\begin{equation*}
    \mathcal{C}':= \left \{
    f \in \mathcal{C}([0,1], \mathbb{C}): 
    \begin{aligned}
    &\text{ either $f$ is not constant on any subinterval of $[0,1]$} \\ &\text{or $f$ is constant on $[0,1]$} 
    \end{aligned}
    \right \}
\end{equation*}
and let $f_1, f_2 \in \mathcal{C}'$ be equivalent if there exists an increasing homeomorphism $\psi: [0,1] \to [0,1]$ with $f_2 = f_1 \circ \psi$. We denote by $[f]$ the equivalence class of $f$ under this equivalence relation and set
\begin{equation*}
    X(\mathbb{C}):= \{ [f]: f \in \mathcal{C}'\}.
\end{equation*}
$X(\mathbb{C})$ is called the space of curves. We turn $X(\mathbb{C})$ into a metric space by equipping it with the metric
\begin{equation*}
    d_X(f,g) := \inf \{ \| f_0 - g_0 \|_{\infty}: \, f_0 \in [f], g_0 \in [g] \}.
\end{equation*}
$(X(\mathbb{C}), d_X)$ is a separable and complete metric space, but it is not compact. Given $\Omega$ a subset of $\mathbb{C}$ with $\partial \Omega \neq \emptyset$ and two marked boundary points $a$ and $b$, we define the space of simple curves from $a$ to $b$ in $\Omega$ as
\begin{equation*}
    X_{\text{simple}}(\Omega, a, b) := \{ [f]: f \in \mathcal{C}', f((0,1)) \subset \Omega, f(0)=a, f(1)=b, \, \text{$f$ injective} \}.
\end{equation*}
We then let $X_0(\Omega, a, b)$ be the closure of $X_{\text{simple}}(\Omega, a, b)$ in $X(\mathbb{C})$ with respect to the metric $d_X$. Curves in $X_0(\Omega, a, b)$ run from $a$ to $b$, may touch $\partial \Omega$ elsewhere than at their endpoints, may touch themselves and have multiple points but they can have no transversal self-crossings. Notice that if $(\PP_n)_n$ is a sequence of probability measures supported on $X_{\text{simple}}(\Omega, a, b)$ that converges weakly to a probability measure $\PP^{*}$, then a priori $\PP^{*}$ is supported on $X_0(\Omega, a, b)$.

\subsection{Loewner chains} \label{sec_Loewner}

Denote by $\HH$ the complex upper-half plane $\{ z \in \mathbb{C}: \Im(z) > 0 \}$ and let $\gamma: [0,\infty) \to \overline{\mathbb{H}}$ be a non-self-crossing curve targeting $\infty$ and such that $\gamma(0)=0$. For $t \geq 0$, let $K_t$ be the hull generated by $\gamma([0,t])$, that is $\HH \setminus K_t$ is the unbounded connected component of $\HH \setminus \gamma([0,t])$. In the case where $\gamma([0,t])$ is non-self-touching, $K_t$ is simply given by $\gamma([0,t])$. For each $t \geq 0$, it is easy to see that there exists a unique conformal $g_t: \HH \setminus K_t \to \HH$ satisfying the normalization $g_t(\infty) = \infty$ and such that $\lim_{z \to \infty} (g_t(z) - z) = 0$. It can then be proved that $g_t$ satisfies the asymptotic
\begin{equation*}
    g_t(z) = z + \frac{a_1(t)}{z} + O(\vert z \vert^{-2}) \quad \text{as } \vert z \vert \to \infty.
\end{equation*}
The coefficient $a_1(t)$ is equal to $\text{hcap}(K_t)$, the half-plane capacity of $K_t$, which, roughly speaking, is a measure of the size of $K_t$ seen from $\infty$. Moreover, one can show that $a_1(0)=0$ and that $t \mapsto a_1(t)$ is continuous and strictly increasing. Therefore, the curve $\gamma$ can be reparametrized in such a way that at each time $t$, $a_1(t) = 2t$. $\gamma$ is then said to be parameterized by half-plane capacity.

In this time-reparametrization and with the normalization of $g_t$ just described, it is known that there exists a unique real-valued function $t \mapsto W_t$, called the driving function, such that the following equation, called the Loewner equation, is satisfied:
\begin{equation} \label{eq_Loewner}
    \partial_t g_t(z) = \frac{2}{g_t(z) - W_t},  \quad g_0(z)=z, \quad \text{for all $z \in \HH \setminus K_t$}.
\end{equation}
Indeed, it can be shown that $g_t$ extends continuously to $\gamma(t)$ and setting $W_t = g_t(\gamma(t))$ yields the above equation, see e.g. \cite[Chapter~4]{book_Lawler} and \cite[Chapter~4]{book_SLE}.

Conversely, given a continuous and real-valued function $t \mapsto W_t$, one can construct a locally growing family of hulls $(K_t)_t$ by solving the equation \eqref{eq_Loewner}. Under additional assumptions on the function $t \mapsto W_t$, the family of hulls obtained using \eqref{eq_Loewner} is generated by a curve, in the sense explained above \cite{slit_Loewner}.

Schramm-Loewner evolutions, or SLE for short, are random Loewner chains introduced by Schramm \cite{Schramm_SLE}. For $\kappa \geq 0$, SLE$_\kappa$ is the Loewner chain obtained by considering the Loewner equation \eqref{eq_Loewner} with driving function $W_t = \sqrt{\kappa}B_t$, where $(B_t, t \geq 0)$ is a standard one-dimensional Brownian motion. As such, SLE$_{\kappa}$ is defined in $\HH$ but, thanks to the conformal invariance of the Loewner equation, SLE$_{\kappa}$ can be defined in any simply connected domain $\Omega \subset \mathbb{C}$ with two marked boundary points $a, b \in \partial \Omega$ by considering a conformal map $\phi: \Omega \to \HH$ with $\phi(a)=0$ and $\phi(b) = \infty$ and taking the image of SLE$_{\kappa}$ in $\HH$ by $\phi^{-1}$. In particular, SLE$_{\kappa}$ is conformally invariant and it turns out that this conformal invariance property together with a certain domain Markov property characterize the family (SLE$_{\kappa}, \kappa \geq 0)$. In what follows, we will be interested in the special case $\kappa =4$. SLE$_4$ can be shown to be almost surely generated by a simple continuous curve that is transient and whose Hausdorff dimension is $3/2$. For a proof of these facts, we refer the reader to \cite[Chapter~5]{book_SLE} and references therein.

\subsection{Annulus crossing estimate and tightness} \label{sec_annulus_cross}

To show tightness, we rely on the framework developed by Kemppainen and Smirnov in \cite{Smirnov}. According to these results, tightness of the sequence of massive harmonic explorers $(\gamma_{\delta})_{\delta}$ can be established in three different topologies if, under $\PP_{\delta}^{(\Omega, a, b, m)}$, an appropriate and uniform in $\delta$ upper bound on the probability of a certain crossing event holds. In our case, we will prove that under $\PP_{\delta}^{(\Omega, a, b, m)}$, the probability that the massive harmonic explorer $\gamma$ crosses a so-called avoidable annulus of modulus $r/R$, defined just below, decays geometrically in $r/R$, with constants independent of $\delta$. Essentially, we will show that the family $(\PP_{\delta}^{(\Omega,a,b,m)})_{\delta < \overline{m}_d^{-1}}$ satisfies Condition G.3 of \cite{Smirnov}. Let us now describe in more detail this condition and its consequences regarding tightness of the family $(\PP_{\delta}^{(\Omega,a,b,m)})_{\delta < \overline{m}_d^{-1}}$.

Let $z_0 \in \mathbb{C}$ and $0 < r \leq R$ and denote by $A(z_0, r, R)$ the annulus centered at $z_0$ with inner radius $r$ and outer radius $R$, that is
\begin{equation*}
    A(z_0,r,R) := \{ z \in \mathbb{C}: r < \vert z_0-z \vert < R \}.
\end{equation*}
Let $\tau$ be a stopping time for $\gamma$ and set $\Omega_{\tau}:= \Omega \setminus \gamma([0,\tau])$. The avoidable component $A_{\tau}^{\Omega}$ of an annulus $A(z_0,r,R)$ at time $\tau$ in $\Omega$ is defined as follows. If $\partial B(z_0,r) \cap \partial \Omega_{\tau} = \emptyset$, then $A_{\tau}^{\Omega} = \emptyset$. Otherwise,
\begin{equation*}
    A_{\tau}^{\Omega} := \left \{ z \in \Omega_{\tau} \cap A(z_0,r,R): 
    \begin{aligned}
    &\text{the connected component of $z$ in $\Omega_{\tau} \cap A(z_0,r,R)$} \\ 
    &\text{does not disconnect $\gamma(\tau)$ from $b$ in $\Omega_{\tau}$}
    \end{aligned}
    \right \}.
\end{equation*}
If an annulus $A(z_0,r,R)$ is such that $A^{\Omega}_{\tau} \neq \emptyset$, we say that $A(z_0,r,R)$ is an avoidable annulus at time $\tau$. Furthermore, if $\gamma([0,\tau])$ crosses $A(z_0,r,R)$ in one of the connected components of $A_{\tau}^{\Omega}$, $\gamma$ is said to make an unforced crossing of $A(z_0,r,R)$ in $\Omega_{\tau}$. For a family $(\PP_{n})_{n}$ of probability measures supported on $X_0(\Omega,a,b)$, Condition G.3 of \cite{Smirnov} then reads as follows. Here, curves are parametrized from $0$ to $1$.

\paragraph{Condition G.3}\label{condG3} The family $(\PP_{n})_{n}$ is said to satisfy a geometric power-law bound on an unforced crossing if there exist $K > 0$ and $\Delta > 0$ such that for any $n$, for any stopping time $0 \leq \tau \leq 1$ and for any annulus $A = A(z_0, r, R)$ where $0 < r \leq R$,
\begin{equation*}
    \PP_{n}(\gamma([\tau,1]) \text{ makes a crossing of $A$ which is contained in $A_{\tau}^{\Omega}$} \vert \, \gamma([0,\tau])) \leq K\bigg(\frac{r}{R}\bigg)^{\Delta}.
\end{equation*}
By \cite[Theorem~1.7]{Smirnov}, verifying this condition for the family $(\PP_n)_n$ allows one to establish tightness of the curves $\gamma$ under $(\PP_n)_n$ in three different topologies (see below for the details in our setting).

We wish to apply this result to the family of massive harmonic explorers under $\PP_\delta^{(\Omega,a,b,m)}$. However, note that under this law, $\gamma=\gamma_\delta$ is an element of $X_0(\hat \Omega_\delta, a_\delta, b_\delta)$ for each $\delta$, whereas Condition G.3 is stated for a family of probability measures all supported on the same $X_0(\Omega, a, b)$. Thus we first need to uniformize the picture. To this end, we let $\phi: \Omega \to \mathbb{H}$ be a conformal map such that $\phi(a) = 0$ and $\phi(b)=\infty$ and similarly, for $\delta > 0$, let $\phi_{\delta}: \hat \Omega_{\delta} \to \mathbb{H}$ be a conformal map such that $\phi_{\delta}(a_{\delta}) = 0$ and $\phi_{\delta}(b_{\delta})=\infty$. For $\delta > 0$, we denote by $W_{\delta}$ the driving function of $\gamma_{\delta}^{\HH}:= \phi_{\delta}(\gamma_{\delta})$, when the curve is parametrized by half-plane capacity. 

In this setting \cite[Theorem~1.7]{Smirnov} yields the following. Suppose that the laws of $(\gamma_\delta^{\HH})_{0<\delta<\overline{m}_\delta^{-1}}$  satisfy Condition \hyperref[condG3]{G.3}, (where $\gamma_\delta$ has law  $(\PP_{\delta}^{(\Omega,a,b,m)})$ for each $\delta$, and $\phi_\delta$ is as above). Then 
\begin{enumerate}[label=\text{T.\arabic*},ref=T.\arabic*]
    \item \label{topo_1} $(\gamma_{\delta}^{\HH})_{\delta}$ is tight in the space of curves $X_0(\HH,0,\infty)$ equipped with the metric $d_X$;
    \item \label{topo_2} $(\gamma_{\delta}^{\HH})_{\delta}$ is tight in the metrizable space of continuous function on $[0, \infty)$ with the topology of uniform convergence on compact subsets of $[0, \infty)$;
    \item \label{topo_3} $(W_{\delta})_{\delta}$ is tight in the metrizable space of continuous function on $[0, \infty)$ with the topology of uniform convergence on compact subsets of $[0, \infty)$.
\end{enumerate}
Moreover, under the assumption that $(\hat \Omega_{\delta}; a_{\delta}, b_{\delta})_{\delta}$ converges in the Carath\'eodory sense to $(\Omega;a,b)$ and that $(a_{\delta})$ and $(b_{\delta})_{\delta}$ are close approximations of the degenerate prime ends $a$ and $b$, \cite[Corollary~1.8]{Smirnov} and \cite[Theorem~4.2]{Karrila} also imply that
\begin{itemize}
    \item the family $(\gamma_{\delta})_{\delta}$ is tight and if $\gamma^{\HH}$ denotes the weak limit of a subsequence $(\gamma_{\delta_k}^{\HH})_{\delta_k}$ of $(\gamma_{\delta}^{\HH})_{\delta}$ (recall that these curves are parametrized by half-plane capacity) in one of the topologies \eqref{topo_1}--\eqref{topo_3}, then the subsequence $(\gamma_{\delta_k})_{\delta_k}$ converges weakly in the space $X(\mathbb{C})$ equipped with the metric $d_X$ to a random curve $\gamma$ that is almost surely supported on $\overline{\Omega}$ and has the same law as $\phi_{\Omega}^{-1}(\gamma^{\HH})$.
\end{itemize}

\subsection{Proof of the annulus crossing estimate} \label{sec_proof_crossing}

As explained above, to establish tightness of the massive harmonic explorers $(\gamma_{\delta})_{\delta}$ in the topologies \eqref{topo_1} -- \eqref{topo_3}, we show that the family $(\PP_{\delta}^{(\Omega,a,b,m)})_{\delta}$ satisfies Condition \hyperref[condG3]{G.3}. This is the content of the following proposition.

\begin{proposition} \label{prop_crossing}
There exist constants $K, \alpha>0$ such that for any $0<\delta<\overline{m}_d^{-1}$, for any stopping time $\tau$ and any annulus $A=A(z_0,r,R)$,
\begin{equation*}
    \PP_{\delta}^{(\Omega,a,b,m)}(\gamma([\tau,1]) \, \text{makes a crossing of $A$ which is contained in $A_{\tau}^{\hat \Omega_{\delta}}$} \vert \, \gamma([0,\tau])) \leq K \bigg (\frac{r}{R}\bigg)^{\alpha}.
\end{equation*}
This implies in particular that the family $(\PP_{\delta}^{(\Omega,a,b,m)})_{\delta < \overline{m}_d^{-1}}$ satisfies Condition \hyperref[condG3]{G.3}.
\end{proposition}

The proof of Proposition \ref{prop_crossing} relies on a martingale argument similar to that used in the proof of \cite[Proposition~6.3]{HE}. Our martingale is a sum of two terms. One of them is the massive version of the martingale used in the proof of \cite[Proposition~6.3]{HE} or, in other words, the total mass of massive random-walk excursions from a well-chosen set of boundary vertices to boundary vertices with sign $+$. However, this term is by itself not a martingale because in our setting, the massive harmonic measure of $\partial \Omega_{\delta,n}$ seen for a vertex $v \in \Omega_{\delta,n}$ is not a martingale. To compensate the drift that arises from it, we must add another term, which is the second term in our martingale. This term accounts for the probability that excursions get killed before leaving the domain. To control the first term, using simple inequalities for the massive harmonic measure and up to some minor modifications, one can argue as in the proof of \cite[Propostion~6.3]{HE}. However, controlling the second term requires new ideas but makes use of several lemmas proved in \cite{HE}.

We remark that it was already observed in \cite[Section~4.4]{Smirnov} that the proof of \cite[Proposition~6.3]{HE} could be used to deduce Proposition \ref{prop_crossing} when $m \equiv 0$, that is to show that the harmonic explorer satisfies Condition \hyperref[condG3]{G.3}. However, in \cite[Section~4.4]{Smirnov}, the argument did not take into account a certain type of geometric configuration for $A_{\tau}^{\hat \Omega_{\delta}}$, corresponding to the collection $\mathcal{A}_2$ in our proof of Proposition \ref{prop_crossing}, thus making the proof incomplete. This gap is filled below.

Before turning to the proof of Proposition \ref{prop_crossing}, we recall \cite[Lemma~6.1]{HE} and its corollary \cite[Corollary~6.2]{HE} as we will repeatedly use them. In order to do so, we need to introduce the discrete excursion measures, which are the discrete analogues of the Brownian excursion measures. For $\delta>0$, let $G_{\delta} \subset \delta\mathbb{T}$ be a graph with boundary $\partial G_{\delta}$. For an oriented edge $(vw)$ of $\delta \mathbb{T}$, we denote by $\text{rev}(vw)$ the same edge with reverse orientation. We let $E=E(G_{\delta})$ denote the set of edges of $G_{\delta}$ whose inital vertex is in $\partial G_{\delta}$ and whose terminal vertex is in $\operatorname{Int}(G_{\delta})$. Let $E_1 \subset E$ and $E_2 \subset \text{rev}(E)$. For every $v \in \partial G_{\delta}$, let $X^v$ be a simple random walk on $\delta \mathbb{T}$ that starts at $v$ and is stopped at the first time $n \geq 1$ such that $X^v_n \notin G_{\delta}$. Let $\nu^v$ denote the restriction of the law of $X^v$ to those walks that use an edge of $E_1$ as their first step and an edge of $E_2$ as their last step. Finally, define
\begin{equation*}
    \nu_{(G_{\delta}, E_1, E_2)} := \sum_{v \in \partial G_{\delta}} \nu^{v}.
\end{equation*}
$\nu_{(G_{\delta}, E_1, E_2)}$ is called the discrete excursion measure from $E_1$ to $E_2$ in $G_{\delta}$: this is a measure on paths starting with an edge of $E_1$ and ending with an edge of $E_2$ and staying in $G_{\delta}$ in between. When $E_2 = \text{rev}(E)$, we simply write $\nu_{(G_{\delta},E_1,E_2)} = \nu_{(G_{\delta},E_1)}$. The first result about the measure $\nu_{(G_{\delta},E_1)}$ that will be instrumental in the proof of Proposition \ref{prop_crossing} is a relation between the expected number of visits to a vertex $x \in \operatorname{Int}(G_{\delta})$ under $\nu_{(G_{\delta},E_1)}$ and the probability that a walk started from $x$ exits $G_{\delta}$ using an edge of $\text{rev}(E_1)$. This is \cite[Lemma~6.1]{HE}.

\begin{lemma} \label{lemma_visits}
Let $G_{\delta}$ be as above and let $E_1 \subset E$. Fix $x \in \operatorname{Int}(G_{\delta})$ and for a path $\omega$, let $n_x(\omega)$ be the number of times $\omega$ visits $x$. Then
\begin{equation*}
    \int n_x(\omega) \nu_{(G_{\delta},E_1)}(d\omega) = H_{G_{\delta}}(x,\text{rev}(E_1))
\end{equation*}
where $H_{G_{\delta}}(x, \text{rev}(E_1))$ is the probability that a simple random walk started at $x$ will first exit $G_{\delta}$ via an edge in $\text{rev}(E_1)$. In particular, $\int n_x(\omega) \nu_{(G_{\delta},E)}(d\omega) = 1$.
\end{lemma}

This lemma can be used to estimate the $\nu_{(G_{\delta}, E_1)}$-mass of paths that visits a ball, provided the ball is sufficiently far away from $\partial G_{\delta}$. This is stated as \cite[Corollary~6.2]{HE} and since this result will be useful in the proof of Proposition \ref{prop_crossing}, let us recall it.

\begin{lemma} \label{lemma_massball}
Let $G_{\delta}$ and $E_1$ be as above and let $x \in \operatorname{Int}(G_{\delta})$. Denote by $\text{rad}_x(G_{\delta})$ the Euclidean distance between $x$ and the boundary of $G_{\delta}$. Assume that the boundary of $G_{\delta}$ is connected. Let $B$ be the ball centered at $x$ whose radius is $\frac{1}{2}\text{rad}_x(G_{\delta})$ and let $\Gamma_B$ be the set of paths that visit $B$. Then
\begin{equation*}
    c^{-1}H_{G_{\delta}}(x, \text{rev}(E_1)) < \nu_{(G_{\delta},E_1)}(\Gamma_B)< c H_{G_{\delta}}(x, \text{rev}(E_1))
\end{equation*}
for some absolute constant $c>0$.
\end{lemma}

The last fact that will be useful in the course of the proof of Proposition \ref{prop_crossing} is the following simple inequality between discrete massive and massless harmonic measures.

\begin{lemma} \label{lemma_mHmeasure}
    Let $G_{\delta} \subset \delta\mathbb{T}$ be a finite graph with boundary $\partial G_{\delta}$ and let $\tilde \partial$ be a subset of $\partial G_{\delta}$. For $w \in \operatorname{Int}(G_{\delta})$, denote by $H_{\delta}^{(0)}(w)$, respectively $H_{\delta}^{(m)}(w)$, the massless, respectively the massive, discrete harmonic measure of $\tilde \partial$ seen from $w$. Then, for any $w \in \operatorname{Int}(G_{\delta})$,
    \begin{equation*}
        H_{\delta}^{(m)}(w) \leq H_{\delta}^{(0)}(w).
    \end{equation*}
\end{lemma}

\begin{proof}
Let $w \in \operatorname{Int}(G_{\delta})$. By definition,
\begin{equation*}
    H_{\delta}^{(m)}(w) = \PP_{\delta, w}^{(m)}(X_{\tau_{\partial G_{\delta}}} \in \tilde \partial)
\end{equation*}
where under $\PP_{\delta,w}^{(m)}$, $X$ is a simple massive random walk on $\delta\mathbb{T}$ started at $w$ and $\tau_{\partial G_{\delta}}$ denotes its first hitting time of $\partial G_{\delta}$. Therefore, we have that
\begin{equation*}
     H_{\delta}^{(m)}(w) = \sum_{\omega: w \to \tilde \partial} \prod_{j=1}^{\tau(\omega)} \frac{1-m^{2}(\omega_j)\delta^2}{6}
\end{equation*}
where the sum is over the set of paths $\omega$ on $\mathbb{T}_{\delta}$ starting at $w$ and ending at a vertex of $\tilde \partial$. For such a path $\omega$, $\omega_{j}$ denotes the $j$th vertex it visits and $\tau(\omega)$ is its length. The above equality then yields that
\begin{align*}
    H_{\delta}^{(m)}(w) &\leq \sum_{\omega: w \to \tilde \partial} \prod_{j=1}^{\tau(\omega)} \frac{1}{6} \\
    &= \PP_{\delta, w}^{(0)}(X_{\tau_{\partial G_{\delta}}} \in \tilde \partial) \\
    &= H_{\delta}^{(0)}(w)
\end{align*}
where under $\PP_{\delta,w}^{(0)}$, $X$ is a simple (massless) random walk started at $w$.
\end{proof}

With these lemmas in hand, let us now turn to the proof of Proposition \ref{prop_crossing}. We first prove the following proposition, which is a special case of Proposition \ref{prop_crossing} when the stopping time $\tau$ is almost surely equal to $0$. Thanks to the Markov property \eqref{discrete_DMP} of the massive harmonic explorer, the proof of Proposition \ref{prop_crossing} will follow the same strategy as that of the proof of this proposition, and we find it easier to first explain the arguments for the time $\tau=0$ and then show how to adapt them to the case of a general stopping time.

\begin{proposition} \label{prop_crossing_0}
There exist constants $K, \alpha>0$ such that for any $0<\delta<\overline{m}_d^{-1}$ and any annulus $A=A(z_0,r,R)$,
\begin{equation} \label{crossing_0}
    \PP_{\delta}^{(\Omega,a,b,m)}(\gamma \, \text{makes a crossing of $A$ which is contained in $A^{\hat \Omega_{\delta}}$}) \leq K \bigg (\frac{r}{R}\bigg)^{\alpha}.
\end{equation}
\end{proposition}

\begin{proof}
Fix $0 < \delta < \overline{m}_d^{-1}$. For clarity, as $\delta$ is fixed, we write $A^{\Omega}$ for $A^{\hat \Omega_{\delta}}$. $A^{\Omega}$ is a collection $(A^{\Omega}_j)_j$ of connected components of $A(z_0,r,R) \cap \hat \Omega_{\delta}$. We are going to split it into two disjoint sub-collections $\mathcal{A}_1^{\Omega}$ and $\mathcal{A}_2^{\Omega}$ of connected components. These collections correspond to two different geometric configurations for the intersection between $\Omega_{\delta}$ and the annulus $A(z_0,r,R)$. We will then upper bound the probability of a crossing of a component of $\mathcal{A}_1^{\Omega}$ and that of a crossing of a component of $\mathcal{A}_2^{\Omega}$ separately. In both cases, the upper bound is established using a martingale argument. Indeed, we will see that thanks to the optional stopping theorem, upper bounding the probability of a crossing in $\mathcal{A}_1^{\Omega}$ or that of a crossing in $\mathcal{A}_2^{\Omega}$ amounts to upper bound a certain martingale at time $0$ and lower bound it at a well-chosen stopping time. However, because of the different geometric configurations reflected in the collections $\mathcal{A}_1^{\Omega}$ and $\mathcal{A}_2^{\Omega}$, we cannot use the same martingale in both cases and this is why we must distinguish between these two cases.

Let us now define the splitting of $A^{\Omega}$ into two disjoint sub-collections $\mathcal{A}_1^{\Omega}:=(A^{\Omega}_{1,j})_j$ and $\mathcal{A}^{\Omega}_2:=(A^{\Omega}_{2,j})_j$ as mentioned above. The collection $\mathcal{A}^{\Omega}_1$ is such the following holds. $A \subset A^{\Omega}$ is an element of $\mathcal{A}^{\Omega}_1$ if and only if $\gamma$ must first intersect $\partial B(z_0,R)$ to cross $A$. In turn, the collection $\mathcal{A}^{\Omega}_2$ is made of those components of $A^{\Omega}$ that do not satisfy this property. In other words, $A^{\Omega}_j$ belongs to $\mathcal{A}^{\Omega}_2$ if and only if $\gamma$ must first intersect $\partial B(z_0,r)$ to cross $A^{\Omega}_j$. See Figure \ref{fig_collection} for an illustration. Observe that $\mathcal{A}^{\Omega}_1$ or $\mathcal{A}^{\Omega}_2$ may be empty. Given this splitting, we define two collections of connected components of $\Omega \setminus A^{\Omega}$
\begin{align*}
    &B_1^{av} :=\big \{ D \subset B(z_0,r) \cap \Omega: \text{ $D$ connected such that } \exists A_{1,j}^{\Omega} \in \mathcal{A}_1^{\Omega} \, \text{such that } \partial D \cap \partial A_{1,j}^{\Omega} \neq \emptyset \big\} \\
    &B_2^{av} :=\big \{ D \subset \Omega \setminus B(z_0,R): \text{ $D$ connected such that } \exists A_{2,j}^{\Omega} \in \mathcal{A}_2^{\Omega} \, \text{such that } \partial D \cap \partial A_{2,j}^{\Omega} \neq \emptyset \big\}.
\end{align*}
Notice that $B_1^{av}$, respectively $B_2^{av}$, is chosen such that if $\gamma$ makes a crossing of $A_{1,j}^{\Omega}$, respectively of $A_{2,j}^{\Omega}$, for some $j$, then there exists $n \in \mathbb{N}$ such that $\gamma(n) \in B_{1}^{av}$, respectively $\gamma(n) \in B_{2}^{av}$.

\begin{figure}
    \centering\includegraphics[width=0.49\textwidth]{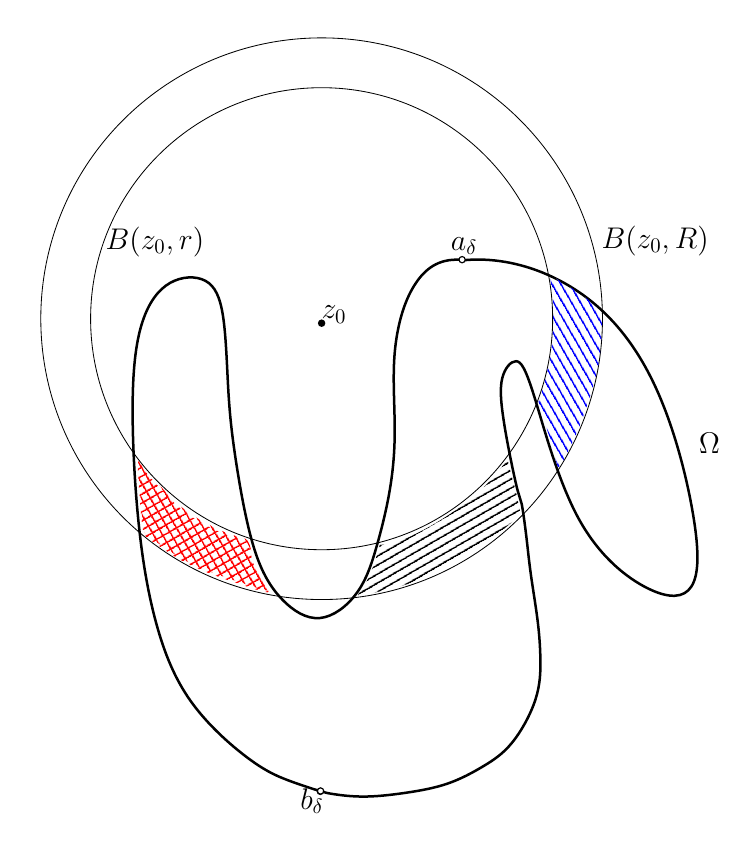}
  \caption{Example of a domain $\Omega$ together with an annulus $A(z_0,r,R)$. Here, the collection $\mathcal{A}_1$ consists of the connected component of $A(z_0,r,R) \cap \Omega$ shaded with red cross-hatch while the collection $\mathcal{A}_2$ consists of the connected component of $A(z_0,r,R) \cap \Omega$ with vertical blue lines. The connected component with horizontal black lines is an unavoidable component of $\Omega \cap A(z_0,r,R)$.}
  \label{fig_collection}
\end{figure}

Let $\mathcal{Q}$ be the event that there exists $j \in \mathbb{N}$ such that $v_j \in B_1^{av} \cup B_2^{av}$ and $\gamma([0,j])$ contains a crossing of $A^{\Omega}$. On $\mathcal{Q}$, denote by $\sigma$ the least such $j$. Observe that $\mathcal{Q}$ can be decomposed as a disjoint union $\mathcal{Q} = \mathcal{Q}_1 \sqcup \mathcal{Q}_2 $, where $\mathcal{Q}_1 := \mathcal{Q} \cap \{v_{\sigma} \in B_{1}^{av} \}$ and $\mathcal{Q}_2 := \mathcal{Q} \cap \{v_{\sigma} \in B_{2}^{av}\}$. Therefore, we see that to bound $\PP_{\delta}^{(\Omega,a,b,m)}(\mathcal{Q})$, it is enough to bound the probabilities $\PP_{\delta}^{(\Omega,a,b,m)}(\mathcal{Q}_1)$ and $\PP_{\delta}^{(\Omega,a,b,m)}(\mathcal{Q}_2)$ separately. We claim that the following bound on $\PP_{\delta}^{(\Omega,a,b,m)}(\mathcal{Q}_1)$ holds.

\begin{claim} \label{claim_Q1}
There exist universal constants $K_1, \alpha_1 > 0$ such that
\begin{equation} \label{bound_Q1}
    \PP_{\delta}^{(\Omega,a,b,m)}(\mathcal{Q}_1) \leq K_1\bigg( \frac{r}{R} \bigg)^{\alpha_1}.
\end{equation}
\end{claim}

As for $\PP_{\delta}^{(\Omega,a,b,m)}(\mathcal{Q}_2)$, we claim that it satisfies a similar bound.

\begin{claim} \label{claim_Q2}
There exist universal constants $K_2, \alpha_2 > 0$ such that
\begin{equation} \label{bound_Q2}
    \PP_{\delta}^{(\Omega,a,b,m)}(\mathcal{Q}_2) \leq K_2\bigg( \frac{r}{R} \bigg)^{\alpha_2}.
\end{equation}
\end{claim}

Claim \ref{claim_Q1} and Claim \ref{claim_Q2} together imply that the inequality \eqref{crossing_0} holds with $K=2\max(K_1,K_2)$ and $\alpha=\min(\alpha_1,\alpha_2)$, thus establishing Proposition \ref{prop_crossing_0}. Let us now turn to the proof of these claims. We start by showing Claim \ref{claim_Q1}.

\begin{proof}[Proof of Claim \ref{claim_Q1}]
Let us assume that $(z_0, r, R)$ is in the following subset of $\mathbb{C} \times \mathbb{R}_{+} \times \mathbb{R}_{+}$
\begin{equation*}
    \mathcal{G}_1:= \{ (z_0, r, R): \text{dist}(z_0,\partial \Omega) \leq \operatorname{diam}(\Omega), \, r \leq \frac{1}{100}\operatorname{diam}(\Omega), \, R \leq 2\operatorname{diam}(\Omega), \, 0< 10r \leq R \}.
\end{equation*}
In this case, we will see that there exist constants $\tilde K_1, c_1, \hat \alpha_1 >0$ such that
\begin{align} \label{crossing_good}
    \PP_{\delta}^{(\Omega,a,b,m)}(\mathcal{Q}_1) \leq \frac{2\tilde K_1}{c_1}\bigg(\frac{r}{R}\bigg)^{\hat \alpha_1}
    + \frac{2^{1+\hat \alpha_1 /2}\tilde K_1}{c_1}\bigg(\frac{r}{R} \bigg)^{\hat \alpha_1 /2}
    + 2\frac{\overline{m}_d^{2}\operatorname{diam}(\Omega)^2}{c_1}\frac{r}{R}.
\end{align}
Therefore, for $(z_0,r,R) \in \mathcal{G}_1$, the inequality \eqref{bound_Q1} is satisfied with
\begin{equation} \label{Kalpha_1}
    C_1:= K_1 = \frac{2\tilde K_1}{c_1} + \frac{2^{1+\hat \alpha_1 /2}\tilde K_1}{c_1} + \frac{2\overline{m}_d^2}{c_1}\operatorname{diam}(\Omega)^2, \quad \text{and} \quad  \alpha_1 = \frac{\hat \alpha_1}{2} \wedge 1.
\end{equation}
If $(z_0,r,R) \notin \mathcal{G}_1$, then notice that either there are no crossings of $A(z_0,r,R)$ that stay in $\Omega$ or the ratio $r/R$ is greater than or equal to a deterministic constant $c>0$. Indeed:
\begin{enumerate}
    \item if $\text{dist}(z_0,\partial \Omega) \geq (3/2)\operatorname{diam}(\Omega)$, then, in order to have a crossing of $A(z_0,r,R)$ that stay in $\Omega$, one must have $r\geq \text{dist}(z_0,\partial \Omega)$ and $R\leq \text{dist}(z_0,\partial \Omega) + \operatorname{diam}(\Omega)$. This implies that
    \begin{align*}
        \frac{R}{r} \leq \frac{\text{dist}(z_0,\partial \Omega)+\operatorname{diam}(\Omega)}{\text{dist}(z_0,\partial \Omega)} \leq 1+1
    \end{align*}
    Therefore, the ratio $r/R$ is lower bounded by $1/2$ in that case.
    \item If $10r \geq R$, then the ratio $r/R$ is lower bounded by $1/10$.
    \item If $\text{dist}(z_0, \partial \Omega) \leq \operatorname{diam}(\Omega)$ and $R \geq 2 \operatorname{diam}(\Omega)$, then there are no crossings of $A(z_0,r,R)$ that stay in $\Omega$.
    \item If $\text{dist}(z_0, \partial \Omega) \leq \operatorname{diam}(\Omega)$ and $r \geq (1/100)\operatorname{diam}(\Omega)$, then in order to have a crossing of $A(z_0,r,R)$, one must have $R \leq 2\operatorname{diam}(\Omega)$. This implies that the ratio $r/R$ is lower bounded by $1/200$.
\end{enumerate}
We thus see that if $(z_0,r,R) \notin \mathcal{G}_1$, then either there are no crossings of $A(z_0,r,R)$ that stay in $\Omega$ or the ratio $r/R$ is lower bounded by $1/200$. Taking $\alpha_1$ as in \eqref{Kalpha_1} and $C_{b1}>0$ such that $C_{b1}\times(1/200)^{\alpha_1}=1$, we trivially obtain that for $(z_0,r,R) \notin \mathcal{G}_1$,
\begin{equation*}
    \PP_{\delta}^{(\Omega,a,b,m)}(\mathcal{Q}_1) \leq C_{b1} \bigg( \frac{r}{R}\bigg)^{\alpha_1}
\end{equation*}
since the right-hand side is greater than $1$. Setting $K_1 = \max(C_1, C_{b1})$ where $C_1$ is the constant found in \eqref{Kalpha_1}, we then obtain the inequality \eqref{bound_Q1} for any $(z_0,r,R) \in \mathbb{C} \times \mathbb{R}_{+} \times \mathbb{R}_{+}$ with $r \leq R$.

Let us now turn to the proof of the inequality \eqref{crossing_good}. Let $(z_0,r,R)$ be in $\mathcal{G}_1$ and consider the annulus $A(z_0,r,R)$. By definition of the collection $A^{\Omega}$, each connected component of $B(z_0,3r) \cap \Omega_{\delta}$ which intersects some $A_{1,j}^{\Omega}$ for some $j$ has boundary entirely in $\partial \Omega_{\delta}^{-}$ or $\partial \Omega_{\delta}^{+}$. Recall that on the event $\mathcal{Q}_1$, the stopping time $\sigma$ is defined as the least $j$ such that $v_j \in B_1^{av}$ and $\gamma([0,j])$ contains a crossing of $A^{\Omega}$. On the event $\mathcal{Q}_1$, let $S$ be the connected component of $B(z,3r) \cap \Omega_{\delta}$ intersecting $\gamma([\sigma-1, \sigma])$ and let $\mathcal{Q}_{1}^{-} \subset \mathcal{Q}_1$, respectively $\mathcal{Q}_1^{+} \subset \mathcal{Q}_{1}$, be the event that $\partial S \subset \partial \Omega_{\delta}^{-}$, respectively $\partial S \subset \partial \Omega_{\delta}^{+}$. Then $\PP_{\delta}^{(\Omega,a,b,m)}(\mathcal{Q}_1) = \PP_{\delta}^{(\Omega,a,b,m)}(\mathcal{Q}_{1}^{-}) + \PP_{\delta}^{(\Omega,a,b,m)}(\mathcal{Q}_{1}^{+})$ and the inequality \eqref{crossing_good} will hold if we can show that there exist constants $\tilde K_1, c_1, \hat \alpha_1 > 0$ such that
\begin{equation*} \label{ineq_Q1}
    \PP_{\delta}^{(\Omega,a,b,m)}(\mathcal{Q}_{1}^{\pm}) \leq  \frac{\tilde K_1}{c_1}\bigg(\frac{r}{R}\bigg)^{\hat \alpha_1}+ \frac{2^{\hat \alpha_1}\tilde K_1}{c_1}\bigg(\frac{r}{R} \bigg)^{\hat \alpha_1 /2} + \frac{\overline{m}_d^{2}\operatorname{diam}(\Omega)^2}{c_1}\frac{r}{R}.
\end{equation*}
Let us prove this inequality for $\PP_{\delta}^{(\Omega,a,b,m)}(\mathcal{Q}_{1}^{-})$. By symmetry, the proof for $\PP_{\delta}^{(\Omega,a,b,m)}(\mathcal{Q}_{1}^{+})$ is virtually the same. Let $E_{-}$ denote the set of directed edges in $E=E(\Omega_{\delta})$ whose initial vertex is in $\partial \Omega_{\delta}^{-} \cap B(z_0,3r)$ and is disconnected from $b$ in $\Omega_{\delta}$ by a connected component of $\mathcal{A}_{1}^{\Omega}$. Denote by $V_{-}$ the set of initial vertices of the edges in $E_{-}$. Then, by Lemma \ref{prop_martingale_obs}, the process
\begin{equation*}
    \| \tilde \nu_{\delta, n}^{(m)} \| := \sum_{v \in V_{-}} \frac{1}{6} \sum_{v \sim w} \bigg( h_{\delta, n}^{m}(w) + \frac{1}{2} \bigg), \quad n \geq 0,
\end{equation*}
is a non-negative martingale for the filtration $(\mathcal{F}_{\delta, n})_n$. Therefore, by the optional stopping theorem,
\begin{equation} \label{OST_nu}
    \| \tilde \nu_{\delta, 0}^{(m)} \| = \EE_{\delta}^{(\Omega,a,b,m)}\big[\| \tilde \nu_{\delta, \sigma}^{(m)} \| \big] \geq \EE_{\delta}^{(\Omega,a,b,m)}\big[\mathbb{I}_{\mathcal{Q}_{1}^{-}}\| \tilde \nu_{\delta, \sigma}^{(m)} \|\big].
\end{equation}
We thus see that in order to bound $\PP_{\delta}^{(\Omega,a,b,m)}(\mathcal{Q}_{1}^{-})$, it is enough to exhibit an appropriate upper bound for $ \| \tilde \nu_{\delta, 0}^{(m)} \|$ and to show that $\| \tilde \nu_{\delta, \sigma}^{(m)} \|$ is bounded away from $0$ by a universal constant on the event $\mathcal{Q}_{1}^{-}$. Let us first focus on $ \| \tilde \nu_{\delta, 0}^{(m)} \|$.

Observe that almost surely, for any $n \in \mathbb{N}$ and any $w \in \operatorname{Int}(\Omega_{\delta, n}) \cup \partial \Omega_{\delta,n}$,
\begin{equation*}
    h_{\delta, n}^{m}(w) + \frac{1}{2} = H_{\delta, n}^{(m)}(w) + \frac{1}{2}\PP_{w}^{(m)}(\tau^{\star} \leq \tau_{\partial \Omega_{\delta, n}})
\end{equation*}
where, as explained in Section \ref{subsec:def_mHE}, $H_{\delta, n}^{(m)}(w)$ is the discrete massive harmonic measure of $\partial \Omega_{\delta, n}^{+}$ seen from $w$. By Lemma \ref{lemma_mHmeasure}, we then have that, almost surely, for any $n \in \mathbb{N}$,
\begin{equation*}
    \| \tilde \nu_{\delta, n}^{(m)} \| \leq \sum_{v \in V_{-}} \frac{1}{6} \sum_{v \sim w} \bigg( H_{\delta,n}^{(0)}(w) + \frac{1}{2} \PP_{w}^{(\overline{m}_{d})}(\tau^{\star} \leq \tau_{\partial \Omega_{\delta, n}}) \bigg),
\end{equation*}
where $H_{\delta, n}^{(0)}(w)$ is the discrete (massless) harmonic measure of $\partial \Omega_{\delta, n}^{+}$ seen from $w$. Taking $n=0$, this yields that
\begin{equation*}
    \| \tilde \nu_{\delta, 0}^{(m)} \| \leq \sum_{v \in V_{-}} \frac{1}{6} \sum_{v \sim w} \bigg( H_{\delta,0}^{(0)}(w) + \frac{1}{2} \PP_{w}^{(\overline{m}_{d})}(\tau^{\star} \leq \tau_{\partial \Omega_{\delta, 0}}) \bigg).
\end{equation*}
Our upper bound on $\| \tilde \nu_{\delta, 0}^{(m)} \|$ that will allow us to establish the inequality \eqref{crossing_good} is a consequence of the following two lemmas.

\begin{lemma} \label{lemma_upper_Hq1}
There exist universal constants $\tilde K_1 , \hat \alpha_1 > 0$ such that, for any $(z_0,r,R) \in \mathcal{G}_1$,
\begin{equation} \label{harmonic_part}
    \sum_{v \in V_{-}} \frac{1}{6} \sum_{v \sim w} H_{\delta,0}^{(0)}(w) \leq \tilde K_1 \bigg( \frac{r}{R}\bigg)^{\hat \alpha_1}.
\end{equation}
\end{lemma}

Lemma \ref{lemma_upper_Hq1} can be derived using the same arguments as those used in the first part of the proof of \cite[Proposition~6.3]{HE}, since we assume that $10r \leq R$ for $(z_0,r,R) \in \mathcal{G}_1$. The next lemma, which controls the term in $\| \tilde \nu_{\delta, 0}^{(m)} \|$ arising from the killing, will require a bit more work.

\begin{lemma} \label{lemma_upper_killing}
For the same constants $\tilde K_1$ and $\hat \alpha_1$ as in Lemma \ref{lemma_upper_Hq1} and still assuming that $(z_0,r,R) \in \mathcal{G}_1$,
\begin{equation*}
     \sum_{v \in V_{-}} \frac{1}{6} \sum_{v \sim w} \frac{1}{2} \PP_{w}^{(\overline{m}_{d})}(\tau^{\star} \leq \tau_{\partial \Omega_{\delta, 0}}) \leq 2^{\hat \alpha_1 /2} \tilde K_1 \bigg(\frac{r}{R}\bigg)^{\hat \alpha_1 /2} + \overline{m}_d^2\operatorname{diam}(\Omega)^2\times \frac{r}{R}.
\end{equation*}
\end{lemma}

We postpone the proof of Lemma \ref{lemma_upper_killing} to the end and show how to proceed from here. Observe that combined together, Lemma \ref{lemma_upper_Hq1} and Lemma \ref{lemma_upper_killing} yield that
\begin{equation} \label{bound_nu0}
     \| \tilde \nu_{\delta, 0}^{(m)} \| \leq \tilde K_{1} \bigg(\frac{r}{R}\bigg)^{\hat \alpha_1} + 2^{\hat \alpha_1 /2} \tilde K_1 \bigg(\frac{r}{R}\bigg)^{\hat \alpha_1 /2} + \frac{1}{2} \times 2\overline{m}_d^2\operatorname{diam}(\Omega)^2 \frac{r}{R}
\end{equation}
where $\tilde K_1 , \hat \alpha_1 > 0$ are universal constants.

Let us now exhibit a lower bound for $\| \tilde \nu_{\delta, \sigma}^{(m)} \|$ on the event $\mathcal{Q}_{1}^{-}$. For $n \geq 0$ and $w \in \operatorname{Int}(\Omega_{\delta, n})$, denote by $\tilde H_{\delta, n}^{(m)}(w)$, respectively $\tilde H_{\delta, n}^{(0)}(w)$, the discrete massive, respectively massless, harmonic measure of $\partial \Omega_{\delta, n}^{-}$ seen from $w$. Almost surely, it holds that, for any $n \in \mathbb{N}$ and $w \in \operatorname{Int}(\Omega_{\delta,n})$,
\begin{equation*}
    H_{\delta, n}^{(m)}(w) + \tilde H_{\delta, n}^{(m)}(w) + \PP_{w}^{(m)}(\tau^{\star} \leq \tau_{\partial \Omega_{\delta,n}}) = H_{\delta, n}^{(0)}(w) + \tilde H_{\delta, n}^{(0)}(w).
\end{equation*}
It follows that almost surely, for any $n \in \mathbb{N}$ and $w \in \operatorname{Int}(\Omega_{\delta,n})$,
\begin{equation} \label{ineq_Hmtau}
    H_{\delta, n}^{(m)}(w) + \PP_{w}^{(m)}(\tau^{\star} \leq \tau_{\partial \Omega_{\delta,n}}) \geq H_{\delta, n}^{(0)}(w).
\end{equation}
since almost surely, for any $n \in \mathbb{N}$ and $w \in \Omega_{\delta, n}$, $\tilde H_{\delta,n}^{(0)}(w) - \tilde H_{\delta,n}^{(m)}(w) \geq 0$. Therefore, almost surely,
\begin{align} \label{ineq_mart_sigma}
   \| \tilde \nu_{\delta, \sigma}^{(m)} \| &= \sum_{v \in V_{-}} \frac{1}{6} \sum_{v \sim w} \bigg(H_{\delta, \sigma}^{(m)}(w) + \frac{1}{2} \PP_{w}^{(m)}(\tau^{\star} \leq \tau_{\partial \Omega_{\delta,\sigma}}) \bigg) \nonumber \\
   &\geq \sum_{v \in V_{-}} \frac{1}{6} \sum_{v \sim w} \bigg( \frac{1}{2} H_{\delta, \sigma}^{(m)}(w) + \frac{1}{2} \PP_{w}^{(m)}(\tau^{\star} \leq \tau_{\partial \Omega_{\delta,\sigma}}) \bigg) \nonumber \\
    &\geq \sum_{v \in V_{-}} \frac{1}{6} \sum_{v \sim w} \frac{1}{2} H_{\delta, \sigma}^{(0)}(w)
\end{align}
where the inequality \eqref{ineq_mart_sigma} follows from the inequality \eqref{ineq_Hmtau} by multiplying both sides by $1/2$. The first inequality simply uses the fact that for any $n \in \mathbb{N}$ and any $w \in \operatorname{Int}(\Omega_{\delta,n})$, $H_{\delta,n}^{(m)}(w)$ is non-negative. The second part of the proof of \cite[Proposition~6.3]{HE} shows that
\begin{equation*}
    \EE_{\delta}^{(\Omega,a,b,m)} \bigg[ \mathbb{I}_{\mathcal{Q}_{1}^{-}}\bigg(\sum_{v \in V_{-}} \frac{1}{6} \sum_{v \sim w} \frac{1}{2} H_{\sigma, \delta}^{(0)}(w) \bigg) \bigg] \geq c_1\PP_{\delta}^{(\Omega,a,b,m)}(\mathcal{Q}_{1}^{-})
\end{equation*}
where $c_1>0$ is a universal constant. This inequality together with \eqref{bound_nu0} and the optional stopping theorem argument explained in \eqref{OST_nu} yield that
\begin{equation*}
    c_1 \PP_{\delta}^{(\Omega,a,b,m)}(\mathcal{Q}_{1}^{-}) \leq \tilde K_{1} \bigg(\frac{r}{R}\bigg)^{\hat \alpha_1} + 2^{\hat \alpha_1 /2} \tilde K_1 \bigg(\frac{r}{R}\bigg)^{\hat \alpha_1 /2} + \overline{m}_d^2\operatorname{diam}(\Omega)^2 \frac{r}{R},
\end{equation*}
which, as explained above, implies the inequality \eqref{crossing_good}.
\end{proof}

To complete the proof of Claim \ref{claim_Q1}, we must prove the auxiliary lemma that we used along the way.

\begin{proof}[Proof of Lemma \ref{lemma_upper_killing}]
Using the same notations as in the proof of Claim \ref{claim_Q1}, we want to upper bound the quantity
\begin{equation*}
    \| K_{\delta, 0}^{(\overline{m}_{d})} \| := \sum_{v \in V_{-}} \frac{1}{6} \sum_{v \sim w} \PP_{w}^{(\overline{m}_d)}(\tau^{\star} \leq \tau_{\partial \Omega_{\delta}}).
\end{equation*}
Let us express this quantity in terms of the integral of a functional with respect to the excursion measure $\nu_{(\Omega_{\delta}, E_{-}, E_{\delta})}$, where $E_{\delta}$ is the set of edges in $\Omega_{\delta}$ whose endpoint is in $\partial \Omega_{\delta}$. We have that
\begin{align*}
    \| K_{\delta, 0}^{(\overline{m}_d)} \| &= \sum_{v \in V_{-}} \frac{1}{6} \sum_{v \sim w} \EE_{w}^{(0)} \bigg[ \overline{m}_d^{2}\delta^2 \sum_{k=0}^{\tau_{\partial \Omega_{\delta}}} (1-\overline{m}_d^2\delta^2)^k \bigg] \\
    &= \sum_{v \in V_{-}} \frac{1}{6} \sum_{v \sim w} \sum_{\omega: w \to \partial \Omega_{\delta}} \overline{m}_d^{2}\delta^2 \bigg[ \sum_{k=0}^{\vert \omega \vert } (1-\overline{m}_d^2\delta^2)^k \bigg] \PP_{\delta}^{(0)}(\omega) 
\end{align*}
where $\PP_{\delta}^{(0)}(\omega)$ denotes the probability that a (non-massive) random walk on $\delta \mathbb{T}$ traces the path $\omega$, $\EE_w^{(0)}$ denotes the expectation with respect to (non-massive) random walk started at $w$ and where, for $w \in \Omega_{\delta}$, we write $\omega: w \to \partial \Omega_{\delta}$ to indicate that $\omega$ is a path from $w$ to $\partial \Omega_{\delta}$ in $\Omega_{\delta}$. We continue with this expansion on paths to obtain that
\begin{align*}    
    \| K_{\delta, 0}^{(\overline{m}_d)} \| &= \sum_{v \in V_{-}} \sum_{\omega: v \to \partial \Omega_{\delta}} \overline{m}_d^{2}\delta^2 \bigg[ \sum_{k=1}^{\vert \omega \vert } (1-\overline{m}_d^2\delta^2)^k \bigg] \frac{1}{6}\PP_{\delta}^{(0)}(\omega_{\vert \geq 1}) \\
    &= \sum_{\omega: E_{-} \to \partial \Omega_{\delta}} \overline{m}_d^{2}\delta^2 \bigg[ \sum_{k=1}^{\vert \omega \vert } (1-\overline{m}_d^2\delta^2)^k \bigg] \PP_{\delta}^{(0)}(\omega) \\
    &= \int \overline{m}_d^{2}\delta^2 \bigg[ \sum_{k=1}^{\vert \omega \vert } (1-\overline{m}_d^2\delta^2)^k \bigg] d\nu_{(\Omega_{\delta}, E_{-}, E_{\delta})}(\omega).
\end{align*}
where $\omega_{\vert \geq 1}$ denotes the path $\omega$ minus its first edge. The above representation of $\| K_{\delta, 0}^{(\overline{m}_d)} \|$ is useful as the discrete excursion measure $\nu_{(\Omega_{\delta}, E_{-}, E_{\delta})}$ is well-understood. To obtain a bound on this integral with respect to this excursion measure, we are going to split the set of excursions into two disjoint sets: the set of excursions that remain in a well-chosen ball $B(z_0, \tilde r)$ and the set of excursions which exit this ball. The radius $\tilde r$ of this ball is going to be chosen such that the total mass of excursions that exit in $B(z_0, \tilde r)$ can be well-controlled while excursions that stay in $B(z_0, \tilde r)$ are not long enough to have a macroscopic probability to be killed. To find the appropriate radius $\tilde r$, it is more convenient to first rescale $\Omega$. So, let us now rescale $\Omega$, and thus $\Omega_{\delta}$, by $(r\operatorname{diam}(\Omega))^{-1}$ and denote by $\Omega_{\delta}(r)$ the rescaled version of $\Omega_{\delta}$. $\Omega_{\delta}(r)$ is a piece of the triangular lattice with meshsize $\tilde \delta := \delta(r\operatorname{diam}(\Omega))^{-1}$. As we want the killing probabilities to agree on $\Omega_{\delta}$ and $\Omega_{\delta}(r)$, we must choose the mass $\tilde m^2 = \tilde m^2(r)$ on $\Omega_{\delta}(r)$ such that
\begin{equation*}
    \tilde m^2 \frac{\delta^2}{r^2\operatorname{diam}(\Omega)^2} = \overline{m}_d^2 \delta^2, \quad \text{that is} \quad \tilde m^2= \overline{m}_d^2r^2\operatorname{diam}(\Omega)^2.
\end{equation*}
Denote by $E_{-}(r)$, respectively $E_{\delta}(r)$, the image of $E_{-}$, respectively $E_{\delta}$, after the rescaling. Notice that for any path $\omega$ starting in $E_{-}$ and ending in $E_{\delta}(r)$, $\nu_{(\Omega_{\delta}, E_{-}, E_{\delta})}(\omega) = \nu_{(\Omega_{\delta}(r), E_{-}(r), E_{\delta}(r))}(\omega_r)$, where $\omega_r$  is the rescaled version of $\omega$. Therefore, we have that
\begin{equation*}
    \| K_{\delta, 0}^{(\overline{m}_d)} \| = \int \tilde {m}^{2} \tilde \delta^2 \bigg[ \sum_{k=1}^{\vert \omega \vert } (1-\tilde{m}^2 \tilde \delta^2)^k \bigg] d\nu_{(\Omega_{\delta}(r), E_{-}(r), E_{\delta}(r))}(\omega).
\end{equation*}
We now use a kind of restriction property of $\nu_{(\Omega_{\delta}(r), \text{rev}(E_{\delta}(r)), E_{\delta}(r))}$ to write
\begin{align*}
    \| K_{\delta, 0}^{(\overline{m}_d)} \| = \int_{(\omega_{0}\omega_{1}) \in E_{-}(r), \, e(\omega) \in E_{\delta}(r)} \tilde {m}^{2} \tilde \delta^2 \bigg[ \sum_{k=1}^{\vert \omega \vert } (1-\tilde{m}^2 \tilde \delta^2)^k \bigg] d\nu_{(\Omega_{\delta}(r), \text{rev}(E_{\delta}(r)), E_{\delta}(r))}(\omega)
\end{align*}
where for a path $\omega$, $(\omega_0\omega_1)$ is the first edge traversed by $\omega$ and $e(\omega)$ is the last one. Denote by $z_0'$ the image of $z_0$ by the rescaling $(r\operatorname{diam}(\Omega))^{-1}$ and define the following sets of paths
\begin{align*}
    &\mathcal{P}_{ext}(r) = \bigg \{\omega: (\omega_0\omega_1) \in E_{-}(r), \, e(\omega) \in E_{\delta}(r), \, \exists k \, \text{such that } \omega_{k} \notin B\big(z_0', \frac{1}{\sqrt{r\operatorname{diam}(\Omega)}} \big) \bigg \} \\
    &\mathcal{P}_{in}(r) = \bigg \{\omega: (\omega_0\omega_1) \in E_{-}(r), \, e(\omega) \in E_{\delta}(r), \, \forall k \, \omega_{k} \in B \big(z_0', \frac{1}{\sqrt{r\operatorname{diam}(\Omega)}} \big) \bigg \}.
\end{align*}
We then have that
\begin{align} \label{ineq_K0}
    \| K_{\delta, 0}^{(\overline{m}_d)} \| &= \tilde m^2 \tilde \delta^2 \int_{\mathcal{P}_{ext}(r)} \bigg[ \sum_{k=1}^{\vert \omega \vert } (1-\tilde{m}^2 \tilde \delta^2)^k \bigg] d\nu_{(\Omega_{\delta}(r), \text{rev}(E_{\delta}(r)), E_{\delta}(r))}(\omega) \nonumber \\
    &+ \tilde m^2 \tilde \delta^2 \int_{\mathcal{P}_{in}(r)} \bigg[ \sum_{k=1}^{\vert \omega \vert } (1-\tilde{m}^2 \tilde \delta^2)^k \bigg] d\nu_{(\Omega_{\delta}(r), \text{rev}(E_{\delta}(r)), E_{\delta}(r)})(\omega) \nonumber \\
    &\leq \nu_{(\Omega_{\delta}(r), \text{rev}(E_{\delta}(r)), E_{\delta}(r))}\big( \mathcal{P}_{ext}(r) \big) + \tilde m^2 \tilde \delta^2 \int_{\mathcal{P}_{in}(r)} \vert \omega \vert d\nu_{(\Omega_{\delta}(r), \text{rev}(E_{\delta}(r)), E_{\delta}(r))}(\omega).
\end{align}
Using this upper bound, we are going to show that,
\begin{equation*}
    \| K_{\delta, 0}^{\overline{m}_d} \| \leq 2^{\hat \alpha_1 /2} \tilde K_1 \bigg( \frac{r}{R}\bigg)^{\hat \alpha_1 /2} + 2\overline{m}_d^2\operatorname{diam}(\Omega)^2 \frac{r}{R},
\end{equation*}
where $\tilde K_1$ and $\hat \alpha_1$ are as in \eqref{harmonic_part}. First, since for $(z_0,r,R) \in \mathcal{G}_1$, $10/\operatorname{diam}(\Omega) \leq 1/\sqrt{r\operatorname{diam}(\Omega)}$ , the same arguments as those used in the first part of the proof of \cite[Proposition~6.3]{HE} show that
\begin{equation} \label{bound_ext}
    \nu_{(\Omega_{\delta}(r), \text{rev}(E_{\delta}(r)), E_{\delta}(r))}\big( \mathcal{P}_{ext}(r) \big) \leq \tilde K_1 \bigg( \frac{1/\operatorname{diam}(\Omega)}{1/\sqrt{r\operatorname{diam}(\Omega)}}\bigg)^{\hat \alpha_1} = \tilde K _1 \bigg( \frac{r}{\operatorname{diam}(\Omega)}\bigg)^{\hat \alpha_1 /2},
\end{equation}
where $\tilde K_1$ and $\hat \alpha_1$ are as in \eqref{harmonic_part}. Since for $(z_0,r,R) \in \mathcal{G}_1$, the ratio $r/\operatorname{diam}(\Omega)$ is upper bounded by $2r/R$, it thus only remains to upper bound the second term in the sum on the right-hand side of \eqref{ineq_K0}. For convenience, let us set $B_{\delta}(r) = B(z, (r\operatorname{diam}(\Omega)^{-1/2})) \cap \delta \mathbb{T}$. We first use the Fubini-Tonelli theorem to write
\begin{align*}
    \tilde m^2 \tilde \delta^2 \int_{\mathcal{P}_{in}(r)} \vert \omega \vert d\nu_{(\Omega_{\delta}(r), \text{rev}(E_{\delta}(r)), E_{\delta})(r)}(\omega) &= \tilde m^2 \tilde \delta^2 \int_{\mathcal{P}_{in}(r)} \bigg[ \sum_{x \in B_{\delta}(r)} n_x(\omega) \bigg] d\nu_{(\Omega_{\delta}(r), \text{rev}(E_{\delta}(r)), E_{\delta}(r))}(\omega) \\
    &= \tilde m^2 \tilde \delta^2 \sum_{x \in B_{\delta}(r)} \int_{\mathcal{P}_{in}(r)} n_x(\omega) d\nu_{(\Omega_{\delta}(r), \text{rev}(E_{\delta}(r)), E_{\delta}(r))}(\omega).
\end{align*}
By Lemma \ref{lemma_visits}, for any $x \in \Omega_{\delta}$,
\begin{equation*}
    \int n_x(\omega) d\nu_{(\Omega_{\delta}(r), \text{rev}(E_{\delta}(r)), E_{\delta})(r)}(\omega) = 1.
\end{equation*}
This yields that
\begin{align*}
    \tilde m^2 \tilde \delta^2 \int_{\mathcal{P}_{in}(r)} \vert \omega \vert d\nu_{(\Omega_{\delta}(r), \text{rev}(E_{\delta}(r)), E_{\delta})(r)}(\omega) &\leq \tilde m^2 \tilde \delta^2 \sum_{x \in B_{\delta}(r)} 1 \\
    &= \frac{\tilde m^2}{r\operatorname{diam}(\Omega)} \tilde \delta^2 \tilde \delta ^{-2}.
\end{align*}
Using that $\tilde m^2 = \overline{m}_d^2 r^2 \operatorname{diam}(\Omega)^2$ and the upper bound \eqref{bound_ext}, we thus obtain that
\begin{equation*}
    \| K_{\delta, 0}^{(\overline{m}_d)} \| \leq \tilde K_1 \bigg(\frac{r}{\operatorname{diam}(\Omega)}\bigg)^{\hat \alpha_1/2} + \overline{m}_d^2\operatorname{diam}(\Omega)r.
\end{equation*}
We observe that for $(z_0,r,R) \in \mathcal{G}_1$, the ratio $r/R$ is lower bounded by $r/2\operatorname{diam}(\Omega)$. Hence, it follows from the inequality above that
\begin{align*}
    \| K_{0}^{(\overline{m}_d)} \| &\leq \tilde K_1 \bigg(\frac{r}{\operatorname{diam}(\Omega)}\bigg)^{\hat \alpha_1 /2} + 2\overline{m}_d^2\operatorname{diam}(\Omega)^2\times \frac{r}{2\operatorname{diam}(\Omega)} \\
    &\leq 2^{\hat \alpha_1 /2} \tilde K_1 \bigg(\frac{r}{R}\bigg)^{\hat \alpha_1 /2} + 2\overline{m}_2^2\operatorname{diam}(\Omega)^2\times \frac{r}{R},
\end{align*}
which, after multiplying both sides by $1/2$, is exactly the claim of Lemma \ref{lemma_upper_killing}.
\end{proof}

We now turn to the proof of Claim \ref{claim_Q2}.

\begin{proof}[Proof of Claim \ref{claim_Q2}]
Let us assume that $(z_0, r, R)$ is in the following subset of $\mathbb{C} \times \mathbb{R}_{+} \times \mathbb{R}_{+}$
\begin{equation*}
    \mathcal{G}_2:= \{ (z_0, r, R): \text{dist}(z_0,\partial \Omega) \leq \operatorname{diam}(\Omega), \, R \leq 2\operatorname{diam}(\Omega), \, 0< 10^4r \leq R \}.
\end{equation*}
In this case, we will see that there exist constants $\tilde K_2, c_2, \hat \alpha_2 >0$ such that
\begin{equation} \label{crossing_good2}
    \PP_{\delta}^{(\Omega,a,b,m)}(\mathcal{Q}_2) \leq \frac{2\tilde K_2}{c_2}\bigg(\frac{r}{R}\bigg)^{\hat \alpha_2}
    + \frac{2\tilde K_2}{c_2}\bigg(\frac{r}{R} \bigg)^{\hat \alpha_2 /4}
    + 2\frac{\overline{m}_d^{2}\operatorname{diam}(\Omega)^2}{c_2}\sqrt{\frac{r}{R}}.
\end{equation}
Therefore, for $(z_0,r,R) \in \mathcal{G}_2$, the inequality \eqref{bound_Q2} is satisfied with
\begin{equation} \label{Kalpha_G}
    C_2:= K_2 = \frac{4\tilde K_2}{c_2} + \frac{4\overline{m}_d^2}{c_2}\operatorname{diam}(\Omega)^2, \quad \text{and} \quad \alpha_2 = \frac{\hat \alpha_2}{4} \wedge \frac{1}{2}.
\end{equation}
If $(z_0,r,R) \notin \mathcal{G}_2$, then notice that, as explained above for the set $\mathcal{G}_1$, either there are no crossings of $A^{\Omega}$ that stay in $\Omega$ or the ratio $r/R$ is lower bounded by $10^{-4}$. Taking $\alpha_2$ as in \eqref{Kalpha_G} and $C_{b2}>0$ such that $C_{b2}\times (10^{-4})^{\alpha_2}=1$, we trivially obtain that for $(z_0,r,R) \notin \mathcal{G}_2$,
\begin{equation*}
    \PP_{\delta}^{(\Omega,a,b,m)}(\mathcal{Q}_2) \leq C_{b2} \bigg( \frac{r}{R}\bigg)^{\alpha_2}
\end{equation*}
since the right-hand side is greater than $1$. Setting $K_2 = \max(C_2, C_{b2})$ where $C_2$ is the constant found in \eqref{Kalpha_G}, we then obtain the inequality \eqref{bound_Q2} for any $(z_0,r,R) \in \mathbb{C} \times \mathbb{R}_{+} \times \mathbb{R}_{+}$ with $r \leq R$.

Let us now turn to the proof of the inequality \eqref{crossing_good2}. Let $(z_0,r,R)$ be in $\mathcal{G}_2$ and consider the annulus $A(z_0,r,R)$. By definition of $A^{\Omega}$, the boundary arcs of the connected components in $A_2^{\Omega}$ that are also arcs of $\partial \Omega_{\delta}$ are either all contained in $\partial \Omega_{\delta}^{-}$ or all contained in $\partial \Omega_{\delta}^{+}$. Let us define the martingale that will plays the role of the martingale $(\| \tilde \nu^{(m)}_{\delta, n}\|, n \geq 0)$ that we used to estimate $\PP_{\delta}^{(\Omega,a,b,m)}(\mathcal{Q}_1)$. We start by rescaling $\delta \mathbb{T}$ by $1/R$ using the map $f_R: z \mapsto z/R$. We denote by $z_0'$ the image of $z_0$ by $f_R$. The image of the ball $B(z_0,r)$, respectively $B(z_0,R)$, under $f_R$ is $B(z_0',r/R)$, respectively $B(z_0',1)$. We also denote by $\tilde \Omega_{\delta}$ the image of $\Omega_{\delta}$ by $f_R$; this is a piece of the triangular lattice with meshsize $\delta_{a} := \delta R^{-1}$. For the killing probabilities to agree on $\tilde \Omega_{\delta}$ and $\Omega_{\delta}$, we must choose the mass $m_{a}^{2}$ on $\tilde \Omega_{\delta}$ such that
\begin{equation*}
    m_a^2 \frac{\delta^2}{R^2} = \overline{m}_d^2 \delta^2, \quad \text{that is} \quad m_a^2= \overline{m}_d^2R^2.
\end{equation*}
Observe that we have the inclusions
\begin{equation*}
    B\big(z_0', \frac{r}{R} \big) \subset B\big(z_0', \frac{1}{3}\sqrt{\frac{r}{R}}\big) \subset B \big(z_0', \sqrt{\frac{r}{R}} \big) \subset B(z_0',1)
\end{equation*}
where the first inclusion holds because for $(z_0,r,R) \in \mathcal{G}_2$, $10r \leq R$. Moreover, the boundary of a connected component $D$ of $A(z_0', (1/3)\sqrt{r/R}, \sqrt{r/R}) \cap f_R(A_2^{\Omega})$ is made of four arcs. One of them is an arc of $\partial B(z_0', (1/3)\sqrt{r/R})$ and its opposite boundary arc is an arc of $\partial B(z_0', \sqrt{r/R})$. The two other boundary arcs, denoted $b_1(D)$ and $b_2(D)$, that are opposite to each other are either two boundary arcs of $\partial \tilde \Omega_{\delta}^{-}$ or of $\partial \tilde \Omega_{\delta}^{+}$. On the event $\mathcal{Q}_{2}$, let $S$ be the connected component of $A(z_0', (1/3)\sqrt{r/R}), \sqrt{r/R}) \cap f_R(A_{2}^{\Omega})$ crossed by $\tilde \gamma_{\delta}([0,\sigma])$, where $\tilde \gamma_{\delta}$ denotes the rescaled version of $\gamma_{\delta}$. Note that the rescaling $f_R$ does not affect the value of $\sigma$. Let $\mathcal{Q}_{2}^{-}$, respectively $\mathcal{Q}_{2}^{+}$, be the event that $b_1(S), b_2(S) \subset \partial \tilde \Omega_{\delta}^{-}$, respectively $b_1(S), b_2(S) \subset \partial \tilde \Omega_{\delta}^{+}$. Notice that we have
\begin{equation*}
    \PP_{\delta}^{(\Omega,a,b,m)}(\mathcal{Q}_{2}) = \PP_{\delta}^{(\Omega,a,b,m)}(f_R^{-1}(\mathcal{Q}_{2}^{-})) + \PP_{\delta}^{(\Omega,a,b,m)}(f_R^{-1}(\mathcal{Q}_{2}^{+}))
\end{equation*}
and the inequality \eqref{crossing_good2} will hold if we can show that there exist constants $\tilde K_2, c_2, \hat \alpha_2 >0$ such that
\begin{equation*}
    \PP_{\delta}^{(\Omega,a,b,m)}(f_R^{-1}(\mathcal{Q}_{2}^{\pm})) \leq  \frac{\tilde K_2}{c_2}\bigg(\frac{r}{R}\bigg)^{\hat \alpha_2/2}+ \frac{\tilde K_2}{c_2}\bigg(\frac{r}{R} \bigg)^{\hat \alpha_2/4} + 2\frac{\overline{m}_d^{2}\operatorname{diam}(\Omega)^2}{c_2}\bigg(\frac{r}{R}\bigg)^{1/2}.
\end{equation*}
Let us prove this inequality for $\PP_{\delta}^{(\Omega,a,b,m)}(f_{R}^{-1}(\mathcal{Q}_{2}^{-}))$. By symmetry, the proof for $\PP_{\delta}^{(\Omega,a,b,m)}(f_R^{-1}(\mathcal{Q}_{2}^{+}))$ is virtually the same. Let $\tilde E_{-}$ denote the set of directed edges in $\tilde E = \tilde E(\tilde \Omega_{\delta})$ whose initial vertex is in $\partial \tilde \Omega_{\delta}^{-} \cap A(z_0', (1/3)\sqrt{r/R}, \sqrt{r/R})$ and on the boundary of $f_R(A_2^{\Omega})$. Denote $\tilde V_{-}$ the set of initial vertices of the edges in $\tilde E_{-}$. Then, by Lemma \ref{prop_martingale_obs}, the process
\begin{equation*}
    \| \mathcal{V}_{\delta,n}^{(m)} \|:= \sum_{v \in f_R^{-1}(\tilde V_{-})} \frac{1}{6} \sum_{v \sim w} \bigg( h_{\delta, n}^{(m)}(w) + \frac{1}{2} \bigg), \quad n \geq 0,
\end{equation*}
is a non-negative martingale for the filtration $(\mathcal{F}_{\delta,n})_n$. Therefore, by the optional stopping theorem,
\begin{equation*}
     \| \mathcal{V}_{\delta,0}^{(m)} \| = \EE_{\delta}^{(\Omega,a,b,m)}\bigg[ \| \mathcal{V}_{\delta,\sigma}^{(m)} \| \bigg] \geq \EE_{\delta}^{(\Omega,a,b,m)} \bigg[ \mathbbm{I}_{f_R^{-1}(\mathcal{Q}_{2}^{-})}\| \mathcal{V}_{\delta,\sigma}^{(m)} \| \bigg].
\end{equation*}
From here, we proceed as in the first case, that is we exhibit an upper bound for $\| \mathcal{V}_{\delta,0}^{(m)} \|$ and a lower bound for $\| \mathcal{V}_{\delta,\sigma}^{(m)} \|$ on the event $f_{R}^{-1}(\mathcal{Q}_{2}^{-})$. Observe that, almost surely, for any $n \in \mathbb{N}$,
\begin{align*}
    \sum_{v \in f_R^{-1}(\tilde V_{-})} \frac{1}{6} \sum_{v \sim w} \bigg( h_{\delta, n}^m(w) + \frac{1}{2} \bigg)&\leq \sum_{v \in f_R^{-1}(\tilde V_{-})} \frac{1}{6} \sum_{v \sim w} \bigg( H_{\delta, n}^{(0)}(w) + \frac{1}{2}\PP_{w}^{(\overline{m}_d)}(\tau^{\star}\leq \tau_{\partial \Omega_{\delta,n}}) \bigg) \\
    &= \sum_{v \in \tilde V_{-}} \frac{1}{6} \sum_{v \sim w} \bigg( H_{\delta, n, R}^{(0)}(w) +\frac{1}{2}\PP_{w}^{({m}_a)}(\tau^{\star} \leq \tau_{\partial \tilde \Omega_{\delta,n}}) \bigg), 
\end{align*}
where $H_{\delta, n, R}^{(0)}$ denote the (non-massive) harmonic measure of $\partial \tilde \Omega_{\delta, n}^{+}$. This implies in particular that
\begin{equation} \label{ineq_V0}
    \| \mathcal{V}_{\delta,0}^{(m)} \| \leq \sum_{v \in \tilde V_{-}} \frac{1}{6} \sum_{v \sim w} H_{\delta, 0, R}^{(0)}(w) + \sum_{v \in \tilde V_{-}} \frac{1}{6} \sum_{v \sim w} \frac{1}{2}\PP_{w}^{({m}_a)}(\tau^{\star} \leq \tau_{\partial \tilde \Omega_{\delta}}).
\end{equation}
Our upper bound on $\| \mathcal{V}_{\delta, 0}^{(m)} \|$ that will ultimately allow us to establish the inequality \eqref{crossing_good2} is a consequence of the following two lemmas.

\begin{lemma} \label{lemma_upper_H2}
There exist universal constants $\tilde K_2, \hat \alpha_2 >0$ such that, for any $(z_0,r,R) \in \mathcal{G}_2$,
\begin{equation*}
    \sum_{v \in \tilde V_{-}} \frac{1}{6} \sum_{v \sim w} H_{\delta, 0, R}^{(0)}(w) \leq \tilde K_{2} \bigg( \frac{r/R}{\sqrt{r/R}}\bigg)^{\hat \alpha_2} = \tilde K_{2} \bigg( \frac{r}{R} \bigg)^{\hat \alpha_2 /2}.
\end{equation*}
\end{lemma}

The proof of Lemma \ref{lemma_upper_H2} follows the same lines as that of the first part of the proof of \cite[Proposition~6.3]{HE} but for the sake of completeness, we will sketch the necessary modifications below. The next lemma controls the term in $\| \mathcal{V}_{\delta,0}^{(m)}\|$ arising from the killing and we will prove it using the same strategy as that used to show Lemma \ref{lemma_upper_killing}.

\begin{lemma} \label{lemma_upper_killing2}
For the same constants $K_2$ and $\hat \alpha_2$ as in Lemma \ref{lemma_upper_H2} and still assuming that $(z_0,r,R) \in \mathcal{G}_2$,
\begin{equation*}
     \sum_{v \in \tilde V_{-}} \frac{1}{6} \sum_{v \sim w} \frac{1}{2}\PP_{w}^{({m}_a)}(\tau^{\star} \leq \tau_{\partial \tilde \Omega_{\delta}}) \leq \tilde K_2 \bigg( \frac{r}{R}\bigg)^{\hat \alpha_2 /4} + \frac{1}{2} \times 4\overline{m}_d^2\operatorname{diam}(\Omega)^{2} \bigg( \frac{r}{R}\bigg)^{1/2}.
\end{equation*}
\end{lemma}

We postpone the proof of Lemma \ref{lemma_upper_killing2} to the end and show how to proceed from here. Observe that combined together, Lemma \ref{lemma_upper_H2} and Lemma \ref{lemma_upper_killing2} yield that
\begin{equation*}
    \| \mathcal{V}_{\delta, 0}^{(m)}\| \leq \tilde K_2 \bigg( \frac{r}{R}\bigg)^{\hat \alpha_2 /2} + \tilde K_2 \bigg( \frac{r}{R}\bigg)^{\hat \alpha_2 /4} + \frac{1}{2} \times 4\overline{m}_d^2\operatorname{diam}(\Omega)^{2} \bigg( \frac{r}{R}\bigg)^{1/2}
\end{equation*}
where $\tilde K_2, \hat \alpha_2 >0$ are universal constants.

It now remains to exhibit a lower bound for $\| \mathcal{V}_{\delta, \sigma}^{(m)}\|$ on the event $f_{R}^{-1}(\mathcal{Q}_2^{-})$. As in \eqref{ineq_mart_sigma}, we first observe that, using Lemma \ref{lemma_mHmeasure}, almost surely,
\begin{equation*}
    \| \mathcal{V}^{(m)}_{\delta, \sigma} \| \geq \sum_{v \in f_{R}^{-1}(\tilde V_{-})} \frac{1}{6} \sum_{w \sim v} \frac{1}{2} H_{\delta, \sigma}^{(0)}(w).
\end{equation*}
To lower bound the right-hand side on the event $f_R^{-1}(\mathcal{Q}_2^{-})$, we can use the same reasoning as in the proof of \cite[Propostion~6.3]{HE}, using the annulus $B(z_0', \sqrt{r/R}) \setminus B(z_0', (1/3)\sqrt{r/R})$ and the arc $\partial B(z_0', (1/2)\sqrt{r/R})$. Since we have the same scaling relations, we obtain the same lower bound as in \cite[Propostion~6.3]{HE} and therefore, there exists a universal constant $c_2 >0$ such that
\begin{equation*}
    \EE^{(\Omega,a,b,m)}_{\delta} \bigg[ \mathbb{I}_{f_R^{-1}(\mathcal{Q}_2^{-})} \bigg( \sum_{v \in f_{R}^{-1}(\tilde V_{-})} \frac{1}{6} \sum_{w \sim v} \frac{1}{2} H_{\delta, \sigma}^{(0)}(w)\bigg) \bigg] \geq c_2\PP^{(\Omega,a,b,m)}_{\delta}(f_{R}^{-1}(\mathcal{Q}_2^{-})).
\end{equation*}
Putting everything together, we have thus shown that
\begin{equation*}
    c_2 \PP_{\delta}^{(\Omega,a,b,m)}(f_R^{-1}(\mathcal{Q}_{2}^{-})) \leq \tilde K_2 \bigg( \frac{r}{R}\bigg)^{\hat \alpha_2 /2} + \tilde K_2 \bigg( \frac{r}{R}\bigg)^{\hat \alpha_2 /4} + 2\overline{m}_d^2\operatorname{diam}(\Omega)^{2} \bigg( \frac{r}{R}\bigg)^{1/2}
\end{equation*}
which as explained above implies the inequality \eqref{crossing_good2}. This also concludes the proof of the claim, conditionally on Lemma \ref{lemma_upper_H2} and Lemma \ref{lemma_upper_killing2}.
\end{proof}

To complete the proof of Claim \ref{claim_Q2}, we must prove the two auxiliary lemmas that we used along the way. Let us start with the proof of Lemma \ref{lemma_upper_H2}. As it is very similar to the first part of the proof of \cite[Proposition~6.3]{HE}, we only briefly explain how to adapt the arguments.

\begin{proof}[Proof of Lemma \ref{lemma_upper_H2}]
We use here the same notations as in the proof of Proposition \ref{prop_crossing_0}. Lemma \ref{lemma_upper_H2} can be shown by applying the same arguments as in the first part of the proof of \cite[Proposition~6.3]{HE} with respect to the balls $B(z_0', (1/3)\sqrt{r/R})$, $B(z_0', (1/6)\sqrt{r/R})$ and $B(z_0',r/R)$. Indeed, excursions starting from a vertex in $V_{-}$ and ending with an edge of $\partial \Omega_{\delta,0}^{+}$ must exit the ball $B(z_0',r/R)$. Moreover, observe that the estimates of Lemma \ref{lemma_visits} and Lemma \ref{lemma_massball} do not depend on the meshsize of the graph, and therefore the rescaling $f_R$ is harmless. This yields that there exist universal constants $\tilde K_2, \hat \alpha_2 >0$ such that
\begin{equation} 
    \sum_{v \in \tilde V_{-}} \frac{1}{6} \sum_{v \sim w} H_{\delta, 0, R}^{(0)}(w) \leq \tilde K_{2} \bigg( \frac{r/R}{\sqrt{r/R}}\bigg)^{\hat \alpha_2} = \tilde K_{2} \bigg( \frac{r}{R} \bigg)^{\hat \alpha_2 /2},
\end{equation}
which is the statement of Lemma \ref{lemma_upper_H2}.
\end{proof}

Let us finally establish Lemma \ref{lemma_upper_killing2}.

\begin{proof}[Proof of Lemma \ref{lemma_upper_killing2}]
As the proof is similar to that of Lemma \ref{lemma_upper_killing}, we will be somewhat brief. Using the same notations as in the proof of Claim \ref{claim_Q2}, we rescale $\tilde \Omega_{\delta}$ by $1/\sqrt{(r/R)}$, which yields the rescaled mass
\begin{equation*}
    m_{b}^{2} := \overline{m}_d^2Rr.
\end{equation*}
We define the sets $\mathcal{P}_{\text{in}}(r)$ and $\mathcal{P}_{\text{ext}}(r)$ in the same way as before, but with respect to the ball $B(z_0'', (r/R)^{-1/4})$ instead, where $z_0''$ denotes the image of $z_0'$ after the rescaling. Notice once again that the conclusions of Lemma \ref{lemma_visits} hold for any graph, independently of its meshsize. When considering excursions in $\mathcal{P}_{\text{in}}(r)$, we thus obtain a term of the form
\begin{equation*}
    m_{b}^{2} \times \bigg(\frac{R}{r}\bigg)^{1/2}.
\end{equation*}
Using the fact that for $(z_0,r,R) \in \mathcal{G}_{2}$, $R \leq 2\operatorname{diam}(\Omega)$, we then get that
\begin{align*}
    m_{b}^{2} \times \bigg(\frac{R}{r}\bigg)^{1/2} &= \overline{m}_{d}^2Rr\bigg(\frac{R}{r}\bigg)^{1/2} \\
    &= \overline{m}_d^2\bigg(\frac{r}{2\operatorname{diam}(\Omega)}\bigg)^{1/2} \times \sqrt{2}R^{3/2}\operatorname{diam}(\Omega)^{1/2} \\
    &\leq 4\overline{m}_d^2\operatorname{diam}(\Omega)^{2} \bigg( \frac{r}{R}\bigg)^{1/2}.
\end{align*}
On the other hand, it follows from Lemma \ref{lemma_visits} and Lemma \ref{lemma_massball} that the term arising from excursions in $\mathcal{P}_{\text{ext}}(r)$ is upper bounded by
\begin{equation*}
    \tilde K_2 \bigg( \frac{1}{(r/R)^{-1/4}} \bigg)^{\hat \alpha_2} = \tilde K_2 \bigg( \frac{r}{R} \bigg)^{\hat \alpha_2/4},
\end{equation*}
since $10 \leq  (r/R)^{-1/4}$ for $(z_0,r,R) \in \mathcal{G}_2$. Therefore, we obtain that
\begin{equation*}
     \sum_{v \in \tilde V_{-}} \frac{1}{6} \sum_{v \sim w} \frac{1}{2}\PP_{w}^{({m}_a)}(\tau^{\star} \leq \tau_{\partial \tilde \Omega_{\delta}}) \leq \tilde K_2 \bigg( \frac{r}{R}\bigg)^{\hat \alpha_2 /4} + \frac{1}{2} \times 4\overline{m}_d^2\operatorname{diam}(\Omega)^{2} \bigg( \frac{r}{R}\bigg)^{1/2},
\end{equation*}
which is exactly the inequality claimed in the statement of Lemma \ref{lemma_upper_killing2}.
\end{proof}

The proof of Proposition \ref{prop_crossing_0} is now complete.
\end{proof}

We now turn to the proof of Proposition \ref{prop_crossing}.

\begin{proof}[Proof of Proposition \ref{prop_crossing}]
Let us explain how to adapt the arguments of the proof of Proposition \ref{prop_crossing_0} to show an estimate similar to \eqref{crossing_0} for the conditional probabilities $\PP_{\delta}^{(\Omega,a,b,m)}(\cdot \, \vert \, \gamma([0,\tau]))$, as required by Condition \hyperref[condG3]{G.3}. To this end, fix $0<\delta<\overline{m}_d^{-1}$ and let $\tau$ be a stopping time for the filtration generated by $\gamma$. Let $A(z_0,r,R)$ be an annulus. As before,  we divide the connected components of $A_{\tau}^{\Omega}$ into two sub-collections $\mathcal{A}_{\tau,1}^{\Omega}$ and $\mathcal{A}_{\tau,2}^{\Omega}$ that are defined in the same way as $\mathcal{A}_{1}^{\Omega}$ and $\mathcal{A}_{2}^{\Omega}$ but with respect to $\Omega_{\tau}$, $\gamma(\tau)$ and $b$ instead of $\Omega$, $a$ and $b$. Notice that in both cases, we can use the same sets $\mathcal{G}_1$ and $\mathcal{G}_2$ as above. Moreover, in both cases, we can also define the same events and processes as for the time $\tau=0$ estimates, except that now everything is conditioned on $\gamma([0,\tau])$. More precisely, conditionally on $\gamma([0,\tau])$, the event $\mathcal{Q}$ is defined as for the time $0$ estimate but with respect to $A^{\Omega}_{\tau}$. Conditionally on $\gamma([0,\tau])$, on $\mathcal{Q}$, the stopping time $\sigma$ is then defined as the least $j$ such that  $v_{\tau+j} \in B_{1,\tau}^{av} \cup B_{2,\tau}^{av}$ and $\gamma([\tau,\tau+j])$ contains a crossing of $A_{\tau}^{\Omega}$. Conditionally on $\gamma([0,\tau])$, the events $\mathcal{Q}_{1}^{\pm}$ and $f_{R}^{-}(\mathcal{Q}_{2}^{\pm})$ are then defined as above but with respect to $\mathcal{A}_{\tau,1}^{\Omega}$, $\mathcal{A}_{\tau, 2}^{\Omega}$ and $\Omega_{\tau}$. The argument based on the optional stopping theorem is replaced by the observation that the martingale property of $(h_{\delta, n}^m(w), n \geq 0)$ for any $w$ together with the domain Markov property \eqref{discrete_DMP} imply that, almost surely,
\begin{align*}
    \| \tilde \nu_{\delta,\tau}^{(m)} \| &= \EE_{\delta}^{(\Omega,a,b,m)}[\| \tilde \nu_{\delta,\tau}^{(m)} \| \, \vert \,\gamma([0,\tau])] \\
    &= \EE_{\delta}^{(\Omega,a,b,m)}[\| \tilde \nu_{\delta,\tau + \sigma}^{(m)} \| \, \vert \, \gamma([0,\tau])] \\
    &\geq \EE_{\delta}^{(\Omega,a,b,m)}[ \mathbbm{I}_{\mathcal{Q}_{1}^{-}} \| \tilde \nu_{\delta,\tau + \sigma}^{(m)} \| \, \vert \, \gamma([0,\tau])].
\end{align*}
Similarly, we have that, almost surely,
\begin{align*}
    \| \mathcal{V}_{\delta,\tau}^{(m)} \| &= \EE_{\delta}^{(\Omega,a,b,m)}[\| \mathcal{V}_{\delta,\tau}^{(m)} \| \, \vert \,\gamma([0,\tau])] \\
    &= \EE_{\delta}^{(\Omega,a,b,m)}[\| \mathcal{V}_{\delta,\tau + \sigma}^{(m)} \| \, \vert \, \gamma([0,\tau])] \\
    &\geq \EE_{\delta}^{(\Omega,a,b,m)}[ \mathbbm{I}_{f_{R}^{-1}(\mathcal{Q}_{2}^{-})} \| \mathcal{V}_{\delta,\tau + \sigma}^{(m)} \| \, \vert \, \gamma([0,\tau])].
\end{align*}
Now observe that the strategy used to upper bound  $\| \mathcal{V}_{\delta,0}^{(m)} \|$ and $ \| \tilde \nu_{\delta,0}^{(m)} \|$ that we used to prove Claim \ref{claim_Q1} and Claim \ref{claim_Q2} can be apply to obtain an almost sure upper bound on $\| \mathcal{V}_{\delta,\tau}^{(m)} \|$ and $ \| \tilde \nu_{\delta,\tau}^{(m)} \|$ conditionally on $\gamma_{\delta}([0,\tau])$, without making any change in the proof. In particular, the upper bounds do not depend on $\operatorname{diam}(\Omega_{\tau})$ but only on $\operatorname{diam}(\Omega)$. Similarly, the exact same arguments as in the time $\tau=0$ case yield the same almost sure lower bound as in the case $\tau=0$ on $\| \tilde \nu_{\delta,\tau + \sigma}^{(m)} \|$ on the event $\mathcal{Q}_1^{-}$ conditionally on $\gamma([0,\tau])$. The analogous statement holds for $\| \mathcal{V}_{\delta,\tau + \sigma}^{(m)} \|$ on the event $f_R^{-1}(\mathcal{Q}_2^{-})$ conditionally on $\gamma([0,\tau])$. This crucially implies that the constants in the upper and lower bounds do not depend on $\tau$ and are the same as in the case $\tau=0$. We can therefore conclude that for the same $K$ and $\alpha$ as in the Proposition \ref{prop_crossing_0}, almost surely,
\begin{equation*}
    \PP_{\delta}^{(\Omega,a,b,m)}(\gamma([\tau,1]) \, \text{makes a crossing of $A$ which is contained in $A_{\tau}^{\Omega}$} \vert \, \gamma([0,\tau])) \leq K \bigg (\frac{r}{R}\bigg)^{\alpha}.
\end{equation*} 
\end{proof}

\section{Scaling limit of the martingale observable} \label{sec_martobs}

In this section, we study the scaling limit of the martingale observable $(h_{\delta,n}^m, n \geq 0)$. The results of Section \ref{sec_tightness} show that the sequence $(\gamma_{\delta}^{\HH})_{\delta}$ is tight in the topologies \eqref{topo_1} -- \eqref{topo_3}, which implies that $(\gamma_{\delta}^{\HH})_{\delta}$ converges along subsequences in these topologies. By \cite[Corollary~1.8]{Smirnov} and \cite[Theorem~4.2]{Karrila}, if $(\gamma_{\delta_k}^{\HH})_{k}$ is such a convergent subsequence, then $(\gamma_{\delta_k})_k$ in turn converges weakly in $X(\mathbb{C})$ equipped with the metric $d_X$ to a random curve that is almost surely supported on $\overline{\Omega}$ and has the same law as $\phi^{-1}(\gamma^{\HH})$, provided that for each $\delta_k$, $\gamma_{\delta_k}$ is parametrized by the half-plane capacity of $\gamma_{\delta_k}^{\HH}$. Here, assuming that $(\hat \Omega_{\delta_k}, a_{\delta_k}, b_{\delta_k})$ converges in the Carath\'eodry topology to $(\Omega,a,b)$, we show that the corresponding subsequence of massive harmonic functions $(h_{\delta_k}^m)_{\delta_k}$ converges pointwise to the continuous massive harmonic function with the same boundary conditions. After suitable reparametrization, we will prove that this in fact holds almost surely for the time-dependent subsequences $(h_{\delta_k, t(\delta_k)}^m)_{\delta_k}$ in the domains $(\Omega_{\delta_k,t(\delta_k)})_{\delta_k}$. This is established under the assumptions that the time-reparametrized (random) sequence $(\hat \Omega_{\delta_k, t(\delta_k)}, a_{t(\delta_k)}, b_{\delta_k})_k$ almost surely converges in the Carath\'eodory topology to $(\Omega_t, a_t, b)$ where $\Omega_t := \Omega \setminus K_t$, with $K_t$ being the hull generated at time $t$ by the limit $\gamma$ of $(\gamma_{\delta_k})_{\delta_k}$ and $a_t=\gamma(t)$. This will be the key to characterize in Section \ref{sec_charac} the subsequential limits obtained as a consequence of the tightness of $(\gamma_{\delta}^{\HH})_{\delta}$ established in the previous section.

To study the scaling limit of the martingale observable, we choose to use the framework developed by Chelkak and Smirnov in \cite{Discrete_analysis}. For $\delta > 0$, we define the Laplacian $\Delta_{\delta}$ on $\Omega_{\delta}$ by, for $f: \Omega_{\delta} \to \mathbb{R}$ and $z \in \operatorname{Int}(\Omega_{\delta})$,
\begin{equation*}
    \Delta_{\delta}f(z) = \frac{1}{A_{\delta}} \sum_{w \sim z} \tan(\theta) (f(w) - f(z)),
\end{equation*}
where as above $A_{\delta} = c\delta^2$ is the area of a face of the graph dual to $\Omega_{\delta}$ and $\theta$ is defined as in the definition \eqref{def_discrete_mass} of $m_d^2$. Let $f: \Omega_{\delta} \to \mathbb{R}$ be a discrete massive harmonic function with mass $m$ in $\Omega_{\delta}$. It follows from the definition of a discrete massive harmonic function given in \eqref{def_mharm} that $f$ satisfies, for all $z \in \operatorname{Int}(\Omega_{\delta})$,
\begin{equation*}
    \sum_{z \sim w} \frac{1- m_d^2(z) \delta^2}{6} (f(z)-f(w)) + m_d^2(z) \delta^2 f(z) = 0.
\end{equation*}
Multiplying both sides by $6\tan(\theta)/A_{\delta}$, this is equivalent to
\begin{equation} \label{mharm_approx}
    -\Delta_{\delta} f(z) + m^{2}(z) f(z) + A_{\delta} \frac{m^2(z)}{6\tan(\theta)} \Delta_{\delta} f(z) = 0.
\end{equation}

The following lemma gives a continuity estimate for discrete massive harmonic functions. This estimate will be useful to show precompactness of the family $(h_{\delta}^m)_{\delta}$ in the proof of Proposition \ref{prop_cvg_martingale}. 

\begin{lemma} \label{lemma_equi}
Let $m: \overline{\Omega} \to \mathbb{R}_{+}$ be a continuous function bounded by some constant $\overline{m}>0$ in $\overline{\Omega}$. Let $\delta > 0$ be such that $\delta < \overline{m}_d^{-1}$. There exist constants $C, \beta > 0$ depending on $\overline{m}$ such that the following holds. Let $H$ be a positive massive discrete harmonic function with mass $m$ defined in $B(z,2r) \cap \delta \mathbb{T}$ with $0<r\leq \overline{m}^{-1}$. Then, for any $w_1, w_2 \in B(z,2r) \cap \delta \mathbb{T}$,
\begin{equation*}
    \vert H(w_1) - H(w_2) \vert \leq C \bigg( \frac{\vert w_1 - w_2 \vert}{r} \bigg)^{\beta} \max_{w \in B(z,2r) \cap \delta \mathbb{T}} H(w).
\end{equation*}
\end{lemma}

\begin{proof}
The proof is almost identical to that of \cite[Lemma~3.4]{off_dimers}, where this lemma is established for constant mass. Indeed, the proof relies on an estimate on the probability that a massive random walk makes a non-contractible loop in an annulus of modulus $1/2$ before exiting it or being killed, when the walk started not too far from the center of the annulus. Therefore, in our setting, we can lower bound this probability by that of the same event taking place when the walk instead has probability $\overline{m}^2_d \delta^2$ to be killed at each step. This latter probability is exactly the probability that is analyzed in the proof of \cite[Lemma~3.4]{off_dimers}.
\end{proof}

Recall that for $\delta > 0$, $\phi_{\delta}: \hat \Omega_{\delta} \to \HH$ is a conformal map such that $\phi_{\delta}(a_{\delta})=0$ and $\phi_{\delta}(b_{\delta})=\infty$. Recall also that, for $\delta > 0$, we denote by $\gamma_{\delta}^{\HH}$ the curve $\phi_{\delta}(\gamma_{\delta})$. To establish convergence of the martingale observable, it will be more convenient to parametrize the discrete curves $(\gamma_{\delta})_{\delta}$ by the half-plane capacity of their conformal images $(\gamma_{\delta}^{\HH})_{\delta}$. Indeed, the curves $(\gamma_{\delta}^{\HH})_{\delta}$ can be considered as continuous curves in $\HH$. As such, they can be canonically parametrized by half-plane capacity, as explained in Section \ref{sec_Loewner}. For $n \geq 0$, let $\mathcal{F}_{\delta,n}$ be the $\sigma$-algebra generated by $\gamma_{\delta}([0,n])$ and, for $t \geq 0$ and $\delta >0$, let us define the following stopping time for the filtration $(\mathcal{F}_{\delta, n})_n$
\begin{equation*}
    t(\delta) := \inf \{ n \in \mathbb{N}: \text{hcap}(\phi_{\delta}(\gamma_{\delta}[0,n])) \geq t \}.
\end{equation*}
We then let $\Omega_{t(\delta)}$ be the connected component of $\Omega_{\delta} \setminus \gamma_{\delta}([0,t(\delta)])$ which contains both $a_{t(\delta)}$ and $b_{\delta}$ on its boundary, where $a_{t(\delta)}$ is the last vertex added to $\partial \Omega_{t(\delta)}$.

By the results of Section \ref{sec_tightness}, the sequence $(\gamma_{\delta}^{\HH})_{\delta}$ is tight in the topologies \eqref{topo_1}--\eqref{topo_3} and each subsequence $(\gamma_{\delta_k})_k$ weakly converges in $X(\mathbb{C})$ with the metric $d_X$ to a curve $\gamma$ supported on $\overline{\Omega}$, provided that the curves $(\gamma_{\delta_k})_k$ are parametrized by the half-plane capacitiy of their conformal images $(\gamma_{\delta_k}^{\HH})_{\delta_k}$. Moreover, $\gamma$ has the same law as $\phi^{-1}(\gamma^{\HH})$. The space of continuous functions on $[0,\infty)$ is a separable and metrizable space and therefore, by Skorokhod representation theorem, we can suppose that for each weakly convergent subsequence of $(\gamma_{\delta_k})_k$, we also have $\gamma_{\delta_k} \to \gamma$ almost surely. In particular, we can assume that, almost surely, in the Carath\'eodory sense,
\begin{equation*}
     (\hat \Omega_{t(k)}; a_{t(k)}, b_{\delta_k}) \to (\Omega_t;a_t,b) \quad \text{as } \delta \to 0,
\end{equation*}
where we have set $t(k):= t(\delta_k)$ and $\Omega_t$ is the connected component of $\Omega \setminus \gamma([0,t])$ which contains $a_t$ and $b$.

\begin{proposition} \label{prop_cvg_martingale}
Let $t \geq 0$. Let $(\delta_k)_{k}$ be a subsequence such that the sequence $(\gamma_{\delta_k}^{\HH})_k$ converges weakly in the topologies \eqref{topo_1} -- \eqref{topo_3} to a random curve $\gamma^{\HH}$, and thus $(\gamma_{\delta_k})_{k}$ also converges weakly to a random curve $\gamma$ in $X(\mathbb{C})$ equipped with the metric $d_X$. Then the sequence of discrete massive harmonic functions $(h_{\delta_{k}, t(\delta_k)}^m)_k$ almost surely converges pointwise in $\overline{\Omega_t}$ to the massive harmonic function $h_t^m: \Omega_t \to \mathbb{R}$ solving the Dirichlet problem
\begin{equation*}
    \begin{cases}
    &[-\Delta + m^2(z)]h_t^m(z) = 0, \quad z \in \Omega_t \\
    &{h_t^m}_{\vert \partial \Omega_t^{+}} = \frac{1}{2}, \quad {h_t^m}_{\vert \partial \Omega_t^{-}} = \frac{-1}{2}.
    \end{cases}
\end{equation*}
\end{proposition}

Proposition \ref{prop_cvg_martingale} will be a consequence of a general deterministic result. To state this result, let us introduce a few notations, which mimic the setting of Proposition \ref{prop_cvg_martingale}. Let $\Lambda$ be an open, bounded and simply connected subset of $\mathbb{C}$ and let $a, b \in \partial \Lambda$. We denote by $\partial \Lambda^{+}$, respectively $\partial \Lambda^{-}$, the clockwise boundary arc from $a$ to $b$, respectively counterclockwise boundary arc from $a$ to $b$. We then approximate $\Lambda$ by a sequence of graphs $(\Lambda_{\delta})_{\delta}$ where for each $\delta > 0$, $\Lambda_{\delta}$ is a portion of $\delta \mathbb{T}$. We define $\partial \Lambda_{\delta}$ as in Section \ref{subsec:domain}. Recall also that $\hat \Lambda_{\delta}$ denotes the open and simply connected polygonal domain obtained from $\Lambda_{\delta}$ by taking the union of all open hexagons with side length $\delta$ centered at vertices of $\Lambda_{\delta}$. As in Section \ref{subsec:domain}, we obtain two sequences $(a_{\delta})_{\delta}$ and $(b_{\delta})_{\delta}$ that approximate the boundary points $a$ and $b$. We separate $\partial \Lambda_{\delta}$ into two subsets, $\partial \Lambda_{\delta}^{+}$ and $\partial \Lambda_{\delta}^{-}$, where $\partial \Lambda_{\delta}^{-}$, respectively $\partial \Lambda_{\delta}^{-}$ are defined in a similar fashion as $\partial \Omega_{\delta}^{+}$ and $\partial \Omega_{\delta}^{-}$. Finally, we let $m: \overline{\Lambda} \to \mathbb{R}_{+}$ be a continuous function which is bounded by some constant $\overline{m} >0$. For $\delta < \overline{m}_d^{-1}$, where $m_{d}$ is defined as in \eqref{def_discrete_mass}, we denote by $h_{\delta}^m$ the discrete massive harmonic function in $\Lambda_{\delta}$ with boundary value $1/2$ on $\partial \Lambda_{\delta}^{+}$ and $-1/2$ on $\partial \Lambda_{\delta}^{-}$.

\begin{lemma}
In the above setting, assume that $(\hat \Lambda_{\delta}; a_{\delta}, b_{\delta})$ converges to $(\Lambda;a,b)$ in the Carath\'eodory sense. Then the sequence of functions $(h_{\delta}^m)_{\delta}$ converges pointwise in $\overline{\Lambda}$ to the massive harmonic function $h^m: \overline{\Lambda} \to \mathbb{R}$ solving the Dirichlet problem
\begin{equation*}
\begin{cases}
    [-\Delta + m^2(z)]h^m(z) = 0, \quad z \in \Lambda \\
    h^m_{\vert \partial \Lambda^{+}} = \frac{1}{2}, \quad h^m_{\vert \partial \Lambda^{-}} = \frac{-1}{2}.
\end{cases}
\end{equation*}
\end{lemma}

\begin{proof}
We first observe that, for any $\delta < \overline{m}_{d}^{-1}$ and any $z \in \operatorname{Int}(\Lambda_{\delta}) \cup \partial \Lambda_{\delta}$,
\begin{equation*}
    \vert h_{\delta}^m (z) \vert \leq \frac{1}{2},
\end{equation*}
that is the sequence $(h_{\delta}^m)_{\delta < \overline{m}_{d}^{-1}}$ is uniformly bounded. $(h_{\delta}^m)_{\delta < \overline{m}_{d}^{-1}}$ is also equicontinuous. Indeed, we have, for any $z \in \operatorname{Int}(\Lambda_{\delta})$,
\begin{equation*}
    h_{\delta}^m(z) = \frac{1}{2}\bigg( H_{\delta}^{(m)} (z) - \tilde H_{\delta}^{(m)} (z)\bigg),
\end{equation*}
where $H_{\delta}^m(z)$, respectively $\tilde H_{\delta}^m(z)$, is the discrete massive harmonic measure of $\partial \Lambda_{\delta}^{+}$, respectively $\partial \Lambda_{\delta}^{-}$, seen from $z$. Therefore, for any $\delta < \overline{m}_d^{-1}$, applying Lemma \ref{lemma_equi} to both $H_{\delta}^{(m)}$ and $\tilde H_{\delta}^{(m)}$, we can see that for any $z \in \Lambda_{\delta}$ and any $w_1, w_2 \in B(z,2r) \cap \delta \mathbb{T}$ with $r \leq \overline{m}_{d}^{-1} \wedge \text{dist}(z, \partial \Lambda_{\delta})$,
\begin{align*}
    \vert h_{\delta}^m(w_1) - h_{\delta}^m(w_2) \vert &\leq \frac{1}{2} \bigg( \vert H_{\delta}^{(m)}(w_1) - H_{\delta}^{(m)}(w_2)\vert + \vert \tilde H_{\delta}^{(m)}(w_1) -\tilde H_{\delta}^{(m)}(w_2)\vert \bigg) \\
    &\leq \frac{C}{2} \bigg(\frac{\vert w_1 - w_2 \vert}{r} \bigg)^\beta \\
    & \times \bigg(\max_{w \in B(z,2r) \cap \delta \mathbb{T}} H_{\delta}^{(m)}(w) + \max_{w \in B(z,2r) \cap \delta \mathbb{T}} \tilde H_{\delta}^{(m)}(w) \bigg) \\
    &\leq C \bigg(\frac{\vert w_1 -  w_2 \vert}{r}\bigg)^\beta.
\end{align*}
By the Arzela-Ascoli theorem, uniform boundedness and equicontinuity of the sequence $(h_{\delta}^m)_{\delta < \overline{m}_{d}^{-1}}$ implies that there exists a function $h^m: \Lambda \to \mathbb{R}$ and a subsequence $(h_{\delta_k}^m)_k$ such that  $(h_{\delta_k}^m)_k$ converges uniformly on compact subsets of $\Lambda$ to $h^m$. Let us show that $h^m$ is the massive harmonic function with mass $m$ in $\Lambda$ and boundary conditions $-1/2$ on $\partial \Lambda^{-}$ and $1/2$ on $\partial \Lambda^{+}$. 

We first prove that $h^m$ is massive harmonic with mass $m$ in $\Lambda$. Let $\varphi: \Lambda \to \mathbb{R}$ be a smooth and compactly supported function on $\Lambda$. We then have
\begin{align*}
    \int_{\Lambda} h^m(z) (-\Delta \varphi(z) + m^2(z) \varphi(z)) dz = \lim_{\delta = \delta_{k} \to 0} \sum_{z \in \operatorname{Int}(\Lambda_{\delta})} A_{\delta} (h^m)^{\delta}(z) (- (\Delta \varphi)^{\delta}(z) + m^{2}(z) \varphi^{\delta}(z))
\end{align*}
where for $f: \Lambda \to \mathbb{R}$ and $\delta>0$, $f^{\delta}: \Lambda_{\delta} \to \mathbb{R}$ is defined as the projection of $f$ onto $\Lambda_{\delta}$. Using \cite[Lemma~2.2]{Discrete_analysis}, we get that
\begin{align*}
    \lim_{\delta = \delta_{k} \to 0} &\sum_{z \in \operatorname{Int}(\Lambda_{\delta})} A_{\delta} (h^m)^{\delta}(z) (- (\Delta \varphi)^{\delta}(z) + m^{2}(z) \varphi^{\delta}(z))\\
    &=  \lim_{\delta = \delta_{k} \to 0} \sum_{z \in \operatorname{Int}(\Lambda_{\delta})} A_{\delta} h_{\delta}^{m}(z) (- \Delta_{\delta} \varphi^{\delta}(z) + m^{2}(z) \varphi^{\delta}(z)).
\end{align*}
We also have that
\begin{equation*}
    \lim_{\delta = \delta_{k} \to 0} A_{\delta} \sum_{z \in \operatorname{Int}(\Lambda_{\delta})} A_{\delta} h_{\delta}^{m}(z) \frac{m^2(z)}{6\tan(\theta)} \Delta_{\delta} \varphi^{\delta}(z) = 0
\end{equation*}
since
\begin{equation*}
    \lim_{\delta = \delta_{k} \to 0} \sum_{z \in \operatorname{Int}(\Lambda_{\delta})} A_{\delta} h_{\delta}^{m}(z) \frac{m^2(z)}{6\tan(\theta)} \Delta_{\delta} \varphi^{\delta}(z) = \int_{\Lambda} h^m(z) \frac{m^2(z)}{6\tan(\theta)} \Delta \varphi(z) dz.
\end{equation*}
Therefore, we have that
\begin{align*}
    \int_{\Lambda} h^m(z) (-\Delta \varphi(z) + m^2(z) \varphi(z)) dz = \lim_{\delta = \delta_{k} \to 0} &\sum_{z \in \operatorname{Int}(\Lambda_{\delta})} A_{\delta} h_{\delta}^{m}(z) (- \Delta_{\delta} \varphi^{\delta}(z) + m^{2}(z) \varphi^{\delta}(z))\\
    & + A_{\delta} \sum_{z \in \operatorname{Int}(\Lambda_{\delta})} A_{\delta} h_{\delta}^{m}(z) \frac{m^2(z)}{6\tan(\theta)} \Delta_{\delta} \varphi^{\delta}(z).
\end{align*}
By discrete integration by part, this implies that
\begin{align*}
    \int_{\Lambda} h^m(z) (-\Delta \varphi(z) + m^2(z) \varphi(z)) = \lim_{\delta = \delta_{k} \to 0} \sum_{z \in \operatorname{Int}(\Lambda_{\delta})} & A_{\delta} \varphi^{\delta}(z) \bigg[ -\Delta_{\delta}  h_{\delta}^{m}(z)  + m^{2}(z)  h_{\delta}^{m}(z) \\
    &+ A_{\delta} \frac{m^2(z)}{6\tan(\theta)} \Delta_{\delta} h_{\delta}^{m}(z) \bigg].
\end{align*}
Since $h_{\delta}^{m}$ is discrete massive harmonic with mass $m$, by \eqref{mharm_approx}, the right-hand side is  equal to 0 and thus,
\begin{equation*}
    \int_{\Lambda} h^m(z) (-\Delta \varphi(z) + m^2(z) \varphi(z)) dz = 0.
\end{equation*}
Therefore, $h^m$ is weakly massive harmonic with mass $m$ in $\Lambda$. But this implies that $h^m$ is in fact massive harmonic with mass $m$ in $\Lambda$.

We now want to show that $h^m$ is equal to the massive harmonic function with mass $m$ and boundary conditions $1/2$ on $\partial \Lambda^{+}$ and $-1/2$ on $\partial \Lambda^{-}$. Recall that, for any $0<\delta < \overline{m}_d^{-1}$ and any $z \in \operatorname{Int}(\Lambda_{\delta}) \cup \partial \Lambda_{\delta}$,
\begin{equation} \label{eq_dec_hm}
    h_{\delta}^m(z) = \frac{1}{2} \bigg( H^{(m)}_{\delta}(z) - \tilde H^{(m)}_{\delta}(z) \bigg).
\end{equation}
The same reasoning as above shows that there exist subsequences $(H^{(m)}_{\delta_q})_q$ and $(\tilde H^{(m)}_{\delta_r})_r$ and functions $H^m: \Lambda \to \mathbb{R}$ and $\tilde H^m: \Lambda \to \mathbb{R}$ such that $(H^m_{\delta_q})_q$, respectively $(\tilde H^m_{\delta_r})_r$, converges uniformly on compact subsets of $\Lambda$ to $H^m$, respectively $\tilde H^m$. Moreover, $H^m$ and $\tilde H^m$ are both massive harmonic with mass $m$ in $\Lambda$. Let us show that $H^m$, respectively $\tilde H^m$, is in fact the massive harmonic measure of $\partial \Lambda^{+}$, respectively $\partial \Lambda^{-}$. As the proof is virtually the same for both $H^m$ and $\tilde H^m$, we only detail the arguments for $H^m$.

Observe that by the weak Beurling estimate for (massless) harmonic measure, see e.g. \cite[Proposition~2.11]{Discrete_analysis}, we have, for any $z \in \operatorname{Int}(\Lambda_{\delta_q})$,
\begin{equation*}
    0 \leq H^{(m)}_{\delta_q}(z) \leq \text{const} \bigg( \frac{\text{dist}(z, \partial \Lambda_{\delta_{q}})}{\text{dist}(z, \partial \Lambda_{\delta_{q}}^{+})}\bigg)^{\hat \beta}
\end{equation*}
where the constant and $\hat \beta > 0$ are independent of $\delta_q$. Passing to the limit $\delta_{q} \to 0$, we obtain that, for any $z \in \Lambda$,
\begin{equation*}
    0 \leq H^{(m)}(z) \leq C \bigg( \frac{\text{dist}(z, \partial \Lambda)}{\text{dist}(z, \partial \Lambda^{+})}\bigg)^{\hat \beta}.
\end{equation*}
We can thus conclude that, for $z \in \partial \Lambda^{-}$, $H^m(z)= 0$. Now, recall that, for any $q \in \mathbb{N}$ and $z \in \operatorname{Int}(\Lambda_{\delta_{q}})$,
\begin{equation*}
    \tilde H_{\delta_{q}}^{(m)} = 1 -  H_{\delta_{q}}^{(m)}(z) - \PP_z^{(m)}(\tau^{\star} \leq \tau_{\partial \Lambda_{\delta_{q}}}).
\end{equation*}
Once again, by the weak Beurling estimate for (massless) harmonic measure, we have that, for any $z \in \operatorname{Int}(\Lambda_{\delta_{q}})$,
\begin{equation} \label{Beurling_w}
    0 \leq \tilde H_{\delta_{q}}^{(m)}(z) \leq \text{const} \bigg( \frac{\text{dist}(z, \partial \Lambda_{\delta_{q}})}{\text{dist}(z, \partial \Lambda_{\delta_{q}}^{-})}\bigg)^{\hat \beta}.
\end{equation}
Moreover, if $(z_{\delta_q})_q$ is a sequence of points such that for each $q \in \mathbb{N}$, $z_{\delta_q} \in \delta_q\mathbb{T}$ and $z_{\delta_q} \to z$, then
\begin{equation*}
    \liminf_{\delta = \delta_{q} \to 0}  \PP_{z_\delta}^{(m)}(\tau^{\star} \leq \tau_{\partial \Lambda_{\delta_{q}}}) \leq 1-\mathbb{I}_{z \notin \Lambda}.
\end{equation*}
Taking the $\limsup$ as $\delta_{q} \to 0$, this yields that
\begin{equation*}
    1- H^m(z) - (1-\mathbb{I}_{z \notin \Lambda}) \leq 1-H^m(z)-\liminf_{\delta = \delta_{q} \to 0} \PP_{\delta}^{z}(\tau^{\star} \leq \tau_{\partial \Lambda_{\delta}}) \leq C \bigg( \frac{\text{dist}(z, \partial \Lambda)}{\text{dist}(z, \partial \Lambda^{-})}\bigg)^{\hat \beta}
\end{equation*}
which implies that
\begin{equation*}
    \limsup_{z \to \partial \Lambda^{+}} 1-H^m(z) \leq 0.
\end{equation*}
On the other hand, since for any $\delta< \overline{m}_d^{-1}$ and $z \in \operatorname{Int}(\Lambda_{\delta}) \cup \partial \Lambda_{\delta}$, $H_{\delta}^{m}(z) \leq 1$, we have that
\begin{equation*}
    0 \leq \limsup_{z \to \partial \Lambda^{+}} \lim_{\delta = \delta_{q} \to 0} 1- H_{\delta}^{(m)}(z) = \limsup_{z \to \partial \Lambda^{+}} 1-H^m(z).
\end{equation*}
Therefore, we obtain that $\lim_{z \to \partial \Lambda^{+}} 1-H^m(z) = 0$, which yields that $H^m$ is equal to $1$ on $\partial \Lambda^{+}$.

The above arguments show that the whole sequences $(H_{\delta}^{(m)})$ and $(\tilde H_{\delta}^{(m)})$ converge pointwise to the massive harmonic measure of $\partial \Lambda^{+}$ and $\partial \Lambda^{-}$, respectively. Recalling the decomposition \eqref{eq_dec_hm} of $h_{\delta}^m$, we can thus conclude that $(h_{\delta}^m)_{\delta}$ converges pointwise along any subsequence, and thus converges pointwise, to the function
\begin{equation*}
    z \in \Lambda \mapsto \frac{1}{2} \bigg( H^m(z) - \tilde H^m(z) \bigg).
\end{equation*}
This function is massive harmonic with mass $m$ in $\Lambda$ and has boundary conditions $1/2$ on $\partial \Lambda^{+}$ and $-1/2$ on $\partial \Lambda^{-}$. By uniqueness of such massive harmonic functions, we obtain that the limit $h^m$ of $h_{\delta}^m$ is indeed solution to the Dirichlet problem of the statement of the lemma.
\end{proof}

\section{Characterization of the limiting continuum curve} \label{sec_charac}

Recall that we have a random sequence $(\gamma_{\delta})_{\delta}$ of curves where for each $\delta > 0$, $\gamma_{\delta}$ is distributed according to $\PP_{\delta}^{(\Omega,a,b,m)}$. For each $\delta >0$, we also have a conformal map $\phi_{\delta} \colon \hat \Omega_{\delta} \to \mathbb{H}$ such that $\phi_{\delta}(a_{\delta})=0$ and $\phi_{\delta}(b_{\delta})=\infty$ and we denote by $\gamma_{\delta}^{\HH}$ the curve $\phi_{\delta}(\gamma_{\delta})$. We have shown in Section \ref{sec_tightness} that the sequence $(\gamma_{\delta}^{\HH})_{\delta}$ is tight in the topologies \eqref{topo_1} -- \eqref{topo_3}, which implies that ${(\gamma_{\delta}^{\HH})}_{\delta}$ converges weakly along subsequences in these topologies. If $(\gamma_{\delta_k}^{\HH})_{k}$ is such a convergent subsequence, then its limit $\gamma^{\HH}$ is a random non-self crossing curve in $\HH$ whose time evolution can therefore be described by the Loewner equation, see \eqref{eq_Loewner}. Moreover, in this case, ${(\gamma_{\delta_k})}_k$ converges weakly in $X(\mathbb{C})$ equipped with the metric $d_X$ to a random curve that is almost surely supported on $\overline{\Omega}$ and has the same law as $\phi^{-1}(\gamma^{\HH})$, provided that for each $\delta_k$, $\gamma_{\delta_k}$ is parametrized by the half-plane capacity of $\gamma_{\delta_k}^{\HH}$. Our goal here is to characterize the limits of such subsequences, or equivalently the Loewner chain describing their time evolution: we are going to show that this limiting Loewner chain is characterized by the martingale property of a certain massive harmonic function. Before stating precisely the result, in Section \ref{sec_intro_Mharm}, we introduce a few notations and recall how to express massive harmonic functions in terms of their harmonic counterparts. The characterization of the limiting Loewner chain is then stated in Section \ref{sec_proof_charac} and Section \ref{subsec_Proof} is devoted to its proof. In Section \ref{subsec_ccl}, we reformulate Theorem \ref{theorem_intro} and show how to prove it by combining the results of Section \ref{sec_proof_crossing}, Section \ref{sec_martobs} and Section \ref{sec_proof_charac}.

\subsection{Massive harmonic functions and massive Poisson kernels} \label{sec_intro_Mharm}

Let $\Lambda \subset \mathbb{C}$ be a bounded, open and simply connected domain. The Laplace operator $-\Delta$ in $\Lambda$ with Dirichlet boundary conditions has a unique Green function $G_{\Lambda}$, which is defined as its inverse in the sense of distributions, that is for $z \in \Lambda$, $-\Delta G_{\Lambda}(z,\cdot) = \delta_z(\cdot)$. In what follows, we will be interested in quantities related to the massive Laplace operator $-\Delta +m^2$ in $\Lambda$ with Dirichlet boundary conditions. This operator acts on a function $f \in \mathcal{C}_c^{\infty}(\Lambda)$ as $[-\Delta + m^2(z)]f(z) = -\Delta f(z) + m^2(z)f(z)$. It also has a unique Green function $G_{\Lambda}^m$ defined as its inverse in the sense of distributions, that is, for $z \in \Lambda$, $[-\Delta + m^2(\cdot)]G_{\Lambda}^m(z,\cdot) = \delta_z(\cdot)$. We call $G_{\Lambda}^m$ the massive Green function (with mass $m$ in $\Lambda$). Since for any $z \in \Lambda$, $-\Delta G_{\Lambda}^m(z, \cdot) = \delta_z(\cdot) - m^2(\cdot)G_{\Lambda}^m(z, \cdot)$ (in the sense of distributions), $G_{\Lambda}^{m}$ is related to $G_{\Lambda}$ as follows: for $z, w \in \Lambda$,
\begin{equation} \label{eq_G_Gm}
    G_{\Lambda}^{m}(z,w) = G_{\Lambda}(z,w) - \int_{\Lambda} m^{2}(y) G_{\Lambda}(z,y)G_{\Lambda}^{m}(w,y) dy.
\end{equation}
Indeed, one can check that the right-hand side of this equality is the inverse of $-\Delta + m^2$ in the sense of distributions, which thus establishes \eqref{eq_G_Gm}. Moreover, $G_{\Lambda}^m$ is conformally covariant in the following sense. Let $\phi: \Lambda \to \tilde \Lambda$ be a conformal map and set, for $y \in \tilde \Lambda$, $\tilde m^2(y) = \vert (\phi^{-1})^{'}(y) \vert^2 m^2(\phi^{-1}(y))$. Then, for any $z,w \in \Lambda$,
\begin{equation} \label{cov_Green}
    G_{\Lambda}^m(z,w) = G_{\tilde \Lambda}^{\tilde m}(\phi(z),\phi(w)).
\end{equation}
This equality is a consequence of the conformal covariance of the two-dimensional massive (killed) Brownian motion. Indeed, as in the case of standard Brownian motion, if $A \subset \Lambda$ is an open set, then
\begin{equation*}
    \int_{A} G_{\Lambda}^m(z,w) dw = \EE_{z}^{(m)}\bigg[ \mathbb{I}_{\tau^{\star} > \tau_{\Lambda}} \int_{0}^{\tau_{\Lambda}} \mathbb{I}_{A}(B_t) dt \bigg]
\end{equation*}
where under $\EE_z^{(m)}$, $B$ has the law of a massive Brownian motion with mass $m$ started at $z \in \Lambda$, $\tau_{\Lambda}$ is its first exit time of $\Lambda$ and $\tau^{\star}$ is its killing time.

Using the massive Green function $G_{\Lambda}^{m}$, one can express massive harmonic functions in $\Lambda$ in terms of their harmonic counterparts. More precisely, let $f: \partial \Lambda \to \mathbb{R}$ be a piecewise smooth function with finitely many discontinuity points. Let $h$ be the unique harmonic function in $\Lambda$ with boundary conditions $f$. Let $h^m$ be the unique massive harmonic function in $\Lambda$ with boundary conditions $f$, that is $h^m$ is the unique solution to the boundary value problem
\begin{align*}
    \begin{cases}
    [-\Delta + m^2(z)]u(z) = 0 \quad \text{in $\Lambda$} \\
    u=f \quad \text{on $\partial \Lambda$}.
    \end{cases}
\end{align*}
Then, it is easy to see that, for any $z \in \Lambda$,
\begin{equation} \label{harm_mass}
    h^m(z) = h(z) - \int_{\Lambda} m^{2}(w) h(w)G_{\Lambda}^m(z,w)dw
\end{equation}
Indeed, this follows from the facts that $[-\Delta + m^2(z)]h(z) = m^{2}(z)h(z)$ for $z \in \Lambda$, $h(z)=f(z)$ for $z \in \partial \Lambda$ and that, by definition of $G_\Lambda^m$, the function
\begin{equation*}
    \zeta^m: z \mapsto \int_{\Lambda} m^{2}(w) h(w)G_{\Lambda}^m(z,w)dw
\end{equation*}
is the unique solution to the boundary value problem
\begin{align*}
    \begin{cases}
    [-\Delta +m^2(z)]\zeta^m(z) = m^{2}(z)h(z) \quad \text{in $\Lambda$} \\
    \zeta^m=0 \quad \text{on $\partial \Lambda$}.
    \end{cases}
\end{align*}
Note that $h^m$ can also be rewritten in the form
\begin{equation} \label{harm_mass_alt}
    h^m(z) = h(z) - \int_{\Lambda} m^{2}(w) h^m(w)G_{\Lambda}(z,w)dw.
\end{equation}
Indeed, we have that, using the relation between $G_{\Lambda}$ and $G_{\Lambda}^m$ and Fubini's theorem ($\Lambda$ is bounded by assumption),
\begin{align*}
    \int_{\Lambda} m^2(w)G_{\Lambda}^m(z,w)h(w) dw &= \int_{\Lambda} m^2(w) \bigg[ G_{\Lambda}(z,w) - \int_{\Lambda} m^2(y)G_{\Lambda}(z,y)G_{\Lambda}^m(y,w) dy \bigg]h(w) dw\\
    &= \int_{\Lambda} m^2(w)G_{\Lambda}(z,w)h(w) dw \\
    &- \int_{\Lambda \times \Lambda} m^2(y)m^2(w)G_{\Lambda}(z,y)G_{\Lambda}^m(y,w)h(w) dwdy \\
    &= \int_{\Lambda} m^2(y)G_{\Lambda}(z,y)h(y) dy \\
    &- \int_{\Lambda} m^2(y)G_{\Lambda}(z,y)\int_{\Lambda}m^2(w)G_{\Lambda}^m(y,w) h(w) dw \\
    &= \int_{\Lambda} m^2(y)G_{\Lambda}(z,y) \bigg[ h(y) - \int_{\Lambda}m^2(w)G_{\Lambda}^m(y,w) h(w) dw \bigg] dy \\
    &= \int_{\Lambda} m^2(y)G_{\Lambda}(z,y) h^m(y)dy.
\end{align*}
We will also need a massive object related to the massive Poisson kernel in $\Lambda$. As we will use it only at one given point on the boundary, we find it convenient to introduce it as follows. Assume that $a$ and $b$ are two marked boundary points of $\partial \Lambda$. Let $\phi_{\Lambda}: \Lambda \to \mathbb{H}$ be a conformal map such that $\phi_{\Lambda}(a)=0$ and $\phi_{\Lambda}(b)=\infty$. For $z \in \Lambda$, set
\begin{equation} \label{P_lambda}
    P_{\Lambda}(z) := \frac{1}{\pi}\Im \bigg( \frac{-1}{\phi_{\Lambda}(z)}\bigg).
\end{equation}
Then $P_{\Lambda}(z)=P_{\HH}(\phi_\Lambda(z))$ is the bulk-to-boundary Poisson kernel in $\HH$ evaluated at the bulk point $\phi_{\Lambda}(z)$ and at the boundary point $0$, i.e. $P_{\Lambda}(z)$ is the density at $0$ of the harmonic measure of $\mathbb{R}$ seen from $\phi_{\Lambda}(z)$. Notice that $P_{\Lambda}(z)$ depends on the boundary points $a$ and $b$ but for conciseness, we do not mention explicitly this dependency in the notation. The massive version of $P_{\Lambda}$ is then defined by, for $z \in \Lambda$, 
\begin{equation} \label{Pm_lambda}
    P_{\Lambda}^m(z) := P_{\Lambda}(z) - \int_{\Lambda} m^2(w)P_{\Lambda}(w)G_{\Lambda}^m(z,w)dw .
\end{equation}
Finiteness of the above integral is shown in \cite[Equation~(4.6)]{mLERW}. Observe that by making the change of variable $u=\phi_{\Lambda}(w)$ in this integral and using the conformal covariance property of the massive Green function given by \eqref{cov_Green}, we have that
\begin{equation*}
    P_{\Lambda}^m(z) = P_{\Lambda}(z) - \int_{\HH} \tilde m^2(w)\frac{1}{\pi}\Im \bigg( \frac{-1}{w}\bigg) G_{\HH}^{\tilde m}(\phi_{\Lambda}(z),w)dw .
\end{equation*}
where $\tilde m^2(w) = \vert (\phi_{\Lambda}^{-1})'(w)) \vert^2 m^2(\phi_{\Lambda}^{-1}(w))$. We can thus see that $P_{\Lambda}^m(z)=P_{\HH}^{\tilde m}(\phi_\Lambda(z))$ is the massive bulk-to-boundary Poisson kernel in $\HH$ with mass $\tilde m$ evaluated at the bulk point $\phi_{\Lambda}(z)$ and at the boundary point $0$, i.e. $P_{\Lambda}^m(z)$ is the density at $0$ of the massive harmonic measure with mass $\tilde m$ of $\mathbb{R}$ seen from $\phi_{\Lambda}(z)$. Moreover, using conformal covariance of the Poisson kernel and that of its massive counterpart (for which the mass also changes under conformal maps), one can see that
\begin{equation} \label{ratio_P_Pm}
    \frac{P_{\Lambda}^m(z)}{P_{\Lambda}(z)} = \frac{\mathfrak{P}_{\Lambda}^m(z)}{\mathfrak{P}_{\Lambda}(z)}
\end{equation}
where $\mathfrak{P}_{\Lambda}^m(z)$, respectively $\mathfrak{P}_{\Lambda}(z)$, is the massive, respectively massless, bulk-to-boundary Poisson kernel in $\Lambda$ evaluated at the bulk point $z$ and at the boundary point $a$. In other words, $\mathfrak{P}_{\Lambda}^m(z)$, respectively $\mathfrak{P}_{\Lambda}(z)$, is the density at $a$ of the masseless, respectively massive, harmonic measure of $\partial \Lambda$ seen from $z$. Here, we consider ratios as $P_{\Lambda}(z)$ and $\mathfrak{P}_{\Lambda}(z)$ are related by the multiplicative factor $\vert \phi_{\Lambda}'(a) \vert$ and similarly for $P_{\Lambda}^m(z)$ and $\mathfrak{P}_{\Lambda}^m(z)$. This requires that the conformal map $\phi$ extends as a differentiable function at $a$, which is not necessarily the case. But the above ratios are nevertheless always well-defined.

\subsection{Martingale characterization of massive SLE$_4$} \label{sec_proof_charac}

Let us now state our characterization result. Although we have in mind its application to the characterization of the scaling limit of the massive harmonic explorer, this result holds under fairly general assumptions, that we now describe. Recall the assumptions made on the domain $\Omega$ and the boundary points $a, b \in \partial \Omega$ in Subsection \ref{subsec:domain}. In this setting, as in Subsection \ref{subsec:def_mHE}, we divide the boundary of $\Omega$ into two parts, $\partial \Omega^{+}$ and $\partial \Omega^{-}$, which are the clockwise, respectively counterclockwise, oriented boundary arcs between $a$ and $b$. Let $\phi: \Omega \to \HH$ be a conformal map such that $\phi(a)=0$ and $\phi(b) = \infty$. As before, we also let $m: \Omega \to \mathbb{R}_{+}$ be a continuous function bounded by some constant $\overline{m}>0$. Assume that $(\gamma(t), t \geq 0)$ is a random non-self-crossing curve in $\overline{\Omega}$ with $\gamma(0)=a$ and $\gamma(\infty)=b$. Let $(\phi(\gamma(t)), t \geq 0)$ be its image in $\mathbb{H}$. This is a non-self crossing curve in $\mathbb{H}$ starting at $0$ and targeting $\infty$. We assume that $(\phi(\gamma(t)), t \geq 0)$ is parametrized by half-plane capacity, or else reparametrize it. For $t \geq 0$, we denote by $K_t$ the hull generated by $\phi(\gamma)([0,t])$. $(K_t, t \geq 0)$ is a random locally growing family of hulls generated by a curve and therefore, as explained in Section \ref{sec_Loewner}, its growth can be described using the Loewner equation. In other words, from the family $(K_t, t \geq 0)$, we can construct a random Loewner chain $(g_t, t \geq 0)$ whose time-evolution is described by a random driving function $(W_t, t \geq 0)$ and the Loewner equation \eqref{eq_Loewner}. We set $\Omega_t: = \phi^{-1}(\mathbb{H} \setminus K_t)$ and denote by $\partial \Omega_{t}^{+}$, respectively $\partial \Omega_t^{-}$, the clockwise, respectively counter-clockwise, oriented boundary arc of $\Omega_t$ from $\gamma(t)$ to $b$. For $z \in \Omega$, we also define the (possibly infinite) stopping time
\begin{equation*}
    \tau_z := \inf\{ t \geq 0: \vert g_t(\phi(z)) - W_t \vert = 0\}.
\end{equation*}
$\tau_z$ corresponds to the time at which $\phi(z)$ is swallowed by the hulls $(K_t, t \geq 0)$ and with this definition, for $t \geq 0$, $K_t=\{ w \in \overline{\HH}: \tau_{\phi^{-1}(w)} \leq t\}$.

In what follows, we are going to consider the time-evolution of the massive Green function and of $P_{\Omega}^m$ under the Loewner maps $(f_t)_t$, where for $t \geq 0$, $f_t:=g_t-W_t$. In view of this, we introduce the following notations. We denote by $G_t^{m}$ the massive Green function with mass $m$ in $\Omega_t$, defined as in the discussion around \eqref{eq_G_Gm}. We also define, for $t \geq 0$ and $z \in \Omega_t$,
\begin{equation} \label{def_Pm_t}
    P_t^m(z) := \frac{1}{\pi} \Im \bigg( \frac{-1}{f_t(\phi(z))} \bigg) - \int_{\Omega_t} m^2(w) \frac{1}{\pi} \Im \bigg( \frac{-1}{f_t(\phi(w))} \bigg) G_t^m(z,w) dw.
\end{equation}
Remark that in the notations of Section \ref{sec_intro_Mharm}, $P_t^m(z) = P_{\Omega_t}^m(z)$ and, as already mentioned there (notice that $f_t \circ \phi$ satisfies the assumptions made on the map denoted $\phi_{\Lambda}$ in \eqref{Pm_lambda}), the integral on the right-hand side of the above equality is well-defined. Setting $P_t(z) := \frac{1}{\pi} \Im ( \frac{-1}{f_t(\phi(z))})$, the ratio $P_t^m(z)/P_t(z)$ can be given the same interpretation as in \eqref{ratio_P_Pm}, with the Poisson kernels being evaluated at the bulk point $z$ and at the boundary point $\gamma(t)$, the tip of the curve. Our characterization result then reads as follows.

\begin{proposition} \label{prop_characterization}
Suppose that $\Omega,a,b,(\gamma(t),t\ge 0)$ and $m$ are as described in the previous two paragraphs. For each $t \geq 0$, let $h_t^m: \Omega_t \to \mathbb{R}$ be the massive harmonic function in $\Omega_t$ with mass $m$ and boundary conditions $-1/2$ on $\partial \Omega_t^{-}$ and $1/2$ on $\partial \Omega_t^{+}$ and assume that $(h_t^{m}(z), t \leq \tau_z)$ is a martingale for all $z \in \Omega$. Let $h_t$ be the massless harmonic function in $\Omega_t$ with the same boundary conditions as $h_t^m$ and recall the definition of $P_t^m(z)$ given in \eqref{def_Pm_t}. Then $\gamma$ is distributed as a massive SLE$_4$ curve from $a$ to $b$ in $\Omega$, that is the driving function $(W_t, t \geq 0)$ of $\phi(\gamma)$ in $\HH$ is given by, for $t \geq 0$,
\begin{equation} \label{driving_mSLE}
    W_t = 2B_t - 2\pi \int_{0}^{t}\int_{\Omega_s} m^2(w) P_s^{m}(w)h_s(w)dw ds.
\end{equation}
\end{proposition}

In the course of the proof of Proposition \ref{prop_characterization}, we will repeatedly use the following massive version of the Hadamard's formula.

\begin{lemma} \label{lemma_massive_Hadamard}
Under the same assumptions on  $\Omega,a,b,(\gamma(t),t\ge 0)$ and $m$ as in Proposition \ref{prop_characterization}, for each $z, w \in \Omega$, the function $G_t^m(z,w)$ is differentiable in $t$, until the first time that either $z \notin \Omega_t$ or $w \notin \Omega_t$. Its differential is given by
\begin{equation*}
    \partial_t G_t^m(z,w) = -2\pi P_t^m(z)P_t^m(w),
\end{equation*}
where $P_t^m(z)$ and $P_t^m(w)$ are given by \eqref{def_Pm_t}.
\end{lemma}

\begin{proof}
When the mass $m$ is constant, the result is shown in \cite[Lemma~4.7]{mLERW}. The arguments can be straightforwardly adapted to the case of a bounded and continuous mass function.
\end{proof}

Before turning to the proof of Proposition \ref{prop_characterization}, we state a preliminary lemma which will allow us to control the ration $P_t^m(w)/P_t(w)$ uniformly in $t$ and $w$.

\begin{lemma} \label{lemma_controlPm}
Let $R>0, \overline{m}>0$ be such that $\Omega \subset B(0,R)$, and $m^2\le \overline{m}^2$. Then almost surely, for any $t \geq 0$ and any $w \in \Omega_t$,
\begin{equation*}
    \frac{P_{t}^m(w)}{P_{t}(w)} \geq \exp(-c_0\overline{m}^2R^2)
\end{equation*}
where $c_0 > 0$ is an absolute constant.
\end{lemma}

\begin{proof}
We first observe that, almost surely, for any $t \geq 0$ and any $w \in \Omega_t$,
\begin{equation*}
    \frac{P_{t}^m(w)}{P_{t}(w)} \geq \frac{P_{t}^{\overline{m}}(w)}{P_{t}(w)}.
\end{equation*}
This follows from the fact that $P_t^m(w)$ is a non-negative massive harmonic function in $\Omega_t$ with mass $m \leq \overline{m}$ while $P_t^{\overline{m}}$ is a non-negative massive harmonic function in $\Omega_t$ with $\overline{m}$ and both have the same boundary values (in the distributional sense). One can then use \cite[Equation~(4.10)]{mLERW} to obtain the lower bound
\begin{equation*}
    \frac{P_{t}^{\overline{m}}(w)}{P_{t}(w)} \geq \exp(-c_0\overline{m}^2R^2)
\end{equation*}
where $c_0 > 0$ is an absolute constant. In \cite{mLERW}, this inequality is first shown for the discrete counterpart of the ratio $P_{t}^{\overline{m}}(w)/P_{t}(w)$ on the square grid of meshsize $\delta >0$ and, since the lower bound is uniform in $\delta > 0$, the inequality in the continuum follows from convergence of the discrete ratio to the continuum one. This convergence holds provided that the discrete domains converge to $\Omega_t$ in the Carath\'eodory topology. We emphasize that the proof in \cite{mLERW} thus does not rely on the fact that the dynamics is that of the massive loop-erased random walk.
\end{proof}

With this lemma in hand, let us now turn to the proof of Proposition \ref{prop_characterization}.

\subsection{Proof of Proposition \ref{prop_characterization}} \label{subsec_Proof}

We prove Proposition \ref{prop_characterization} through a sequence of claims, that we now state and will prove in turn.

\begin{claim}\label{claim_semimartingale}
The driving function $(W_t, t \geq 0)$ of $\phi(\gamma)$ is a semi-martingale. It can therefore be decomposed as $W_t = M_t + V_t$ where $(M_t, t \geq 0)$ is a local martingale and $(V_t, t \geq 0)$ is a process with bounded variations.
\end{claim}

In view of Claim \ref{claim_semimartingale}, in order to prove Proposition \ref{prop_characterization}, we must identify the local martingale $(M_t, t\geq 0)$ and the process $(V_t, t\geq 0)$. To do this, we rely on the assumption that for each $z \in \Omega$, the process $(h_t^m(z), 0 \leq t \leq \tau_z)$ is a martingale. Indeed, by computing its Ito derivative and using its martingale property, we will obtain equations satisfied by the process $(V_t, t \geq 0)$ and the quadratic variation $(\langle M \rangle_t, t \geq 0)$ of $(M_t, t\geq 0)$ that will uniquely determine $(V_t, t \geq 0)$ and $(M_t, t \geq 0)$. Let us first compute the Ito derivative of $(h_t^m(z), 0 \leq t \leq \tau_z)$.

\begin{claim} \label{claim_SDE}
For each $z \in \Omega$, the process 
\begin{align*}
    Q_t^m(z) = \frac{1}{\pi} \Im \bigg( \frac{-1}{(g_t(\phi(z))-W_t)^2} \bigg) - \int_{\Omega_t} m^2(w) G_t^m(z,w) \frac{1}{\pi} \Im \bigg( \frac{-1}{(g_t(\phi(w))-W_t)^2} \bigg) dw, \quad t \leq \tau_{z},
\end{align*} is well defined, and the process $(h_t^m(z), t \leq \tau_z)$ satisfies the SDE
\begin{align} \label{SDE_h}
    dh_t^m(z) = &P_t^m(z)dM_t + P_t^m(z) dV_t + \frac{1}{2}Q_t^{m}(z)d\langle M \rangle_t -2Q^m_t(z)dt \nonumber \\
    &+ 2\pi P_t^m(z) \int_{\Omega_t}m^2(w)P_t^m(w)h_t(w) dw dt, \quad t \leq \tau_z.
\end{align}
\end{claim}

Since by assumption, for any $z \in \Omega$, $(h_t^m(z), 0 \leq t \leq \tau_z)$ is a martingale, we can deduce from Claim \ref{claim_SDE} that, almost surely, for any $t \geq 0$ and $z \in \Omega_t$,
\begin{equation} \label{eq_0mart}
    \int_{0}^{t} P_s^m(z)\bigg[ dV_s + 2\pi \int_{\Omega_s} m^2(w)P_s^m(w)h_s^m(w) dw ds \bigg] + \int_{0}^{t} Q_s^m(z) \bigg[\frac{1}{2}d\langle M \rangle_t -2dt \bigg] = 0.
\end{equation}
To identify $(V_t, t \geq 0)$ and $(\langle M \rangle_t, t \geq 0)$, we will use this equality evaluated at a well-chosen sequence of points and then take a limit. The next claim establishes the existence of this (subsequential) limit. 

\begin{claim}\label{claim_independence} Set
\begin{equation} \label{def_tildeV_A}
    \tilde V_t = V_t + 2\pi \int_{0}^{t} \int_{\Omega_s}m^2(w)P_s^m(w)h_s(w) dw ds \quad \text{and} 
    \quad A_t = \frac{1}{2} \langle M \rangle_t - 2t. 
\end{equation}
Fix $t > 0$ and consider the sequence of points $(z_n)_n = (\phi^{-1}(in))_n$. Then, almost surely, there exists a subsequence $(n(k))_k$ such that 
\begin{equation}\label{limPmQm}
     \mathcal{P}_{s}^{m}(b) := \lim_{k \to \infty} \frac{P_s^m(z_{n(k)})}{P_s(z_{n(k)})} \quad \text{ and } \quad \mathcal{Q}_{s}^{m}(b) := \lim_{n \to \infty} \frac{Q_s^m(z_{n( k)})}{P_s^m(z_{n(k)})} 
\end{equation}
exist, and moreover, 
\begin{align} \label{lim_SDE}
    \lim_{k \to \infty} \int_{0}^{t} n(k)P_{s}^{m}(z_{n( k)}) d\tilde V_{s} + \int_{0}^{t} n(k) Q_{s}^{m}(z_{n( k)}) d A_{s}
    = \int_{0}^{t} \mathcal{P}_{s}^{m}(b) d\tilde V_{s} + \int_{0}^{t} \mathcal{P}_{s}^{m}(b) \mathcal{Q}_{s}^{m}(b) d A_{s}.
\end{align}
\end{claim}

Finally, based on Claim \ref{claim_independence}, we will be able to identify $(M_t, t \geq 0)$ and $(V_t, t \geq 0)$.

\begin{claim}\label{claim_MV}
$(M_s)_{s\ge 0}$ has the law of $2$ times a standard Brownian motion, and with probability one,
\begin{equation*}
V_s = -2\pi \int_{0}^{s} \int_{\Omega_r} m^2(w)P_r^m(w)h_r(w)dwdr \quad 
\end{equation*}
for all $s\ge 0$.
\end{claim}

\begin{proof}[Proof of Proposition \ref{prop_characterization}]
This follows immediately from Claims \ref{claim_semimartingale} and \ref{claim_MV}.
\end{proof}

\begin{proof}[Proof of Claim \ref{claim_semimartingale}]
For $t \geq 0$, let us first relate the massive harmonic function $h_t^m$ to the massless harmonic function $h_t$. As explained in \eqref{harm_mass}, for $z \in \Omega_t$, we have
\begin{equation} \label{hm_with0}
    h_t^m(z) = h_t(z) - \int_{\Omega_t} m^2(w)G_t^m(z,w)h_t(w) dw 
\end{equation}
where $h_t$ can be written explicitly as
\begin{equation*}
    h_t(z) = \frac{1}{\pi}\text{arg}(g_t(\phi(z))-W_t) - \frac{1}{2}= \frac{1}{\pi}\text{arg}(f_t(\phi(z))) - \frac{1}{2}.
\end{equation*}
Using the representation \eqref{harm_mass_alt} of massive harmonic functions, we have the equality
\begin{equation*}
    h_t^m(z) = h_t(z) - \int_{\Omega_t} m^2(w)G_t(z,w)h_t^m(w) dw
\end{equation*}
from which we deduce that
\begin{equation*}
    h_t(z) = h_t^m(z) + \int_{\Omega_t} m^2(w)G_t(z,w)h_t^m(w) dw.
\end{equation*}
Moreover, since $G_t$ is almost surely equal to zero outside $\Omega_t\subset \Omega$, this yields that for $t \geq 0$ and $z \in \Omega_t$,
\begin{equation} \label{eq_ht_semimart}
     h_t(z) = h_t^m(z) + \int_{\Omega} m^2(w)G_t(z,w)h_t^m(w) dw.
\end{equation}
By the (non-massive) Hadamard formula, almost surely, for any $w \in \Omega_t$, $\partial_t G_t(z,w) = -2\pi P_t(z) P_t(w)$ and therefore, almost surely, for any $ w \in \Omega_t$, the function $s \mapsto G_s(z,w)$ is decreasing on $[0,t]$. Since $(h_s^m(w), 0 \leq s \leq t)$ is a martingale by assumption, we deduce from this that, for any $w \in \Omega_t$, $(G_s(z,w)h_s^m(w), 0 \leq s \leq t)$ is a semi-martingale. This implies that the process 
\begin{equation*}
    \bigg(\int_{\Omega} m^2(w)G_t(z,w)h_t^m(w) dw, \, 0 \leq t \leq \tau_z \bigg)
\end{equation*}
is a semi-martingale as well. Again, since $(h_t^m(z), 0 \leq t \leq \tau_z)$ is a martingale by assumption, the equality \eqref{eq_ht_semimart} then shows that for each $z \in \Omega_t$, $(h_t(z), 0 \leq t \leq \tau_z)$ is a semi-martingale. Now, writing $f_t(z)=X_t(z)+iY_t(z)$, we have
\begin{align*}
    & Y_t(z) - Y_0 = \int_{0}^{t} \Im \bigg( \frac{2}{f_s(z)} \bigg) ds \quad \text{and} \\
    & X_t(z) - X_0 = \frac{Y_t(z)}{\tan(\pi(h_t(z)+1/2))}.
\end{align*}
The process $(Y_t(z), t \geq 0)$ has bounded variations since $Y_t(z)=\Im(f_t(z))=\Im(g_t(z))$. As we have just established that $(h_t(z), 0 \leq t \leq \tau_z)$ is a semi-martingale, this in turn implies that the process $(X_t(z), 0 \leq t \leq \tau_z)$ is a semi-martingale. Writing
\begin{equation*}
    W_t = -X_t(z) + X_0 + \int_{0}^{t} \Re \bigg( \frac{2}{f_s(z)}\bigg) ds,
\end{equation*}
then shows that $(W_t, t \geq 0)$ is a semi-martingale.
\end{proof}

\begin{proof}[Proof of Claim \ref{claim_MV} given Claims \ref{claim_SDE} and \ref{claim_independence}]
Let us deduce from Claim \ref{claim_MV} and Claim \ref{claim_SDE} that the processes $(\tilde V_s, s \leq t)$ and $(A_s, s \leq t)$ defined in \eqref{def_tildeV_A} are both $0$. We will then explain why this yields Claim \ref{claim_MV}. From the equality \eqref{eq_0mart}, we obtain that almost surely, for any $k \in \mathbb{N}$,
\begin{equation*}
    \int_{0}^{t} n(k) P_s^m(z_{n(k)}) d \tilde V_s + \int_{0}^{t} n(k)Q_s^m(z_{n(k)}) dA_s = 0,
\end{equation*}
where $(n(k))_k$ is the subsequence obtained in Claim \ref{claim_independence}. Together with \eqref{lim_SDE}, this implies that, almost surely,
\begin{equation} \label{eq_Ps_b0}
    \int_{0}^{t} \mathcal{P}_s^{m}(b)\big[d\tilde V_s +\mathcal{Q}_{s}^{m}(b)dA_s \big] = 0.
\end{equation}
The process $(\mathcal{P}_s^{m}(b), s \leq t)$ is almost surely strictly positive on $[0,t]$ due to Lemma \ref{lemma_controlPm}. Therefore, we obtain from \eqref{eq_Ps_b0} that, almost surely, for all $r_1, r_2 \in [0,t]$ and any measurable function $f: [0,t] \to \mathbb{R}$,
\begin{equation} \label{eq_Vs_As}
    \int_{r_1}^{r_2} f(s) d \tilde V_s = - \int_{r_1}^{r_2} f(s) \mathcal{Q}_s^{m}(b) dA_s.
\end{equation}
This equality applied to the function $f(s) = P_s^m(z)$ for some $z \in \Omega_t$ together with \eqref{eq_0mart} yields that for any $z \in \Omega_t$, almost surely, for all $r_1, r_2 \in [0,t]$,
\begin{equation*}
    \int_{r_1}^{r_2} P_s^m(z) \mathcal{Q}_{s}^{m}(b)dA_s = \int_{r_1}^{r_2} Q_s^m(z) dA_s.
\end{equation*}
This implies that $A_s = 0$ for all $s \in [0,t]$. From the equality \eqref{eq_Vs_As}, we then conclude that $\tilde V_s = 0$ for all $s \in [0,t]$ as well. In view of the definitions of $(\tilde V_s, 0 \leq s \leq t)$ and $(A_s, 0 \leq s \leq t)$ given in \eqref{def_tildeV_A}, this yields that, almost surely, for $s \in [0,t]$,
\begin{align*}
    &V_s = -2\pi \int_{0}^{s} \int_{\Omega_r} m^2(w)P_r^m(w)h_r(w)dwdr \quad \text{and}\\
    &\langle M \rangle_s = 4s.
\end{align*}
Since $M_0=0$ and $(M_t, t \geq 0)$ is a continuous process, by L\'evy's characterization of Brownian motion, this implies that $M_t=2B_t$, where $(B_t, t \geq 0)$ is a standard one-dimensional Brownian motion. Therefore, we can conclude that
\begin{equation*}
    W_t = 2B_t - 2\pi \int_{0}^{t} \int_{\Omega_s} m^2(w) P_s^m(w)h_s(w)dwds
\end{equation*}
which is the statement of Claim \ref{claim_MV}.
\end{proof}

\begin{proof}[Proof of Claim \ref{claim_independence} given Claim \ref{claim_SDE}]
Let us first show that there almost surely exists a subsequence $(n(k))_k$ such that \eqref{limPmQm} holds. In order to do so, we are going to first prove that the sequence
\begin{equation*}
    \bigg( \frac{P_s^m(z_n)}{P_s(z_n)} \bigg)_{n \in \mathbb{N}}
\end{equation*}
is almost surely bounded, which implies that there almost surely exists a subsequence $(n(p))_p$ such that $((P_s^m(z_{n(p)})/P_s(z_{n(p)}))_p$ converges. We will then show that the subsequence
\begin{equation*}
     \bigg( \frac{Q_s^m(z_{n(p)})}{P_s^m(z_{n(p)})}\bigg)_{p \in \mathbb{N}}
\end{equation*}
is almost surely bounded. It will follow from this that there almost surely exists a subsequence $(n(k))_k$ of $(n(p))_p$ such that $((Q_s^m(z_{n(k)})/P_s^m(z_{n(k)}))_k$ converges, and thus that \eqref{limPmQm} holds.

The almost sure boundedness of the sequence $(P_s^m(z_n)/P_s(z_n))_n$ simply follows from the fact that, almost surely, for any $s \in [0,t]$ and any $n \in \mathbb{N}$,
\begin{equation*}
    \frac{P_s^m(z_n)}{P_s(z_n)} \leq 1.
\end{equation*}
We thus obtain the almost sure existence of a subsequence $(n(p))_p$ along which $(P_s^m(z_n)/P_s(z_n))_n$ almost surely converges. Let us denote by $\mathcal{P}_s^m(b)$ the limit as $p \to \infty$ of $(P_s^m(z_{n(p)})/P_s(z_{n(p)}))_p$. Note that due to Lemma \ref{lemma_controlPm}, almost surely, for any $s \in [0,t]$, $\mathcal{P}_s^m(b)$ is strictly positive. Let us now show that $((Q_s^m(z_{n(p)})/P_s^m(z_{n(p)}))_p$ is almost surely bounded. By Lemma \ref{lemma_controlPm} and \cite[Equation~(4.7)]{mLERW}, almost surely, for any $p$ and any $s \in [0,t]$,
\begin{align*}
    \frac{\vert Q_s^m(z_{n(p))}) \vert}{P_s^m(z_{n(p)})} &\leq \exp(c_0\overline{m}^2R^2) \frac{\vert Q_s^m(z_{n(p)}) \vert}{P_s(z_{n(p)})}\\
    &\leq \exp(c_0\overline{m}^2R^2) \times \bigg( \frac{\vert Q_s(z_{n(p)}) \vert}{P_s(z_{n(p)})} + \frac{1}{P_s(z_{n(p)})} \int_{\Omega_s} m^2(w) \vert Q_s(w) \vert G_s^m(z_{n(p)},w) dw \bigg)\\
    &\leq \exp(c_0\overline{m}^2R^2) \times \bigg(\frac{\vert Q_s(z_{n(p)}) \vert}{P_s(z_{n(p)})} + C\overline{m}^2\int_{\Omega_s}P_s(w) dw \\
    &+C\overline{m}^2\frac{\vert Q_s(z_{n(p)}) \vert}{P_s(z_{n(p)})} \int_{\Omega_s} G_s(z_{n(p)},w) dw
    + C\overline{m}^2\text{vol}(\Omega_s)\frac{\vert Q_s(z_{n(p)}) \vert}{P_s(z_{n(p)})} \bigg)
\end{align*}
where $C>0$ is an absolute (non-random) constant and almost surely, for any $s \in [0,t]$, $\int_{\Omega_s} P_s(w) dw$ is finite by \cite[Corollary~4.6(i)]{mLERW}. Observe that almost surely,
\begin{equation} \label{lim_QP}
    \lim_{n \to \infty} \frac{Q_s(z_n)}{P_s(z_n)} = 0
\end{equation}
and the convergence is almost surely uniform on the interval $[0,t]$. Let $\eps > 0$. It follows from \eqref{lim_QP} that there almost surely exists $K > 0$ such that for any $p \geq K$ and any $s \in [0,t]$, $\frac{\vert Q_s(z_{n(p)}) \vert}{P_s(z_{n(p)})} \leq \eps$. Moreover, almost surely, for any $s \in [0,t]$, $\text{vol}(\Omega_s) \leq \text{vol}(\Omega)$ and almost surely, for any $p$ and any $s \in [0,t]$,
\begin{equation*}
    \int_{\Omega_s} G_s(z_{n(p)}, w) dw \leq \int_{\Omega} G_0(z_{n(p)},w) dw \leq \tilde C \operatorname{diam}(\Omega)^2,
\end{equation*}
where $\tilde C > 0$ is a (non-random) constant. Furthermore, the function $s \mapsto \int_{\Omega_s} P_s(w) dw$ is almost surely continuous on $[0,t]$ and therefore has an almost sure maximum $M(t)$ on $[0,t]$. This maximum is almost surely non-negative since $P_s$ is almost surely non-negative. Hence, we have that, almost surely, for any $p \geq K$ and any $s \in [0,t]$,
\begin{equation} \label{eq_QmP_final}
    \frac{\vert Q_s^m(z_{n(p)}) \vert}{P_s^m(z_{n(p)})} \leq \exp(c_0\overline{m}^2R^2)(\eps + C\overline{m}^2M(t) + C \tilde C \overline{m}^2 \operatorname{diam}(\Omega)^2+C\overline{m}^2\eps \text{vol}(\Omega)).
\end{equation}
This shows that almost surely, for any $s \in [0,t]$, $( Q_s^m(z_{n(p)}) / P_s^m(z_{n(p)}))_p$ is a bounded sequence. Therefore, there almost surely exists a subsequence $(n(k))_{k}$ such that for any $s \in [0,t]$, the limit as $ k \to \infty$ of $ Q_s^m(z_{n(k)})/P_s^m(z_{n(k))})$ exists. For $s \in [0,t]$, we denote this limit by $\mathcal{Q}_s^m(b)$. We have thus establish \eqref{limPmQm}.

To show that the limit on the right-hand side of \eqref{lim_SDE} exists and is such that the equality \eqref{lim_SDE} holds, we use the dominated convergence theorem. We first establish that, almost surely,
\begin{equation*}
    \lim_{k \to \infty} \int_{0}^{t} n(k)P_s^m(z_{n(k)}) d\tilde V_s = \int_{0}^{t} \mathcal{P}_s^m(b) d\tilde V_s
\end{equation*}
where $\mathcal{P}_s^m(b)$ is given by \eqref{limPmQm} and $(\tilde V_s, s \geq 0)$ is as defined in \eqref{def_tildeV_A}. The process $(\tilde V_s, 0 \leq s \leq t)$ is a process of bounded variations. It can thus be decomposed as $\tilde V_s = \tilde V_s^{+} - \tilde V_s^{-}$ where $\mu^{+}([0,s)) = \tilde V_s^{+}$ and $\mu^{-}([0,s)) = \tilde V_s^{-}$ are non-negative measures. We first observe that, almost surely,
\begin{equation} \label{cvg_P0n}
    \lim_{n \to \infty} nP_s(z_n) = 1
\end{equation}
and the convergence is almost surely uniform on the interval $[0,t]$. Let $\epsilon > 0$. The previous observation implies that there almost surely exists $K_0 \in \mathbb{N}$ such that almost surely, for any $k \geq K_0$ and all $s \in [0,t]$, $n(k)P_s(z_{n(k)}) \leq 1 +\eps$. Therefore, almost surely, for any $k \geq K_0$ and any $s \in [0,t]$,
\begin{equation*}
    n(k)P_{s}^m(z_{n(k)}) = n(k)P_{s}(z_{n(k)})\frac{P_{s}^m(z_{n(k)})}{P_{s}(z_{n(k)})} \leq 1 +\eps
\end{equation*}
where in the last inequality, we also used the fact that, almost surely, for any $s \in [0,t]$ and $w \in \Omega_t$, $P_s^m(w) \leq P_s(w)$. The right-hand side of the above inequality is integrable with respect to $d\tilde V_s^{+}$ and $d\tilde V_s^{-}$. Therefore, by the dominated convergence theorem, almost surely,
\begin{align*}
    & \lim_{k \to \infty} \int_{0}^{t}  n(k)P_{s}^m(z_{n(k)}) d\tilde V_s^{+} = \int_{0}^{t} \mathcal{P}_s^{m}(b) d\tilde V_s^{+} \quad \text{and}\\
    & \lim_{k \to \infty} \int_{0}^{t}  n(k)P_{s}^m(z_n(k)) d\tilde V_s^{-} = \int_{0}^{t} \mathcal{P}_s^{m}(b) d\tilde V_s^{-},
\end{align*}
which shows that, almost surely,
\begin{equation*}
    \lim_{k \to \infty} \int_{0}^{t}  n(k)P_{s}^m(z_{n(k)}) d\tilde V_s = \int_{0}^{t} \mathcal{P}_s^{m}(b) d\tilde V_s.
\end{equation*}
Let us now show that, almost surely,
\begin{equation*}
    \lim_{k \to \infty} \int_{0}^{t}  n(k)Q_{s}^m(z_{n(k)}) d A_s = \int_{0}^{t} \mathcal{Q}_s^{m}(b) \mathcal{P}_s^m(b) d A_s
\end{equation*}
where $\mathcal{Q}_s^m(b)$ is given by \eqref{limPmQm} and $(A_s, s \geq 0)$ is as defined in \eqref{def_tildeV_A}. As before, we decompose the process $(A_s, 0 \leq s \leq t)$ as $A_s = A_s^{+} - A_s^{-}$ in order to apply the dominated convergence theorem with respect to $dA_s^{+}$ and $dA_s^{-}$. We first observe that, almost surely,
\begin{equation*}
    \lim_{n \to \infty} nQ_s(z_n) = 0
\end{equation*}
and the convergence is almost surely uniform on the interval $[0,t]$. Let $\eps > 0$. The previous observation, together with the uniform convergence \eqref{cvg_P0n}, implies that there almost surely exists $K_1 \in \mathbb{N}$ such that almost surely for any $k \geq K_1$ and any $s \in [0,t]$, $n(k)\vert Q_s(z_{n(k)}) \vert \leq \eps$ and $n(k)P_s(z_{n(k)}) \leq 1+\eps$. As above, using \cite[Equation~(4.7)]{mLERW}, we then obtain that, almost surely, for any $k \geq K_1$ and any $s \in [0,t]$,
\begin{align*}
    n(k)\vert Q_s^m(z_{n(k)}) \vert &\leq n(k) \vert Q_s(z_{n(k)}) \vert + n(k) \int_{\Omega_s} m^{2}(w) \vert Q_s(w) \vert G_s^m(z_{n(k)},w) dw \\
    &\leq \eps + n(k) C\overline{m}^2 \int_{\Omega_s} P_s(z_{n(k)}) P_s(w) dw + C \overline{m}^2 \int_{\Omega_s} n(k) \vert Q_s(z_{n(k)}) \vert G_s(z_{n(k)}, w) dw \\
    &+ C\overline{m}^2n(k)\vert Q_s(z_{n(k)}) \vert \text{vol}(\Omega_s) \\
    &\leq \eps + C\overline{m}^2(1+\eps) \int_{\Omega_s}P_s(w) dw +C\overline{m}^2\eps \int_{\Omega_s} G_s(z_{n(k)},w) dw + C\overline{m}^2 \eps \text{vol}(\Omega_s) \\
    &\leq \eps + C\overline{m}^2(1+\eps) M(t) +C \tilde C \overline{m}^2\eps \operatorname{diam}(\Omega)^2 + C\overline{m}^2 \eps \text{vol}(\Omega)
\end{align*}
where $C, \tilde C >0$ are (non-random) constants and $M(t)$ is defined as \eqref{eq_QmP_final}. The right-hand side of this inequality is integrable on $[0,t]$ with respect to $dA_s^{+}$ and $dA_s^{-}$. Therefore, by the dominated convergence theorem, almost surely,
\begin{align*}
    & \lim_{k \to \infty} \int_{0}^{t} n(k)Q_{s}^{m}(z_{n(k)}) dA_s^{+} =  \int_{0}^{t} \mathcal{Q}_s^{m}(b) \mathcal{P}_s^m(b)dA_s^{+} \quad \text{and} \\
    & \lim_{k \to \infty} \int_{0}^{t} n(k)Q_s^{m}(z_{n(k)}) dA_s^{-} =  \int_{0}^{t} \mathcal{Q}_s^{m}(b) \mathcal{P}_s^m(b)dA_s^{-}
\end{align*}
which shows that, almost surely,
\begin{equation*}
    \lim_{k \to \infty} \int_{0}^{t} n(k)Q_s^{m}(z_{n(k)}) dA_s =  \int_{0}^{t} \mathcal{P}_s^m(b) \mathcal{Q}_s^{m}(b) dA_s.
\end{equation*}
Above, we have also used the decomposition
\begin{equation*}
    n(k)Q_s^m(z_{n(k)}) = n(k)P_s(z_{n(k)}) \times \frac{P_s^m(z_{n(k)})}{P_s(z_{n(k)})} \times \frac{Q_s^m(z_{n(k)})}{P_s^m(z_{n(k)})}
\end{equation*}
and the fact that the three factors in this product almost surely converge to $1$, $\mathcal{P}_s^m(b)$ and $\mathcal{Q}_s^m(b)$, respectively, as $k \to \infty$.
\end{proof}

\begin{proof}[Proof of Claim \ref{claim_SDE}]
Recall the equality \eqref{hm_with0} relating $h_t^m(z)$ to $h_t(z)$ for $z \in \Omega$ and $t \leq \tau_z$. In view of this equality, a natural strategy to compute the Ito derivative of $h_t^m(z)$ would be to first apply Ito's lemma to the product $G_t^m(z,w)h_t(w)$ and then use the stochastic Fubini theorem to switch the stochastic integral and the integral over $\Omega$. However, as we have a priori no information of the local martingale $(M_t, t \geq 0)$, checking that the conditions of the stochastic Fubini theorem hold is not possible. We therefore follow a different strategy. Note that if absolute continuity of the limiting curve with respect to SLE$_4$ could be established from the discrete, then one may be able to use the stochastic Fubini theorem to compute the Ito derivative of $h_t^m$, as in the case of massive loop-erased random walk and massive SLE$_2$ \cite{mLERW}.

Let us first use the equality \eqref{hm_with0} to express $h_t^m$ in terms of the bulk-to-boundary Poisson kernel $P_{\HH}: \mathbb{R} \times \HH \to \mathbb{R}_{+}$ in $\HH$ given by $P_{\HH}(x,z) = (1/\pi)\Im(-1/(z-x))$. As explained around \eqref{P_lambda}, $P_{\HH}(x,z)$ is the density at $x$ of the harmonic measure of $\mathbb{R}$ seen from $z$ and, with the notation of \eqref{P_lambda}, $P_{\HH}(0,z) = P_{\HH}(z)$. Using the fact that
\begin{equation*}
    \frac{1}{\pi} \text{arg}(f_t(z)) = \int_{-\infty}^{0} P_{\HH}(x, f_t(\phi(z))) dx,
\end{equation*}
we have that, for $z \in \Omega$ and $t \leq \tau_z$,
\begin{equation*}
    h_t^m(z) = \int_{-\infty}^{0} P_{\mathbb{H}} (x, f_t(\phi(z))) dx -\frac{1}{2} - \int_{\Omega_t} m^2(w) G_t^m(z,w) \bigg[ \int_{-\infty}^{0} P_{\mathbb{H}} (x, f_t(\phi(w))) dx - \frac{1}{2} \bigg] dw.
\end{equation*}
By Fubini-Tonelli theorem, since the function $(w,x) \mapsto m(w)^2 G_t^m(z,w) P_{\HH}(x,f_t(\phi(w)))$ is non-negative, we can switch the integral over $\Omega_t$ and $(-\infty, 0)$. This yields that
\begin{equation*}
    \int_{\Omega_t} m^2(w) G_t^m(z,w)  \int_{-\infty}^{0} P_{\mathbb{H}} (x, f_t(\phi(w))) dx dw = \int_{-\infty}^{0} \int_{\Omega_t} m^2(w) G_t^m(z,w) P_{\HH}(x,f_t(\phi(w))) dw dx.
\end{equation*}
Notice that the integral $\int_{\Omega_t}m(w)^2 G_t^m(z,w) P_{\HH}(x,f_t(\phi(w))) dw$ is finite since the only divergence is at $z=w$ where the integrand is bounded from above by a multiple of the (massless) Green function $G_{0}(z,w)$. Therefore, we obtain that
\begin{align}
    h_t^m(z) = &\int_{-\infty}^{0} \bigg[ P_{\HH}(x,f_t(\phi(z))) - \int_{\Omega_t}m^2(w)G_t^m(z,w) P_{\HH}(x,f_t(\phi(w))) dw \bigg] dx \label{hm_P_var}\\
    & -\frac{1}{2} + \frac{1}{2} \int_{\Omega_t} m^2(w) G_t^m(z,w) dw.
\end{align}
We now make the change of variable $u = g_t(\phi(w))$ in the first integral over $\Omega_t$. By conformal covariance of the massive Green function stated in \eqref{cov_Green}, we have, for any $z, w \in \Omega_t$,
\begin{equation*}
    G_t^m(z,w) = G_{\HH}^{m_t}(g_t(\phi(z)), g_t(\phi(w)))
\end{equation*}
where $G_{\HH}^{m_t}$ is the massive Green function in $\HH$ with mass $m_t$ given by, for $u \in \HH$,
\begin{equation} \label{def_mt}
    m_t(u)^2 = \vert ((g_t \circ \phi)^{-1})'(u) \vert^2 m((g_t \circ \phi)^{-1}(u))^2.
\end{equation}
Going back to the equality \eqref{hm_P_var} for $h_t^m(z)$, we see that the changes of variables $u = g_t(\phi(w))$ and $v=x+W_t$ in the integral \eqref{hm_P_var} yield that
\begin{align} \label{eq_Pmt}
    h_t^m(z) = &\int_{-\infty}^{W_t} P_{\HH}^{m_t}(v, g_t(\phi(z))) dv \\
    & -\frac{1}{2} + \frac{1}{2} \int_{\Omega_t} m^2(w) G_t^m(z,w) dw.
\end{align}
where we have set for $z \in \Omega_t$ and $x \in \mathbb{R}$,
\begin{equation} \label{def_Pm_upper}
    P_{\HH}^{m_t}(x, g_t(\phi(z))) := P_{\HH}(x, g_t(\phi(z))) - \int_{\HH} m_t^2(w) G_{\HH}^{m_t}(g_t(\phi(z)),w) P_{\HH}(x,w) dw.
\end{equation}
Observe that by using the same changes of variable as above in the definition of $P_t^m$, we obtain that
\begin{equation*}
    P_t^m(z) = P_{\HH}(W_t,g_t(\phi(z))) - \int_{\HH} m_t^2(w) G_{\HH}^{m_t}(g_t(\phi(z)),w) P_{\HH}(W_t, w) dw.
\end{equation*}
Next, we want to compute the It\^o derivative of $h_t^m(z)$ using the expression \eqref{eq_Pmt} and It\^o's lemma. Let us first write the result and then explain how each term arises. The It\^o derivative of $h_t^m(z)$ reads (we will justify the appearance of each term below):
\begin{align}
    dh_t^m(z) = &P_t(z) dM_t - \int_{\Omega_t}m^2(w)P_t(w)G_t^m(z,w) dw dM_t \label{term_Mt}\\
    &+ P_t(z) dV_t - \int_{\Omega_t}m^2(w)P_t(w)G_t^m(z,w) dw dV_t \label{term_Vt} \\
    &+ \frac{1}{2}Q_t(z) d\langle M \rangle_s -  \frac{1}{2} \int_{\Omega_t}m^2(w)Q_t(w)G_t^m(z,w) dw d\langle M \rangle_s \label{term_bracket} \\
    & - 2Q_t(z) dt -2 \int_{\Omega_t}m^2(w)Q_t(w)G_t^m(z,w) dw dt \label{term_gt} \\
    &+ 2\pi P_t(z) \int_{\Omega_t}m^2(w)P_t^m(w)h_t(w)dwdt. \label{term_extra}
\end{align}
where we have set, for $z \in \Omega_t$,
\begin{equation*}
    Q_t(z) = \frac{1}{\pi} \Im \bigg( \frac{-1}{(g_t(\phi(z))-W_t)^2} \bigg).
\end{equation*}
Let us start by explaining where the term \eqref{term_Mt} comes from. Using \eqref{def_Pm_upper}, we can write
\begin{align}
    \int_{-\infty}^{W_t} P_{\HH}^{m_t}(x, g_t(\phi(z))) dx = &\int_{-\infty}^{W_t} P_{\HH}(x, g_t(\phi(z))) dx \label{hm_summand_1} \\
    & - \int_{-\infty}^{W_t} \int_{\HH} m_t^2(w) G_{\HH}^{m_t}(g_t(\phi(z)),w) P_{\HH}(x,w) dw dx. \label{hm_summand_2}
\end{align}
Observe that the function
\begin{equation*}
    y \in \mathbb{R} \mapsto \int_{-\infty}^{y} P_{\HH}^{m_t}(x, g_t(\phi(z))) dx
\end{equation*}
is differentiable and its derivative at $y \in \mathbb{R}$ is
\begin{equation*}
    P_{\HH}^{m_t}(y, g_t(\phi(z))).
\end{equation*}
To evaluate this derivative at $y=W_t$ in order to compute the term depending on $dM_t$ in the stochastic derivative of $h_t^m$, we must be slightly careful since both $W_t$ and the integrand in \eqref{hm_summand_1}--\eqref{hm_summand_2} depend on $t$. Quite straightforwardly, the term \eqref{hm_summand_1} gives rise to the term $P_t(z) dM_t$, while the term \eqref{hm_summand_2} gives rise to the other term depending on $dM_t$ in \eqref{term_Mt}. Indeed, observe that from the expression of $m_t$ given in \eqref{def_mt}, no terms depending on $dM_t$ arise from $m_t$. Similarly, the derivative of $G_{\HH}^{m_t}(g_t(\phi(z)),w)$ in \eqref{hm_summand_2} does not yield any term depending on $dM_t$. Indeed, by the massive Hadamard formula of Lemma \ref{lemma_massive_Hadamard}, we have that
\begin{equation*}
    \partial_t G_{\HH}^{m_t}(u,v) = \partial_t G_t^m((g_t \circ \phi)^{-1}(u), (g_t \circ \phi)^{-1}(v)) -2\pi P_t^m(((g_t \circ \phi)^{-1}(u)) P_t^m(((g_t \circ \phi)^{-1}(v)) dt
\end{equation*}
Therefore, going back to \eqref{hm_summand_1}--\eqref{hm_summand_2}, we can conclude that the term depending on $dM_t$ in the It\^o derivative of $h_t^m(z)$ is
\begin{equation*}
    P_{\HH}^{m_t}(W_t, g_t(\phi(z)))dM_t = P_{\HH}(W_t, g_t(\phi(z))) dM_t - \bigg[ \int_{\HH} m_t^2(w) G_{\HH}^{m_t}(g_t(\phi(z)),w) P_{\HH}(W_t,w) dw \bigg] dM_t
\end{equation*}
which can also be rewritten as
\begin{equation*}
    P_{t}(z) dM_t - \bigg[ \int_{\Omega_t} m^2(w) G_{t}^{m}(z,w) P_{t}(w) dw \bigg] dM_t = P_t^m(z) dM_t.
\end{equation*}
The term \eqref{term_Vt} in the stochastic derivative of $h_t^m$ arises for exactly the same reasons as the term \eqref{term_Mt}. Let us now explain where the quadratic variation term \eqref{term_bracket} comes from. To compute it, we see that we must take the derivative of the function
\begin{equation} \label{int_Pm_X}
    x \in \mathbb{R} \mapsto \Im \bigg( \frac{-1}{g_t(\phi(z)) - x} \bigg) - \int_{\Omega_t} m^2(w) G_t^m(z,w) \Im \bigg( \frac{-1}{g_t(\phi(w)) - x} \bigg) dw
\end{equation}
and evaluate it at $x=W_t$. The above function is indeed differentiable and its derivative at $x \in \mathbb{R}$ is
\begin{equation} \label{int_Qm_X}
    \Im \bigg( \frac{-1}{(g_t(\phi(z)) - x)^2} \bigg) - \int_{\Omega_t} m^2(w) G_t^m(z,w) \Im \bigg( \frac{-1}{(g_t(\phi(w)) - x)^2} \bigg) dw
\end{equation}
Indeed, the integral
\begin{equation*}
    \int_{\Omega_t} m^2(w) G_t^m(z,w) \bigg \vert \Im \bigg( \frac{-1}{(g_t(\phi(w)) - x)^2} \bigg) \bigg \vert dw
\end{equation*}
is finite since the only divergence is at $z=w$ where the integral is bounded from above by a multiple of the Green function $G_0(z,w)$. One can thus differentiate under the integral sign, which yields the expression \eqref{int_Qm_X} for the derivative of the function \eqref{int_Pm_X}. Once again, to evaluate the derivative of this function at $x=W_t$, we must be careful since both $W_t$ and the integrand depend on $t$. However, for the same reasons as above, no term depending on $d\langle M \rangle_t$ arise from the integrand and therefore, we obtain that the quadratic variation term in the It\^o derivative of $h_t^m$ is
\begin{equation*}
    \frac{1}{2}\bigg[ \Im \bigg( \frac{-1}{(g_t(\phi(z)) - W_t)^2} \bigg) - \int_{\Omega_t} m^2(w) G_t^m(z,w) \Im \bigg( \frac{-1}{(g_t(\phi(w)) - W_t)^2} \bigg) dw \bigg] d\langle M \rangle_t,
\end{equation*}
which, using the definition of $Q_t$, is exactly the term \eqref{term_bracket}. Let us now turn to the terms \eqref{term_gt} and \eqref{term_extra}. These terms come from the time-derivative of $g_t$, whose expression is given by the Loewner equation, and of $G_t^m$, which can be computed using the massive Hadamard formula of Lemma \ref{lemma_massive_Hadamard}. Writing
\begin{equation*}
    h_t^m(z) = \frac{1}{\pi}\text{arg}(g_t(\phi(z))-W_t) -\frac{1}{2} - \int_{\Omega_t} m^2(w)G_t(z,w) \bigg( \frac{1}{\pi}\text{arg}(g_t(\phi(w))-W_t) -\frac{1}{2} \bigg) dw,
\end{equation*}
we can see that the first summand in the term \eqref{term_gt} simply comes from the first term in the above expression for $h_t^m(z)$.  As for the second summand in \eqref{term_gt} and the term \eqref{term_extra}, we observe that by the Loewner equation and the massive Hadamard formula of Lemma \ref{lemma_massive_Hadamard}, we have that, for $z,w \in \Omega_t$,
\begin{equation*}
    \partial_t \big( G_t^{m}(z,w)h_t(w)\big) = -2\pi P_t^m(z)P_t^m(w)h_t(w) + G_t^m(z,w)Q_t(w).
\end{equation*}
Moreover, the integral $\int_{\Omega_t}P_t^m(w)h_t(w)dw$ is well-defined since $ \vert h_t(w) \vert$ is bounded from above by $\frac{1}{2}$ and the integral $\int_{\Omega_t}P_t^m(w)dw$ is finite by \cite[Corollary~4.6(i)]{mLERW}, see also \cite[Remark~4.3]{mLERW}. We have also seen above that the integral $\int_{\Omega_t} G_t^m(z,w)Q_t(w) dw$ is well-defined. Therefore, using the fact that $G_t^{m}(z,w)=0$ for $w \notin \Omega_t$, we obtain that
\begin{align*}
    \partial_t \bigg( \int_{\Omega_t} m^{2}(w) G_t^{m}(z,w)h_t(w) dw \bigg) &= \partial_t \bigg( \int_{\Omega} m^2(w) G_t^{m}(z,w)h_t(w) dw \bigg)\\
    &= -2\pi P_t^m(z) \int_{\Omega_t} m^2(w) P_t^m(w)h_t(w) dw  \\
    &+ \int_{\Omega_t} m^2(w)G_t^m(z,w)Q_t(w)dw,
\end{align*}
which exactly corresponds to the second summand of \eqref{term_gt} and the term \eqref{term_extra}. Moreover, inspecting the above arguments, one can see that no other terms arise in the Ito derivative of $h_t^m(z)$. Therefore, from the definition of $Q_t^m(z)$ in the statement of Claim \ref{claim_SDE}, we see that we have obtained the desired SDE for $h_t^m(z)$ and the proof of Claim \ref{claim_SDE} is complete.
\end{proof}

\subsection{Reformulation and proof of Theorem \ref{theorem_intro}} \label{subsec_ccl}

To conclude, let us now give a rigorous formulation of Theorem \ref{theorem_intro} and show how to combine the results of the previous sections to prove it.

\begin{theorem} \label{theorem_final}
Let $(\Omega_{\delta}, a_{\delta}, b_{\delta})_{\delta}$ be a sequence of subgraphs of $\delta \mathbb{T}$ with two marked boundary points $a_{\delta}$ and $b_{\delta}$. Assign a sign to boundary vertices of $\Omega_{\delta}$ in the fashion described in Section \ref{subsec:def_mHE}. Assume that $(\hat \Omega_{\delta}, a_{\delta}, b_{\delta})$ converges in the Carath\'eodory topology to $(\Omega, a, b)$, where $\Omega \subset \mathbb{C}$ and $a,b \in \partial \Omega$ satisfy the assumptions of Section \ref{subsec:domain}. Let $m: \Omega \to \mathbb{R}_{+}$ be a bounded and continuous function. For $\delta >0$, assign a mass to each edge of $\Omega_{\delta}$ in the manner described in Section \ref{subsec:def_mHE}. Let $(\gamma_{\delta})_{\delta}$ be a sequence of random paths distributed according to $(\PP_{\delta}^{(\Omega,a,b,m)})_{\delta}$.

Then, in the notations of Section \ref{sec_annulus_cross}, the sequence $(\gamma_{\delta}^{\HH})_{\delta}$ converges weakly to a random curve $\gamma^{\HH}$ in the topologies \eqref{topo_1} -- \eqref{topo_3}. The driving function of $\gamma^{\HH}$ parametrized by half-plane capacity satisfies the SDE
\begin{equation} \label{SDE_ccl}
    dW_t = 2dB_t - 2\pi \bigg( \int_{\Omega_t} m^2(w)P_t^m(w)h_t(w)dw \bigg) dt, \quad W_0=0,
\end{equation}
where $(B_t, t \geq 0)$ is a standard one-dimensional Brownian motion and where $P_t^m$ and $h_t$ have been defined in Section \ref{sec_proof_charac}. This SDE has a unique weak solution whose law is absolutely continuous with respect to $(2B_t, t \geq 0)$.

Moreover, provided that for each $\delta > 0$, $\gamma_{\delta}$ is parametrized by the half-plane capacity of $\gamma_{\delta}^{\HH}$, the sequence $(\gamma_{\delta})_{\delta}$ converges weakly in $X(\mathbb{C})$ equipped with the metric $d_X$ to a random curve $\gamma$ that is almost surely supported on $\overline{\Omega}$ and has the same law as $\phi_{\Omega}^{-1}(\gamma^{\HH})$. This implies in particular that $\gamma$ is absolutely continuous with respect to SLE$_4$ in $\Omega$ from $a$ to $b$.
\end{theorem}

\begin{proof}
By the results of Section \ref{sec_proof_crossing}, the sequence $(\gamma_{\delta}^{\HH})_{\delta}$ is tight in the topologies \eqref{topo_1} -- \eqref{topo_3}. Provided that for each $\delta > 0$, $\gamma_{\delta}$ is parametrized by the half-plane capacity of $\gamma_{\delta}^{\HH}$, this implies that $(\gamma_{\delta})_{\delta}$ is tight as well, in the space $X(\mathbb{C})$ equipped with the metric $d_X$. Let $(\gamma_{\delta_k})_k$ be a convergent subsequence and denote by $\gamma$ its weak limit. By Proposition \ref{prop_cvg_martingale}, for each $t \geq 0$, the corresponding subsequence of discrete massive harmonic functions $(h_{\delta_k, t(\delta_k)}^m)_k$ almost surely converges pointwise to $h_t^m$. Moreover, for each $\delta_k$ and any $v \in \operatorname{Int}(\Omega_{\delta_k}) \cup \partial \Omega_{\delta_k}$, $(h_{\delta_k, n}^m(v), n \geq 0)$ is a martingale for the filtration $(\mathcal{F}_{\delta_k,n})_n$. Therefore, we obtain that for any $z \in \Omega$, $(h_t^m(z), 0 \leq t \leq \tau_z)$ is a martingale for the filtration generated by $\gamma$. To conclude, we use the characterization result of Section \ref{sec_proof_charac}. Indeed, since for any subsequential limit $\gamma$ of $(\gamma_{\delta})_{\delta}$ and any $z \in \Omega$, $(h_t^m(z), 0 \leq t \leq \tau_z)$ is a martingale, Proposition \ref{prop_characterization} implies that the driving function of all subsequential limits of $(\gamma_{\delta}^{\HH})_{\delta}$ satisfies the SDE \eqref{SDE_ccl}. The fact that this SDE has a unique weak solution whose law is absolutely continuous with respect to $(2B_t, t \geq 0)$ is shown below in Lemma \ref{lemma_abscont_SLE}. The last part of the statement of Theorem \ref{theorem_final} is a consequence of \cite[Corollary~1.8]{Smirnov} and \cite[Theorem~4.2]{Karrila}. 
\end{proof}

\section{Massive Gaussian free field and massive SLE$_4$: level line coupling} \label{sec_coupling}

\subsection{Absolute continuity of massive SLE$_4$ with respect to SLE$_4$ and conformal covariance of massive SLE$_4$} \label{subsec_abs_cov}

Absolute continuity of the massive harmonic explorer with respect to the (non-massive) harmonic explorer is not straightforward to see at the discrete level, which explains why establishing tightness is more involved than in the case of massive loop-erased random walk \cite{mLERW}. However, in the continuum, absolute continuity of massive SLE$_4$ with respect to SLE$_4$ is easily shown, as pointed out by Makarov and Smirnov in \cite{off_SLE}. Here, we prove this fact for space-dependent mass, following the sketch of proof given in \cite[Section~3.2]{off_SLE} when the mass is constant. This implies in particular that massive SLE$_4$ shares many geometric properties of SLE$_4$.

\begin{lemma} \label{lemma_abscont_SLE}
    Let $\Omega \subset \mathbb{C}$ be a bounded, open and simply connected domain with two marked boundary points $a, b \in \partial \Omega$. Let $\alpha > 0$. With the same notations as in Section \ref{sec_proof_charac}, there exists a unique weak solution to the stochastic differential equation
    \begin{equation*}
        dW_t = 2dB_t - F_t^m dt, \quad W_0=0,
    \end{equation*}
    where, for $t \geq 0$,
    \begin{equation*}
        F_t^m := \alpha \int_{\Omega_t} m^2(w)P_t^m(w)h_t(w) dw.
    \end{equation*}
    This solution is absolutely continuous with respect to $(2B_t, t \geq 0)$. This implies that the massive SLE$_4$ Loewner chain $(f_t^m)_t$ with mass $m$ from $a$ to $b$ in $\Omega$ as defined in \eqref{driving_mSLE} is absolutely continuous with respect to the SLE$_4$ Loewner chain $(f_t)_t$ from $a$ to $b$ in $\Omega$. In particular, the hulls of massive SLE$_4$ are almost surely generated by a simple continuous curve $\gamma$ of Hausdorff dimension $3/2$. Moreover, $\gamma$ almost surely reaches its target point $b$ and does not intersect $\partial \Omega$, except at its endpoints.
\end{lemma}

\begin{proof}
We adapt the strategy outlined in \cite{off_SLE}, taking into account the fact that here the mass is position-dependent and therefore, a priori, $G^m(z,w) \neq G^m(w,z)$. To show existence, uniqueness and absolute continuity with respect to $(2B_t, t \geq 0)$ of the solution to the above SDE, we apply Novikov's criterion to the drift term $F_t^m$. Since the mass function $m$ and the harmonic function $h_t$ are both almost surely bounded for any $t \geq 0$, we have that, almost surely, for any $t \geq 0$,
\begin{equation*}
    \vert F_t^m \vert \leq \alpha M \int_{\Omega_t} P_t(z) dz,
\end{equation*}
for some (non-random) constant $M >0$ depending on $\overline{m}^2$. We recall that the integral on the right-hand side is almost surely finite by \cite[Corollary~4.6]{mLERW}. Using the Fubini-Tonelli theorem and the massive Hadamard formula of Lemma \ref{lemma_massive_Hadamard}, the above inequality yields that, almost surely,
\begin{align*}
    \int_{0}^{\infty} \vert F_t^m \vert^2 dt &\leq \alpha M^2 \int_{0}^{\infty} \int_{\Omega_t \times \Omega_t} P_t^m(z)P_t^m(w) dzdw dt \\
    &= \alpha M^2 \int_{\Omega \times \Omega} \int_{0}^{\infty} \mathbb{I}_{z \in \Omega_t} \mathbb{I}_{w \in \Omega_t} P_t^m(z)P_t^m(w) dzdw \\
    &= (-2\pi)\alpha M^2 \int_{\Omega \times \Omega} \int_{0}^{\infty} \mathbb{I}_{z \in \Omega_t} \mathbb{I}_{w \in \Omega_t} \partial_t G_t^m(z,w) dzdw dt \displaybreak[1] \\
    &= \tilde M \int_{\Omega \times \Omega} G_{\Omega}^{m}(z,w) dzdw - \int_{\Omega_{\infty} \times \Omega_{\infty}} G_{\infty}^{m}(z,w) dzdw \\
    &\leq \tilde M \int_{\Omega \times \Omega} G_{\Omega}^{m}(z,w) dzdw \\
    &\leq \tilde M \int_{\Omega \times \Omega} G_{\Omega}(z,w) dzdw.
\end{align*}
Since $\Omega$ is bounded, the last (non-random) integral on the right-hand side is finite. Therefore, we obtain that
\begin{equation*}
    \EE\bigg[\exp \bigg( \frac{1}{2} \int_{0}^{\infty} \vert F_t \vert^2 dt \bigg)\bigg] \leq \exp(C)
\end{equation*}
for some constant $C >0$ depending on $\alpha$, $m$ and $\Omega$. This shows that Novikov's criterion holds and therefore that there exists a unique weak solution to the SDE
\begin{equation*}
    dW_t = 2dB_t - F_t^m dt, \quad W_0=0,
\end{equation*}
which is absolutely continuous with respect to $(2B_t, t \geq 0)$. The rest of the statement of Lemma \ref{lemma_abscont_SLE} follows from the corresponding properties of SLE$_4$, see e.g. \cite[Chapter~5]{book_SLE}.
\end{proof}

The next lemma shows that massive SLE$_4$ is conformally covariant. Note that as a consequence of this result, one can extend the definition of massive SLE$_4$ and its absolute continuity with respect to SLE$_4$ to the case of unbounded domains provided that the mass is inherited from a bounded domain via conformal mapping.

\begin{lemma} \label{lemma_conformal_cov}
Let $\Omega \subset \mathbb{C}$ be a bounded, open and simply connected domain with two marked boundary points $a, b \in \partial \Omega$. Let $\varphi: \Omega \to \tilde \Omega$ be a conformal map such that $\varphi(a)=\tilde a$ and $\varphi(b)= \tilde b$, where $\tilde a, \tilde b \in \partial \tilde \Omega$. Let $m:\Omega \to \mathbb{R}_{+}$ be a bounded and continuous function. If $\gamma$ has the law of a massive SLE$_4$ curve in $\Omega$ from $a$ to $b$ with mass $m$, then $\varphi(\gamma)$ has the law of a massive SLE$_4$ in $\tilde \Omega$ from $\tilde a$ to $\tilde b$ with mass $\tilde m: \tilde \Omega \to \mathbb{R}_{+}$ given by, for $\tilde w \in \tilde \Omega$,
\begin{equation} \label{def_mass_cov}
    \tilde m^2(\tilde w) = \vert (\varphi^{-1})'(\tilde w) \vert^2 m^{2}(\varphi^{-1}(\tilde w)).
\end{equation}
In other words, massive SLE$_4$ is conformally covariant.
\end{lemma}

\begin{proof}
This simply follows from a change of variable in the drift term of massive SLE$_4$ in $\Omega$ from $a$ to $b$ with mass $m$. Recall that $\phi$ is a conformal map from $\Omega$ to $\HH$ such that $\phi(a)=0$ and $\phi(b)=\infty$. We have that
\begin{align*}
    P_t^m(z) = \frac{1}{\pi}\Im \bigg( \frac{-1}{f_t(\phi(z))}\bigg) - \int_{\Omega_t} m^2(w) G_t^m(z,w) \frac{1}{\pi}\Im \bigg( \frac{-1}{f_t(\phi(w))}\bigg) dw.
\end{align*}
Setting $z = \varphi^{-1}(\tilde z)$ and $w = \varphi^{-1}(\tilde w)$ and using the conformal covariance of the massive Green function (see \eqref{cov_Green}), we then obtain that
\begin{equation*}
    P_t^m(z) = \frac{1}{\pi}\Im \bigg( \frac{-1}{f_t((\phi \circ \varphi^{-1})(\tilde z))}\bigg) - \int_{\tilde \Omega_t} \tilde m^2(\tilde w) \tilde G_t^{\tilde m}(\tilde z, \tilde w) \frac{1}{\pi}\Im \bigg( \frac{-1}{f_t((\phi \circ \varphi^{-1})(\tilde w))}\bigg) d\tilde w
\end{equation*}
where $\tilde m$ is as in \eqref{def_mass_cov} and $\tilde G_t^{\tilde m}$ denotes the massive Green function with mass $\tilde m$ in $\tilde \Omega_{t} = \varphi(\Omega_t)$. In the notation of Section \ref{sec_intro_Mharm}, the right-hand side of the above equality is equal to $P_{\tilde \Omega_t}^{\tilde m}(\tilde z)$ since the map $\phi \circ \varphi^{-1}$ is a conformal map from $\tilde \Omega$ to $\HH$ such that $(\phi \circ \varphi^{-1})(\tilde a)=0$ and $(\phi \circ \varphi^{-1})(\tilde b)=\infty$. By conformal invariance of harmonic functions, this then yields that
\begin{equation} \label{drift_tilde}
    \int_{\Omega_t} m^2(w)P_t^m(w)h_t(w) dw = \int_{\tilde \Omega_t} \tilde m^2(\tilde w)P_{\tilde \Omega_t}^{\tilde m}(\tilde w)\tilde h_t(\tilde w) d\tilde w
\end{equation}
where $\tilde h_t$ is the harmonic function in $\tilde \Omega_t = \tilde \Omega \setminus \varphi(\gamma([0,t]))$ with boundary values $-1/2$ on the counter-clockwise oriented boundary arc $(\tilde a \tilde b)$ and the right side of $\varphi(\gamma([0,t]))$ and $+1/2$ on the clockwise oriented boundary arc $(\tilde a \tilde b)$ and the left side of $\varphi(\gamma([0,t]))$. The right-hand side of \eqref{drift_tilde} is exactly the drift term in the driving function of $\phi \circ \varphi^{-1}(\tilde \gamma)$ if $\tilde \gamma$ has the law of massive SLE$_4$ in $\tilde \Omega$ from $\tilde a$ to $\tilde b$ with mass $\tilde m$. This thus shows that $\varphi(\gamma)$ has indeed the law of massive SLE$_4$ in $\tilde \Omega$ from $\tilde a$ to $\tilde b$ with mass $\tilde m$.
\end{proof}

\subsection{Coupling of the massive Gaussian free field and massive SLE$_4$}

In this section, we show the existence of a coupling between a massive GFF and a massive SLE$_4$ curve stated in the introduction as Theorem \ref{coupling_intro}. This result is shown for space-dependent mass, so let us first define the massive GFF in this case. The definition is very similar to that in the constant mass case.

Let $\Omega \subset \mathbb{C}$ be an open, bounded and simply connected domain and let $a, b \in \partial \Omega$ be two boundary points. Fix a conformal map $\phi: \Omega \to \mathbb{H}$ such that $\phi(a)=0$ and $\phi(b)=\infty$. Let $m: \Omega \to \mathbb{R}_{+}$ be a continuous function bounded by some constant $\overline{m} > 0$ in $\Omega$. A massive GFF $\Gamma$ in $\Omega$ with mass $m$ and Dirichlet boundary conditions is a centered Gaussian process indexed by $\mathcal{C}_c^{\infty}(\Omega)$ with covariance given by, for $f,g \in \mathcal{C}_c^{\infty}(\Omega)$,
\begin{equation*}
    \EE[(\Gamma,f)(\Gamma,g)] = \int_{\Omega \times \Omega} f(z)G_{\Omega}^{m}(z,w)g(w) dz dw.
\end{equation*}
Above, as before, $G_{\Omega}^{m}$ is the massive Green function in $\Omega$ with mass $m$ or, in other words, this is the inverse in the sense of distributions of the operator $-\Delta + m^2$ with Dirichlet boundary conditions in $\Omega$. The massive GFF $\Gamma$ is absolutely continuous with respect to the massless GFF, with Radon-Nikodym derivative
\begin{equation*}
    \frac{1}{Z} \exp \bigg( -\frac{1}{2}\int_{\Omega} m^2(z) :\Gamma^{0}(z)^2: dz \bigg)
\end{equation*}
where $:\Gamma^{0}(z)^2:$ denotes the Wick-ordered square of the massless GFF $\Gamma^0$ and $Z$ is a normalization constant chosen so that the expectation of this random variable is one. Finiteness of the exponential term is established in \cite[Lemma~3.5]{mGFF_cont} when the mass $m$ is constant but the proof can easily be adapted to the case of a non-constant mass by viewing $m^2(z)dz$ as the volume form $\exp(\log(m^2(z)))dz$.

With these definitions in hand, we can now state the existence of a coupling between a massive GFF and a massive SLE$_4$ curve. In the statement below, we stress that the domain $\Omega$ satisfies the assumptions introduced at the beginning of this subsection, so that in particular $\Omega$ is bounded.

\begin{lemma} \label{lemma_coupling}
    Set $\lambda=\sqrt{\pi/8}$ and let $\gamma$ be a massive SLE$_4$ curve from $a$ to $b$ in $\Omega$ with drift term, for $t \geq 0$,
    \begin{equation*}
        F_t^m = 2\sqrt{2\pi} \int_{\Omega_t} m^2(w) P_t^m(w) \bigg(\frac{1}{\sqrt{2\pi}}\text{arg}(f_t(\phi(z)))-\lambda \bigg) dw.
    \end{equation*} 
    There exists a coupling $(\Gamma, \gamma)$ where $\Gamma$ is a massive GFF with mass $m$ in $\Omega$ such that the following domain Markov property is satisfied. Assume that $\tau$ is an almost surely finite stopping time for the filtration generated by $\gamma$ and define the following massive harmonic function 
    \begin{equation*}
        \eta_{\tau}^m : z \in \Omega_{\tau} \mapsto  \frac{1}{\sqrt{2\pi}} \text{arg}(f_{\tau}(\phi(z))) - \lambda - \int_{\Omega_{\tau}} m^2(w) G_{\tau}^m(z,w) \bigg(\frac{1}{\sqrt{2\pi}} \text{arg}(f_{\tau}(\phi(z))) - \lambda \bigg) dw.
    \end{equation*}
    In other words, $\eta_{\tau}^m$ is the unique massive harmonic function with mass $m$ in $\Omega \setminus \gamma([0,\tau])$ that has boundary conditions $-\lambda$ on the counter-clockwise oriented boundary arc $(ab)$ and the right side of $\gamma([0,\tau])$ and $\lambda$ on the clockwise oriented boundary arc $(ab)$ and the left side of $\gamma([0,\tau])$. Then, given $\gamma([0,\tau])$, the conditional law of $\Gamma+\eta_0^m$ restricted to $\Omega_{\tau}$ is that of the sum of a massive GFF in $\Omega_{\tau}$ with mass $m$ and Dirichlet boundary conditions plus the function $\eta_{\tau}^m$.
\end{lemma}

Let us point out that
 \begin{equation*}
        F_t^m = 2\pi \int_{\Omega_t} m^2(w) P_t^m(w) \bigg(\frac{1}{\pi}\text{arg}(f_t(\phi(z)))-\frac12\bigg) dw=2\pi \int_{\Omega_t} m^2(w) P_t^m(w) h_t(w) \, dw,
\end{equation*} 
that is, $F_t^m$ is the same drift term as the one appearing in Lemma \ref{lemma_abscont_SLE} with $\alpha=2\pi$. We have chosen to write $F_t^m$ in this form to emphasize that it can be expressed in terms of the harmonic function
\begin{equation*}
    z \in \Omega_t \mapsto \bigg(\frac{1}{\sqrt{2\pi}}\text{arg}(f_t(\phi(z)))-\lambda \bigg).
\end{equation*}
This function corresponds to the harmonic function appearing in the coupling between a (massless) GFF in $\Omega$ and a SLE$_4$ curve in $\Omega$ from $a$ to $b$. The function $\eta_t^m$ of Lemma \ref{lemma_coupling} is its massive version.

Moreover, we also note that by conformal covariance of massive SLE$_4$, see Lemma \ref{lemma_conformal_cov}, and of the massive GFF (which follows from the conformal covariance of the massive Green function \eqref{cov_Green}), Lemma \ref{lemma_coupling} can be extended to unbounded domains with appropriate mass functions. These mass functions are of the form \eqref{def_mass_cov}, that is are inherited from a bounded domain via conformal mapping.

\begin{proof}[Proof of Lemma \ref{lemma_coupling}]
    The proof goes along the same lines as in the massless case, see e.g. \cite[Proposition~2.2.7]{Wu_level}. Let $\eta$ be the continuous harmonic function in $\Omega$ with boundary conditions $-\lambda$ on the boundary arc $(ab)$ oriented clockwise and $\lambda$ on the boundary arc $(ba)$ oriented clockwise. More explicitly, for $z \in \Omega$,
    \begin{equation*}
        \eta(z) = \frac{1}{\sqrt{2\pi}} \text{arg}(\phi(z)) - \lambda.
    \end{equation*}
    For $t \geq 0$ and $z$ such that $\tau_z > t$, set
    \begin{equation*}
        \eta_t(z) = \frac{1}{\sqrt{2\pi}} \text{arg}(f_t(\phi(z))) - \lambda. 
    \end{equation*}
    Let $\eta^{m}$ be the massive harmonic function in $\Omega$ with the same boundary values as $\eta$, that is
    \begin{equation*}
        \eta^{m} = \eta(z) - \int_{\Omega} m^2(w) \eta(w) G_{\Omega}^{m}(z,w) dw.
    \end{equation*}
    For $t \geq 0$, let $\eta_t^{m}$ be the massive harmonic function in $\Omega_t$ with boundary values $-\lambda$ on the left side of $\gamma([0,t])$ and the clockwise-oriented boundary arc $(ab)$ and $\lambda$ on the right side of $\gamma([0,t])$ and the clockwise-oriented boundary arc $(ba)$. That is, for $z \in \Omega$ such that $\tau_z > t$,
    \begin{equation*}
        \eta_t^{m}(z) = \eta_t(z) - \int_{\Omega_t} m^2(w) \eta_t(w) G_{t}^{m}(z,w) dw. 
    \end{equation*}
    Fix $z \in \Omega$ and let us show that $t \mapsto \eta_t^{m}(z)$ is a continuous martingale until the possibly infinite stopping time $\tau_z$. Indeed, by the computations done in the proof of Claim \ref{claim_SDE}, $\eta_t^{m}(z)$ satisfies the SDE
    \begin{equation} \label{SDE_mharm_GFF}
        d\eta_t^{m}(z) = \sqrt{\frac{\pi}{2}} P_t^{m}(z) 2dB_t = \sqrt{2\pi} P_t^m(z)dB_t.
    \end{equation}
    Therefore, $t \mapsto \eta_t^{m}(z)$ is a local martingale. But since $t \mapsto \eta_t^{m}(z)$ is almost surely bounded uniformly over $t$ by $\lambda$, this is in fact a continuous martingale. 
    
    Next, let us show that for $z, w \in \Omega$, $t \mapsto \eta_t^{m}(z) \eta_t^{m}(w) + G_{t}^{m}(z,w)$ is a continuous martingale until the first time that either $\tau_z \leq t$ or $\tau_w \leq t$, This essentially follows from the massive Hadamard formula of Lemma \ref{lemma_massive_Hadamard} which, together with the SDE \eqref{SDE_mharm_GFF}, implies that
    \begin{equation*}
        d\langle \eta^{m}(z), \eta^{m}(w) \rangle_t = -\partial_t G_t^{m}(z,w).
    \end{equation*}
    Therefore, $t \mapsto \eta_t^{m}(z) \eta_t^{m}(w) + G_{t}^{m}(z,w)$ is a local martingale until the first time that either $\tau_z \leq t$ or $\tau_w \leq t$, Moreover, $\eta_t^{m}(w)$ and $\eta_t^{m}(w)$ are continuous and uniformly bounded over $z,w$ by $\lambda$ and $G_t^{m}$ is non-increasing in $t$. Thus, $t \mapsto \eta_t^{m}(z) \eta_t^{m}(w) + G_{t}^{m}(z,w)$ is a continuous martingale until the first time that either $\tau_z \leq t$ or $\tau_w \leq t$,

    Now, let $\varphi \in \mathcal{C}_c^{\infty}(\Omega)$ and define, for $t \geq 0$,
    \begin{equation*}
        E_t^m(\varphi) = \int_{\Omega_t \times \Omega_t} \varphi(z) G_t^{m}(z,w) \varphi(w) dz dw.
    \end{equation*}
    We want to show that $(\eta_t^{m},\varphi)$ is a continuous martingale with quadratic variation
    \begin{equation} \label{quadvar_int}
        d\langle (\eta^{m}, \varphi) \rangle_t = -dE_t^m(\varphi).
    \end{equation}
    Since $\eta_t^{m}(z)$ is a continuous martingale and is bounded uniformly over in $\Omega$ by $\lambda$, by Fubini's theorem, the integral $(\eta_t^{m}, \varphi)$ is also a bounded and continuous martingale. To show that its quadratic variation is given by \eqref{quadvar_int}, it suffices to show that $(\eta_t^{m}, \varphi)^2 + E_t^m(\varphi)$ is a martingale. We observe that
    \begin{equation*}
        (\eta_t^{m}, \varphi)^2 + E_t^m(\varphi) = \int_{\Omega_t \times \Omega_t} \varphi(z) \varphi(w) [\eta_t^{m}(z) \eta_t^{m}(w) + G_t^{m}(z,w)] dz dw.
    \end{equation*}
    We have already shown that $\eta_t^{m}(z) \eta_t^{m}(w) + G_t^{m}(z,w)$ is a continuous martingale and that $\eta_t^{m}(z)$, $\eta_t^{m}(w)$ are bounded uniformly over $z,w$. Moreover, $G_t^m(z,w)$ is non-negative and non-increasing in $t$. Therefore, we can apply Fubini's theorem, which yields that $(\eta_t^{m}, \varphi)^2 + E_t^m(\varphi)$ is a continuous martingale.

    It now remains to construct a coupling that satisfies the domain Markov property. For $z \in \Omega$, define
    \begin{equation*}
        \eta_{\infty}^{m} := \lim_{t \to \infty} \eta_{t}^{m}(z).
    \end{equation*}
    This limit exists almost surely for fixed $z$ since $\eta_{t}^{m}(z)$ is a bounded martingale. For $z, w \in \Omega$ and $\varphi \in \mathcal{C}_c^{\infty}(\Omega)$ non-negative, define also
    \begin{equation*}
        G_{\infty}^{m}(z,w) := \lim_{t \to \infty} G_t^{m}(z,w), \quad E_{\infty}^m(\varphi) := \lim_{t \to \infty} E_t^m(\varphi).
    \end{equation*}
    These limits exist because $G_t^{m}(z,w)$ and $E_t^m(\varphi)$ are both non-negative. Let $h$ be a massive GFF in $\Omega_{\infty}$ with mass $m$ and boundary conditions given by $\eta_{\infty}^m - \eta_0^m$. Then, for any $\varphi \in \mathcal{C}_c^{\infty}(\Omega)$ which is non-negative and any $\mu \geq 0$, we have
    \begin{align*}
        \EE[\exp(-\mu(\Gamma, \varphi))] &= \EE[\EE[\exp(-\mu(\Gamma, \varphi)) \vert K_{\infty}]] \\
        &= \EE \bigg[ \exp \bigg(- (\eta_{\infty}^m - \eta_0^m, \varphi) -\frac{\mu^2}{2} E_{\infty}^m(\varphi) \bigg)\bigg] \\
        &= \exp \bigg( -\frac{\mu^2}{2} E_0^m(\varphi)\bigg) \EE \bigg[ \exp \bigg(- (\eta_{\infty}^m - \eta_0^m, \varphi) -\frac{\mu^2}{2} (E_{\infty}^m(\varphi) - E_0^m(\varphi) \bigg)\bigg] \\
        &= \exp \bigg( -\frac{\mu^2}{2} E_0^m(\varphi)\bigg).
    \end{align*}
    The last equality holds because $(\eta_t^{m}, \varphi)$ is a continuous and bounded martingale with mean $(\eta_0^{m}, \varphi)$ and quadratic variation $E_0^m(\varphi) - E_t^m(\varphi)$. Finally, the coupling $(\Gamma, \gamma)$ satisfies the domain Markov property since for any function $\varphi \in \mathcal{C}_{c}^{\infty}(\Omega)$, the conditional law of $(\Gamma+{\eta_0^m}_{\vert \Omega_\tau}, \varphi)$ given $\gamma([0,\tau])$ is that of a Gaussian random variable with mean $(\eta_{\tau}^m, \varphi)$ and variance $E_{\tau}^m(\varphi)$.
\end{proof}

\begin{remark}
We observe that using exactly the same arguments as in the proof of Proposition \ref{lemma_coupling}, one can show that a massive GFF with appropriate boundary conditions can be coupled to a massive version of SLE$_4(\rho)$, where the drift is exactly the same as that of massive SLE$_4$ except that the harmonic function $h_t$ has different boundary conditions.
\end{remark}

\bibliography{main_arxiv}

\end{document}